\documentclass[a4paper,reqno,12pt]{amsart} % draft, oneside
\usepackage{amsmath}
\usepackage{graphicx}
\usepackage{enumitem}
\usepackage{mathrsfs}
\usepackage{amssymb,amsmath, amsfonts, amsthm}
\usepackage{color}
\usepackage[colorlinks]{hyperref}
\hypersetup{citecolor=red, linkcolor=blue}
\usepackage[margin=2.73cm]{geometry}
\usepackage{nccmath}

\usepackage{mathtools}

\newtheorem{theorem}{Theorem}[section]
\newtheorem{proposition}[theorem]{Proposition}

\newtheorem{lemma}[theorem]{Lemma}

\newtheorem*{theoremAone}{Theorem A}

\theoremstyle{definition}
\newtheorem{remark}[theorem]{Remark}

\newcommand{\ep}{\epsilon}

\newcommand{\R}{\mathbb{R}}

\newcommand{\N}{\mathbb{N}}

\newcommand{\mcp}{\mathcal{P}}

\newcommand{\wtb}{\widetilde{B}}

\newcommand{\Tau}{\boldsymbol{\tau}}

\newcommand{\dx}{\textup{d}x}
\newcommand{\dy}{\textup{d}y}
\newcommand{\dz}{\textup{d}z}

\numberwithin{equation}{section}
\allowdisplaybreaks

\makeatletter
\@namedef{subjclassname@2020}{%
\textup{2020} Mathematics Subject Classification}
\makeatother

\begin{document}
\title[Sign-changing solutions for critical Hamiltonian systems]%
{Infinitely many sign-changing solutions for critical Hamiltonian systems with linear perturbation}

\author{Yuxia Guo and Congzheng Xuanyuan}

\address[Yuxia Guo]{Department of Mathematical Science, Tsinghua University, Beijing 100084, P. R. China}
\email{yguo@tsinghua.edu.cn}

\address[Congzheng Xuanyuan]{Department of Mathematical Science, Tsinghua University, Beijing 100084, P. R. China}
\email{xycz23@mails.tsinghua.edu.cn}

\begin{abstract}
    In this paper, we study the following elliptic system
\begin{equation}\label{main_1}
\begin{cases}
-\Delta u = |v|^{p-1} v + \epsilon (\alpha u + \beta_1 v), & \text{in } \Omega, \\
-\Delta v = |u|^{q-1} u + \epsilon (\beta_2 u + \alpha v), & \text{in } \Omega, \\
u = v = 0, & \text{on } \partial \Omega,
\end{cases}
\tag{*}
\end{equation}
where \(\Omega\) is the unit ball in $\R^N$, \(\epsilon\) is a small parameter, \(\alpha\), \(\beta_1\) and \(\beta_2\) are real numbers, \((p, q)\) is a pair of positive numbers lying on the critical hyperbola
\begin{equation}
\frac{1}{p+1} + \frac{1}{q+1} = \frac{N-2}{N}.\nonumber
\end{equation}
Under suitable assumptions and suitable restrictions on $(p,q)$ and $N$, we construct infinitely many sign-changing solutions to \eqref{main_1} which look like a positive radial solution to \eqref{main_1} crowned by $k$ negative bubbles arranged on a regular polygon of a suitable radius, whose energy can be arbitrarily large.

\end{abstract}

\maketitle

{\textbf{Keywords:} Hamiltonian system, Brezis-Nirenberg problem, Infinitely many solutions, Lyapunov-Schmidt reduction.}

{\textbf {2020 MSC} 35A01, 35B33, 35J57.}

\section{Introduction}

We consider the following elliptic system
\begin{equation}\label{mainsystem}
\begin{cases}
-\Delta u = |v|^{p-1} v + \epsilon (\alpha u + \beta_1 v), & \text{in } \Omega, \\
-\Delta v = |u|^{q-1} u + \epsilon (\beta_2 u + \alpha v), & \text{in } \Omega, \\
u = v = 0, & \text{on } \partial \Omega,
\end{cases}
\end{equation}
where \(\Omega\) is a smooth bounded domain in \(\mathbb{R}^N\), \(N \geq 3\), \(\epsilon\) is a small parameter, \(\alpha\), \(\beta_1\) and \(\beta_2\) are real numbers, \((p, q)\) is a pair of positive numbers lying on the critical hyperbola
\begin{equation}\label{c-hyperbola}
\frac{1}{p+1} + \frac{1}{q+1} = \frac{N-2}{N}.
\end{equation}
Without loss of generality, we may assume that $p\leq \frac{N+2}{N-2}\leq q$.

If $u=v$ and $p=q=\frac{N+2}{N-2}$, system \eqref{mainsystem} reduces to the classical Brezis--Nirenberg problem \cite{BN}
\begin{equation}\label{BN-problem}
\begin{cases}
-\Delta u = |u|^{\frac{4}{N-2}}u+\lambda u, & \text{in } \Omega, \\
u = 0, & \text{on } \partial\Omega .
\end{cases}
\end{equation}
It is well-known that the classical Pohozaev’s identity \cite{Pohozaev} implies that if \(\lambda \leq 0\) and \(\Omega\) is star-shaped, then \eqref{BN-problem} has no solution. On the other hand, the existence of a positive solution was established in \cite{BN} provided
$\lambda\in(0,\lambda_1(\Omega))$ and $N\geq 4$. Here and after,
$\lambda_n(\Omega)$ denotes the $n$-th eigenvalue of $-\Delta$ with Dirichlet boundary condition. Devillanova and Solimini \cite{Solimini2002} (see also \cite{Fortunato1985,Solimini1986}) proved that if $N\geq 7$, then for any $\lambda>0$, problem \eqref{BN-problem} admits infinitely many solutions by using a compactness result which fails when $N\leq 6$. Moreover, Schechter and Zou showed in \cite{Zou} that \eqref{BN-problem} has infinitely many sign-changing solutions if $N\geq7$. When $\Omega$ is a ball, Atkinson, Brezis and Peletier \cite{ABP} proved the nonexistence of sign-changing radial solutions in lower dimensions $N\leq 6$, whereas Fortunato and Jannelli \cite{Fortunato1987} obtained infinitely many non-radial sign-changing solutions.

For system \eqref{mainsystem}, Mitidieri \cite{Mitidieri1993} and Van der Vorst \cite{Vorst1992} proved that if $\Omega$ is star-shaped and
$$\begin{pmatrix}
-\dfrac{\beta_2(q-1)}{2(q+1)} & -\dfrac{\alpha}{N} \\[0.6em]
-\dfrac{\alpha}{N} & -\dfrac{\beta_1(p-1)}{2(p+1)}
\end{pmatrix}$$
is positive semi-definite, then system \eqref{mainsystem} has no positive solution. In particular, this gives non-existence when $p,q>1$, $\alpha=0$, and $\beta_1,\beta_2\leq0$.
Conversely, using the dual variational formulation of
Clarke and Ekeland \cite{Ekeland1980}, Hulshof, Mitidieri and Van der Vorst \cite{Vorst1998} proved that if $p,q>1$, $\alpha\geq0$, and either $\beta_1>0$ or $\beta_1=0,\beta_2>0$, then system \eqref{mainsystem} admits
a solution for $N$ sufficiently large, provided that $\varepsilon^2\beta_1\beta_2\neq\lambda_n^2(\Omega)$ for all $n\in\mathbb N$. Moreover, the solution is positive when $\beta_1,\beta_2>0$ and $\varepsilon^2\beta_1\beta_2<\lambda_1^2(\Omega)$. More recently, Kim and Pistoia \cite{KP21} proved the existence of blowing-up solutions for system \eqref{mainsystem} with $\epsilon$ small.

 In this paper, we prove the existence of infinitely many non-radial sign-changing solutions to the Hamiltonian system whose energy can be arbitrarily large. Inspired by \cite{LBN}, we construct the solution which can be visualized as a superposition of a radial positive solution to \eqref{mainsystem} with a large number of negative bubbles of the form \eqref{lane-embden} arranged on a regular polygon.

 We point out that del Pino, Musso, Pacard and Pistoia in \cite{pino1, pino2} built such crown type sign-changing solutions to the Yamabe equation in $\R^N$. Subsequently, this idea was used in other critical elliptic problems; for instance, see \cite{GLPY} and \cite{LBN}.

 We also need the following property.
 \begin{align*}
     \textbf{(H)} \text{ there exists a positive radial solution } (u_\epsilon,v_\epsilon) \text{ to } \eqref{mainsystem} \text{ which is non-degenerate.}
 \end{align*}
where $(u_\ep,v_\ep)$ is non-degenerate in the sense that if $(\eta,\xi)\in \left(H_0^1(\Omega)\right)^2$ is a solution pair of the following linearized problem:
\begin{equation*}
    \begin{aligned}
        \begin{cases}
            -\Delta \eta =pv_\ep^{p-1}\xi+\ep \beta_1 \xi+\ep \alpha \eta,\qquad &\text{in }\Omega,\\
            -\Delta \xi =q u_\ep^{q-1}\eta+\ep \alpha \xi+\ep \beta_2 \eta,\qquad &\text{in }\Omega,\\
            \eta=\xi=0,\qquad &\text{on }\partial\Omega.
        \end{cases}
    \end{aligned}
\end{equation*}
 Then $(\eta,\xi)=(0,0)$.

 In a suitable range of $p$, $N$, $\alpha$, $\beta_1$ and $\beta_2$, Kim and Pistoia proved in \cite{KP21} the existence of single-bubble solutions to system \eqref{mainsystem}, and Guo, Hu, and Peng later established in \cite{GPH} the non-degeneracy of such solutions. Their results can be stated as follows.
\begin{theoremAone}[Theorem 1.1 in \cite{KP21}, Theorem 1.1 in \cite{GPH}]
 Assume that $N\ge 8$, $p\in (1,\dfrac{N-1}{N-2})$, and $(p,q)$ satisfies \eqref{c-hyperbola}, if one of the following conditions is satisfied:
\begin{equation}
\begin{aligned}
&\text{(i)}\ \beta_1>0, \qquad
\text{(ii)}\ \beta_1=0 \ \text{and}\ \alpha>0, \qquad
\text{(iii)}\ \beta_1=\alpha=0 \ \text{and}\ \beta_2>0,
\end{aligned}
\end{equation} then there exists a small number \(\epsilon_0 > 0\) depending only on \(N, p, \Omega, \alpha, \beta_1\) and \(\beta_2\) such that for any \(\epsilon \in (0, \epsilon_0)\), system (1.1) has a solution in \((C^2(\Omega))^2\) which blows up at $\xi_0$ ($\xi_0\in \Omega$ and $\nabla \tau(\xi_0)=0$, $\tau$ is defined in Lemma~\ref{lemma:H}) as $\ep\to 0$.

Moreover, if $\beta_1=\alpha=0$, $\beta_2>0$ and $\xi_0$ is the non-degenerate critical point of $\tau(x)$, then there exists a small number $\tilde{\epsilon}_0>0$ depending on $N,p,\Omega$ and $\beta_2$ such that for any $\epsilon\in (0,\tilde{\epsilon}_0)$, the above solution $(u_\epsilon,v_\epsilon)$ is non-degenerate.
\end{theoremAone}

In fact, when $\Omega$ is the unit ball, the above theorem yields a nondegenerate radial positive solution to \eqref{mainsystem}. More precisely, for $N\geq 8$ and $p\in (1,\dfrac{N-1}{N-2})$, assumption \textbf{(H)} follows from Theorem~A together with Lemma~\ref{lemma:H}. For other cases, we keep \textbf{(H)} as an assumption. The detailed argument is left to Lemma~\ref{lemma:H}.

 Let $\Omega$ be a unit ball. We impose the following condition:
 \begin{align*}
    \textbf{(P)}\qquad \begin{cases}
       p\in (\frac{11+\sqrt{57}}{16},\frac{7}{5}) ,\qquad&\text{if }N=7,\\
       p\in\left(1,\min\{\frac{N}{N-2},\frac{N+6}{2(N-3)}\}\right) ,\qquad&\text{if }8\leq N \leq 11.
    \end{cases}
 \end{align*}

    \begin{theorem}\label{Thm1}
    Assume that $7\leq N\leq 11$, $(p,q)$ satisfies \eqref{c-hyperbola} and \textnormal{(\textbf{P})}. If $\alpha=0,\ \beta_1=0,\ \beta_2>0$ and \textnormal{(\textbf{H})} holds, then problem~\eqref{mainsystem} admits infinitely many sign-changing solutions whose energy can be made arbitrarily large.
    \end{theorem}
        \begin{remark}
For $N\geq 8$, assume that $(p,q)$ satisfies \eqref{c-hyperbola} and
$p\in\left(1,\frac{N}{N-2}\right)$. Then
\[
p<\frac{N+6}{2(N-3)}
\quad\Longleftrightarrow\quad
q>2.
\]
Thus, for $8\leq N\leq 11$, condition \textnormal{(\textbf{P})} is equivalent to $p\in(1,\frac{N}{N-2})$, and $q>2$.
The condition $q>2$ is a technical restriction which comes from \eqref{restriction-1}. We also note that $\frac{N+6}{2(N-3)}> 1$ if and only if $N< 12$
which explains the restriction $N\leq 11$.

Moreover, we have
\[
\min\left\{\frac{N}{N-2},\frac{N+6}{2(N-3)}\right\}
=
\begin{cases}
\displaystyle \frac{N}{N-2}, & \text{if } N=7,8,\\[0.6em]
\displaystyle \frac{N+6}{2(N-3)}, & \text{if } N>8.
\end{cases}
\]
\end{remark}

Before outlining the main idea of the proof of Theorem \ref{Thm1}, we first introduce some notation and definitions. Let $N \geq 3$, $(p,q)$ satisfies \eqref{c-hyperbola} and let $(U,V)$ be a positive ground state solution of (see \cite{Lions1})
    \begin{equation}\label{lane-embden}
        \begin{cases}
            -\Delta U = |V|^{p-1}V, \quad \text{in } \R^N,\\
            -\Delta V = |U|^{q-1}U, \quad \text{in } \R^N,\\
            (U,V) \in \dot{W}^{2,\frac{p+1}{p}}(\R^N) \times \dot{W}^{2,\frac{q+1}{q}}(\R^N).
        \end{cases}
    \end{equation}
    It is known that $(U,V)$ is radially symmetric and decreasing after a suitable translation (see \cite{Lions2}). More precisely, the results of Wang \cite{Wang} and Hulshof and Van der Vorst \cite{HV} show that there exists a positive ground state solution $(U_{0,1}, V_{0,1})$ such that $U_{0,1}(0) = 1$, which is unique up to translation and scaling. The family $\{(U_{x,\mu}, V_{x,\mu})\}$ defined by
    \begin{equation}\label{bubble}
        (U_{x,\mu}, V_{x,\mu}) = \big(\mu^{\frac{N}{q+1}} U_{0,1}(\mu(y - x)), \ \mu^{\frac{N}{p+1}} V_{0,1}(\mu(y - x))\big), \quad \text{for any } x \in \R^N \text{ and } \mu > 0,
    \end{equation}
    exhausts all the positive ground state solutions of \eqref{lane-embden}.

Define the configuration space of parameters
\begin{equation}\label{mcp}
 \mcp :=\{(r,\lambda)\in(r_0-l_{00},r_0+l_{00}) \times (\lambda_{00}^{-1},\lambda_{00})\},
\end{equation}
where $r_0\in (0,1)$ is such that $r_0^{\frac{N}{q+1}}u_\ep(r_0)=\max\limits_{r\in[0,1]}r^{\frac{N}{q+1}}u_\ep(r)$, $l_{00}>0$ and $\lambda_{00}>0$ large enough.
For any $(r,\lambda) \in \mcp$ and $k \in \N$, we set
\[x_j := \left(r\cos\left(\tfrac{2(j-1)\pi}{k}\right), r\sin\left(\tfrac{2(j-1)\pi}{k}\right),{\bf{0}}\right) \in \R^2 \times \R^{N-2}, \quad j=1,2,\ldots,k,\]
and $\mu := \lambda k^{\frac{p+1}{p}}$ \big(i.e. $(k\mu^{-1})^{p(N-2)-2}\approx\mu^{-\frac{N}{q+1}}$, this is due to \eqref{bf_K}\big).

Given $h=1,\ldots,N$ and $j=1,\ldots,k$, let $\Phi_j$ be a rotation operator defined as
\begin{equation}\label{Aj}
\Phi_j(\rho\cos\phi, \rho\sin\phi, y'') := \left(\rho\cos\left(\phi+\tfrac{2(j-1)\pi}k\right), \rho\sin\left(\phi+\tfrac{2(j-1)\pi}k\right), y''\right)
\end{equation}
for $\rho > 0$, $\phi \in [0,2\pi)$, $y'' \in \R^{N-2}$, and $\Psi_h$ a reflection operator defined as
\begin{equation}\label{Bi}
\Psi_hy := (y_1,\ldots, y_{h-1}, -y_h, y_{h+1}, \ldots, y_N)
\end{equation}
for $y \in \R^N$.

We define the function spaces $L_s$ by
\begin{equation}\label{Ls}
 \begin{aligned}
       &L_s = \Big\{ (u,v)  \ : u, v\ \text{are measurable functions in} \,\,\Omega, \,(u,v)(\Psi_hy) = (u,v)(y)\\
       &(u,v)(\Phi_j(\rho\cos\phi, \rho\sin\phi, y'')) = (u,v)(\rho\cos\phi, \rho\sin\phi, y'') \text{ for } j = 1,\cdots,k;\, h=1,\cdots,N.
        \Big\}
 \end{aligned}
\end{equation}

For a fixed parameter $(r,\mu)$, we denote
$$(U_j,V_j)=(U_{x_j,\mu},V_{x_j,\mu})\qquad\text{for }j=1,\cdots,k.$$
Since $U_{x,\mu}(y)$ and $V_{x,\mu}(y)$ are not zero in $ B_1^c(0)$, we need to define $PU_{x,\mu}$ as the projection of $U_{x,\mu}$ and $PV_{x,\mu}$ as the projection of $V_{x,\mu}$, i.e., the solution of the following problem:
\begin{equation}\label{projection}
\begin{cases}
-\Delta PU_{x,\mu} = V_{x,\mu}^p, & \text{in } \Omega, \\
-\Delta PV_{x,\mu} = U_{x,\mu}^q, & \text{in } \Omega, \\
PU_{x,\mu} = PV_{x,\mu} = 0, & \text{on } \partial \Omega.
\end{cases}
\end{equation}
We also denote
$$(PU_j,PV_j)=(PU_{x_j,\mu},PV_{x_j,\mu})\qquad\text{for }j=1,\cdots,k.$$
Let 
\begin{equation}
      PV=\sum_{j=1}^k PV_j,
\end{equation}
and $PU$ be the unique solution of
\begin{equation}\label{eq:PU}
\begin{cases}
    -\Delta PU =PV ^p, & \text{in } \Omega,\\\
    PU=0, & \text{on } \partial \Omega.
\end{cases}
\end{equation}
%More details on $PU$ can be found in Appendix \ref{sec:est_PU}.

We will prove Theorem 1.1 by establishing the following result.
    \begin{theorem}\label{Thm2}
        Under the assumptions of Theorem \ref{Thm1}, there exists a positive integer $k_0 > 0$, such that for any integer $k \geq k_0$, \eqref{mainsystem} has a solution $(u_k,v_k)$ of the form
        \begin{align}\label{def:approx}
                        (u_k,v_k) \approx (PU_*,PV_*):=(u_\epsilon,v_\epsilon)-(PU,PV),
        \end{align}
        where $r_k\in (r_0-l_{00},r_0+l_{00})$, $\mu_k \in [\lambda^{-1}_{00} k^{\frac{p+1}{p}}, \lambda_{00} k^{\frac{p+1}{p}}]$. Moreover, as $k\to +\infty$, $r_k\to r_0$, $\mu_k\to+\infty$.
    \end{theorem}

    The proof of Theorem \ref{Thm2} is based on the finite-dimensional reduction method. We briefly outline it in the following. For a fixed $\ep$, problem \ref{mainsystem} is not a perturbation problem. Then the number $k$ of bubbles is used as the parameter to carry out the reduction procedure. This technique was first introduced by \cite{WY} and has been used successfully to study some non-compact elliptic problems. See \cite{WY1}--\cite{WY2}. Our analysis is carried out in a space with a weighted maximum norm which can capture the main pointwise estimate of the superposition of the bubbles when $k$ is large. Combining a priori estimate obtained by the blow up analysis and Fredholm alternative, we prove the invertibility of the linearized operator in a suitably chosen subspace of the above space. To derive a priori estimate for the linear theory, inspired by \cite{GKPY}, we first examine the existence and qualitative behavior of the Green’s function of a specific linear operator (the operator $L^*$ in \eqref{1-25-10n2}). Then the Lyapunov-Schmidt reduction allows us to reduce the problem to a 2-dimensional problem: finding critical points of the reduced energy $K_\ep(r,\lambda)$ in \eqref{eq:redene}.

    The main difficulties of this paper are as follows. 
    
    First, a crucial point is to construct a sufficiently accurate approximate solution. In the scalar case \cite{GLPY}, the ansatz is the sum of a positive radial solution and $k$ negative bubbles. Such an approximation does not work for the system when $p\in(1,\frac{N}{N-2})$. We must employ the approximate solution defined in \eqref{projection}--\eqref{def:approx}. This modification in the approximation of the u-part makes the estimates much more sophisticated. 
    
    Second, for the system when \(p\in\left(1,\frac{N}{N-2}\right)\), estimates for the approximate solution \(PU\), the energy expansion, and \(\|(l_1,l_2)\|_{**}\) in Section \ref{sec:L-S}, especially those terms involving \(PV\), differ from those in the case \(p\in\left(\frac{N}{N-2},\frac{N+2}{N-2}\right)\) studied in \cite{GKPY}, and therefore require new and more delicate analysis and computations. We believe that the ideas and the techniques we developed in this paper can be applied to solve other related problems in this area.

\vskip8pt
    Our paper is organized as follows. In Section 2, we derive the expansion of the energy functional near the approximate solutions. In Section 3, we prove the invertibility of the linearized operator in a suitable weighted functional space. In Section 4, we obtain the error estimate and employ the Lyapunov-Schmidt reduction to find the correction term $(\phi,\psi)$. In Section 5, we study the reduced finite-dimensional problem and prove Theorem \ref{Thm2}. In the Appendix, we collect all the  essential technical estimates needed throughout the paper. This includes the properties of the bubbles and the projection (Appendix A), sharp estimates of the approximate solution $PU$ (Appendix B), the Green's function (Appendix C),
    and various other algebra, integral and norm estimates (Appendix D).

    \medskip

    Throughout the paper, we use $C$ to denote a positive constant which may vary from line to line. In addition, we sometimes use $\int_A f$ to denote $\int_{A}f(x) \ \dx$ if there is no confusion.

\section{The energy expansion}

From now on, we assume that \(\alpha=\beta_1=0\) and \(\beta_2>0\). Set \(\frac{1}{p^*} = \frac{1}{q+1} + \frac{1}{N}\) and \(\frac{1}{q^*} = \frac{1}{p+1} + \frac{1}{N}\). We define the Banach space
\begin{equation*}
X_{p,q} := \left( W^{2,\frac{p+1}{p}} (\Omega) \cap W_0^{1,p^*}(\Omega) \right) \times \left( W^{2,\frac{q+1}{q}} (\Omega) \cap W_0^{1,q^*}(\Omega) \right)
\end{equation*}
equipped with the norm
\begin{equation*}
\|(u, v)\|_{X_{p,q}} := \|\Delta u\|_{L^{\frac{p+1}{p}}} + \|\Delta v\|_{L^{\frac{q+1}{q}}}.
\end{equation*}

Define an energy functional \( I_\ep : X_{p,q} \to \mathbb{R} \) by
\begin{equation}\label{def:I_ep}
\begin{aligned}
I_\epsilon(u, v) = \int_{\Omega} \nabla u \cdot \nabla v - \frac{1}{p+1} \int_{\Omega} |v|^{p+1} - \frac{1}{q+1} \int_{\Omega} |u|^{q+1} \\
-  \frac{\ep\beta_2}{2} \int_{\Omega} u^2 \quad \text{for } (u, v) \in X_{p,q}.
\end{aligned}
\end{equation}
Since \( q \geq p > 1 \), \( I_\epsilon \) is of class \( C^2(X_{p,q}) \). Also, \((u, v) \in X_{p,q} \) is a solution to \eqref{mainsystem} if and only if it is a critical point of \( I_\epsilon \).

In particular, we have
$$I_0(u, v) = \int_{\Omega} \nabla u \cdot \nabla v - \frac{1}{p+1} \int_{\Omega} |v|^{p+1} - \frac{1}{q+1} \int_{\Omega} |u|^{q+1}.$$

For later use, we define the following sets:
        \begin{equation}\label{Omega}
            \Omega_j = \left\{ y = (y',y'') \in \R^2 \times \R^{N-2}, \langle \dfrac{y'}{|y'|} , \dfrac{x_j'}{|x_j'|} \rangle \geq \cos \dfrac{\pi}{k}\right\}\cap \Omega, \;\; j=1,2,\cdots,k,
        \end{equation}
        and
\begin{equation}\label{eq:S}
    S := B_{\frac{\pi}{2}r_0^{-1}k^{-1}}(x_1) = \left\{y \in \Omega: |y-x_1| < \frac{\pi}{2}r_0^{-1}k^{-1}\right\}\subset \Omega_1,
\end{equation}
where $r_0>0$ is defined by \eqref{mcp}.

\begin{proposition}\label{prop:Iexpan}
Assume that $N \ge 7$ and $p \in (1,\frac{N}{N-2})$. For $k \in \N$ large enough, we have
\begin{equation}\label{energy_000}
    \begin{aligned}
        &I_\epsilon(PU_*, PV_*) \\
&=I_\epsilon(u_\epsilon,v_\epsilon)+ kA
+ k\left[-\frac{(B_1+B_2 \widetilde{H}(\tilde x_{1})) k^{p(N-2)-2}}{r^{p(N-2)-2} \mu^{p(N-2)-2}} + \frac{B_4 u_\epsilon(r)}{\mu^{\frac{N}{q+1}}}+O\left({\mu^{-\frac{N}{q+1}-\sigma}}\right)\right]\\
&\hspace{8em}-\chi_{\{2(p(N-2)-2)>N\}}\cdot k\left[\frac{\epsilon \beta_2 B_3}{\mu^{N-\frac{2N}{q+1}}}+O\left({\mu^{\frac{2N}{q+1}-N-\sigma}}\right)\right].
    \end{aligned}
\end{equation}
Here, $\widetilde{H}$ is defined in \eqref{H}, $A = I_0(U_{0,1}, V_{0,1})$, $B_1$, $B_2$, $B_3$ and $B_4$ are positive constants depending only on $N$ and $p$, and  $\sigma>0$ is a sufficiently small number.
\end{proposition}

To establish Proposition~\ref{prop:Iexpan}, we first prove the following lemma.

\begin{lemma}Assume that $N \ge 7$ and $p \in (1,\frac{N}{N-2})$. For $k \in \N$ large enough, we have
    \begin{equation}\label{energy_PUPV}
I_0(PU, PV) = k\Big[A-\frac{(B_1+B_2 \widetilde{H}(\tilde x_{1})) k^{p(N-2)-2}}{r^{p(N-2)-2} \mu^{p(N-2)-2}}
+ O\left({\mu^{-\frac{N}{q+1}-\sigma}}\right)\Big],
\end{equation}
where $A = I_0(U_{0,1}, V_{0,1})$, $B_1$ and $B_2$ are positive constants depending only on $N$ and $p$, and  $\sigma>0$ is a sufficiently small number.
\end{lemma}
\begin{proof}

By symmetry, it holds that
\begin{align*}
I_0(PU, PV) &= \frac1{p+1} \int_{\Omega} \nabla PU \cdot \nabla PV-\frac1{p+1} \int_{\Omega}PV^{p+1}\nonumber\\
&\quad + \left(1-\frac1{p+1}\right) \int_{\Omega} \nabla PU \cdot \nabla PV - \frac1{q+1}\int_{\Omega}PU^{q+1}\nonumber\\
&= \left(1-\frac1{p+1}\right) \int_{\Omega} PU \sum_{j=1}^k U_j^q -\frac1{q+1}\int_{\Omega}PU^{q+1}.
\end{align*}
From \eqref{pui}, we can compute
\begin{equation}\label{energy_01}
\begin{aligned}
\quad\int_{\Omega} PU \sum_{j=1}^k U_j^q
&= k \left[\int_{S} \bigg(\sum_{j=1}^k PU_j +\varphi\bigg)\sum_{j=1}^k U_j^q + \int_{\Omega_1 \setminus S} PU\sum_{j=1}^k U_j^q\right] \\
&= k \left[\int_{S} U_1^{q+1}+ \int_{S} U_1^q \bigg(\sum_{j=2}^k U_j+\varphi\bigg) + \int_{S} PU\sum_{j=2}^k U_j^q \right.\\
&\hspace{10em}\left.+\int_{\Omega_1 \setminus S} PU\sum_{j=1}^k U_j^q+O\left(\mu^{-\frac{N}{q+1}-\sigma}\right)\right],
\end{aligned}
\end{equation}
and
\begin{multline}\label{energy_02}
\int_{\Omega}PU^{q+1} = k \left[\int_{S} U_1^{q+1} + (q+1)\int_{S} U_1^q \bigg(\sum_{j=2}^k U_j+\varphi\bigg) \right.\\
\left. + O\left(\int_{S} PU_1^{q-\delta} \bigg(\sum_{j=2}^k PU_j + \varphi\bigg)^{1+\delta}\right) + \int_{\Omega_1 \setminus S} PU^{q+1}+O\left(\mu^{-\frac{N}{q+1}-\sigma}\right)\right].
\end{multline}

Hence, from \eqref{c-hyperbola}, we obtain
\begin{equation}\label{eq:IUV}
    \begin{aligned}
        I_0(PU, PV)&=k\left[\frac{2}{N} \int_{S} U_1^{q+1}-\frac{1}{p+1} \int_{S} U_1^q \bigg(\sum_{j=2}^k U_j+\varphi\bigg)\right.\\
        &\left.+(1-\frac{1}{p+1})\int_{ S} PU\sum_{j=2}^k U_j^q+ O\left(\int_{S} PU_1^{q-\delta} \bigg(\sum_{j=2}^k PU_j + \varphi\bigg)^{1+\delta}\right)\right.\\
        &\left.+(1-\frac{1}{p+1})\int_{\Omega_1 \setminus S} PU\sum_{j=1}^k U_j^q-\frac{1}{q+1}\int_{\Omega_1 \setminus S} PU^{q+1}+O\left(\mu^{-\frac{N}{q+1}-\sigma}\right)
        \right].
    \end{aligned}
\end{equation}

We will estimate each term on the right-hand side of \eqref{eq:IUV}.

First, using \eqref{bubble}, \eqref{U10est} and the fact that
$$(q+1)(p(N-2)-2)-N=(q+1)\frac{(p+1)N}{q+1}-N=pN>\frac{N}{q+1},$$
we can check that
\begin{equation}
    \begin{aligned}
        \int_{S} U_1^{q+1}
        &= \int_{\R^N}U_{0,1}^{q+1}+O\left( (\frac{k}{\mu})^{(q+1)(p(N-2)-2)-N}\right) \\
        &= \int_{\R^N}U_{0,1}^{q+1}+O\left( \mu^{-\frac{N}{q+1}-\sigma}\right).
    \end{aligned}
\end{equation}

Similar to the argument in Appendix A in \cite{MussoJMPA}, there exists a constant $\tilde{B}_{11}$ depending on $p$ and $N$, such that
\begin{align}\label{est:00}
    \sum_{j=2}^k \frac{1}{|x_j-x_1|^{p(N-2)-2}}= \wtb_{11}\cdot\frac{k^{p(N-2)-2}}{r^{p(N-2)-2}}+O(k^{p(N-2)-2-\sigma}).
\end{align}

Let $z=\mu(y-x_1)$, for any $y \in S$, we have
$$|z-\mu(x_j-x_1)|>C|\mu(x_j-x_1)|>>1$$
for some constant $C>0$ independent of $\mu$. Thus, it follows from \eqref{U10est}, \eqref{est:00} and Lemma \ref{A1} that
\begin{align*}
\int_{S} U_1^q \sum_{j=2}^k U_j &= \int_{\R^N} U_{0,1}^q \sum_{j=2}^k \frac{a_{N,p}}{(\mu|x_j-x_1|)^{p(N-2)-2}} + O\left(\mu^{-\frac{N}{q+1}-\sigma}\right) \\
&= \left(a_{N,p}\wtb_{11} \int_{\R^N} U_{0,1}^q\right) \frac{k^{p(N-2)-2}}{r^{p(N-2)-2}\mu^{p(N-2)-2}} + O\left(\mu^{-\frac{N}{q+1}-\sigma}\right).
\end{align*}
Note that
$$pq>1\Leftrightarrow q(p(N-2)-2)=\frac{qN(p+1)}{q+1}>N.$$
Hence,
$$\int_{\R^N} U_{0,1}^q<+\infty.$$

Moreover, Lemma~\ref{nnl2-19-1} implies that
\begin{align*}
\int_{S} U_1^q \varphi &= \int_{S} U_1^q(y) \left[\frac {k^{p(N-2)-2}}{r^{p(N-2)-2}\mu^{\frac {Np}{q+1}}}\widetilde{H}(r^{-1} k y)+O\Bigl(   \frac{k^{p(N-2)-2}}{\mu^{\frac{pN}{q+1}+\sigma}}\Bigr)\right]\dy \\
&=\frac {k^{p(N-2)-2}}{r^{p(N-2)-2}\mu^{p(N-2)-2}}\int_{B_{\frac{\pi\mu}{2r_0k}}(0)}U_{0,1}^q(z) \widetilde{H}(\frac{kz}{r\mu}+\tilde{x}_1)\dz+O\left(\mu^{-\frac{N}{q+1}-\sigma}\right)\\
&=\widetilde{H}(\tilde{x}_1)\int_{\R^N}U_{0,1}^q\cdot\frac {k^{p(N-2)-2}}{r^{p(N-2)-2}\mu^{p(N-2)-2}}+O\left(\mu^{-\frac{N}{q+1}-\sigma}\right).
\end{align*}
Thus, we have
\begin{equation}\label{energy_21}
    \int_{S} U_1^q \bigg(\sum_{j=2}^k U_j+\varphi\bigg)=\frac {k^{p(N-2)-2}}{r^{p(N-2)-2}\mu^{p(N-2)-2}}\left(a_{N,p}\wtb_{11}+\widetilde{H}(\tilde{x}_1)\right)\int_{\R^N}U_{0,1}^q+O\left(\mu^{-\frac{N}{q+1}-\sigma}\right) .
\end{equation}

Direct computation shows that
\begin{equation}\label{energy_31}
    \int_{S} PU_1^{q-\delta} \bigg(\sum_{j=2}^k PU_j + \varphi\bigg)^{1+\delta}\leq C (\frac{k^{p(N-2)-2}}{\mu^{p(N-2)-2}})^{1+\delta}\cdot \mu^{N}\int_{S} U_{0,1}(\mu(y-x_1))^{q-\delta}=O\left(\mu^{-\frac{N}{q+1}-\sigma}\right),
\end{equation}
where $\delta$ is chosen small enough such that $(q-\delta)(p(N-2)-2)>N$.

Also, we can compute
\begin{equation}\label{energy_32}
\int_{S} PU\sum_{j=2}^k U_j^q \le C\int_{S} U_1^{q-\delta} \cdot \sum_{j=2}^k U_j^{1+\delta} =O\left(\mu^{-\frac{N}{q+1}-\sigma}\right).
\end{equation}

On the other hand, it follows from Lemma \ref{lemma:U} that
\begin{equation}\label{eq:IUV41}
    \begin{aligned}
        \int_{\Omega_1 \setminus S} PU^{q+1}&\le C\mu^N\int_{\Omega_1 \setminus S} \big[\sum_{i=1}^k \frac{1}{(\mu|y-x_i|)^{p(N-2)-2}}\big]^{q+1}\dy\\
        &\hspace{4em}+\frac{C}{\mu^{pN}}\int_{\Omega_1 \setminus S}\big[ \sum_{j=1}^k \frac{k^{p(N-2)-2}}{(1+k|y-x_j|)^{p(N-3-\theta)-2}} \big]^{q+1}\dy.
    \end{aligned}
\end{equation}

To estimate \eqref{eq:IUV41}, we need some algebra inequalities as follows. We define
\begin{align*}
    \eta_1:&=(p(N-2)-3)(q+1),\\
    \eta_2:&=(p(N-3)-2)(q+1),\\
    \eta_3:&=(p(N-3)-3)(q+1).
\end{align*}
Then,
\begin{align}
    \eta_1>\eta_2>N,\qquad&\text{if}\quad N\geq 7,\, p\in(1,\,\frac{N}{N-2}),\label{alg:01}\\
    \eta_3>N,\qquad&\text{if}\quad N\geq 8,\, p\in(1,\,\frac{N}{N-2}).\label{alg:02}
\end{align}
However, when $N=7$, $\eta_3>N$ does not hold for all $p\in(1,\,\frac{N}{N-2})$.

Given any $\theta>0$ small enough, it holds that
\begin{align}\label{tip_1}
    \sum_{j=2}^k\frac{1}{|y-x_j|^{p(N-2)-2}}
    &\leq \frac{1}{|y-x_1|^{p(N-2)-3-\theta}}\sum_{j=2}^k\frac{1}{|y-x_j|^{1+\theta}} \nonumber\\
    &\leq \frac{C}{|y-x_1|^{p(N-2)-3-\theta}}\sum_{j=2}^k\frac{1}{|x_1-x_j|^{1+\theta}}\nonumber\\
    &\leq \frac{Ck^{1+\theta}}{|y-{x}_1|^{N-3-\theta}},\qquad y\in \Omega_1
\end{align}

Then, employing \eqref{alg:01} and \eqref{tip_1}, we compute
\begin{equation}\label{eq:IUV411}
    \begin{aligned}
        &\quad\mu^N\int_{\Omega_1 \setminus S} \big[\sum_{j=1}^k \frac{1}{(\mu|y-x_j|)^{p(N-2)-2}}\big]^{q+1}\dy\\
        &\le  C\int_{\Omega_1 \setminus S} \frac{1}{{\mu^{pN}}|y-x_1|^{(p+1)N}}+\frac{k^{(q+1)(1+\theta)}}{{\mu^{pN}}|y-x_1|^{(q+1)(p(N-2)-3-\theta)}}\dy \\
        &=O(\frac{k^{pN}}{\mu^{pN}})=O({\mu^{-\frac{N}{q+1}-\sigma}}).
    \end{aligned}
\end{equation}

If $\eta_3>N$, similar to \eqref{eq:IUV411}, the second integral on the RHS of \eqref{eq:IUV41} is bounded by
\begin{equation}
    \begin{aligned}
        &\quad  \frac{k^{(p+1)N}}{\mu^{pN}}\int_{\Omega_1 \setminus S}\frac{1}{(1+k|y-x_1|)^{(q+1)(p(N-3-\theta)-2)}}
        +\frac{1}{(1+k|y-x_1|)^{(q+1)(p(N-3-\theta)-3-\theta)}}\dy\\
        &\le \frac{k^{pN}}{\mu^{pN}} \int_{\R^N}\frac{1}{(1+|y|)^{(q+1)(p(N-3-\theta)-2)}}+\frac{1}{(1+|y|)^{(q+1)(p(N-3-\theta)-3-\theta)}}\dy\\
        &\le O(\frac{k^{pN}}{\mu^{pN}})=O({\mu^{-\frac{N}{q+1}-\sigma}}).
    \end{aligned}
\end{equation}

If $\eta_3\leq N$, the second integral on the RHS of \eqref{eq:IUV41} is bounded by
\begin{equation}
    \begin{aligned}
        & \quad \frac{1}{\mu^{pN}}\int_{\Omega_1 \setminus S}\frac{k^{(q+1)(p(N-2)-2)}}{(k|y-x_1|)^{(q+1)(p(N-3-\theta)-2)}}
        +\frac{k^{(q+1)(p(N-2)-2)}}{(k|y-x_1|)^{(q+1)(p(N-3-\theta)-3-\theta)}}\dy\\
        &\leq O(\frac{k^{pN}}{\mu^{pN}}+\frac{k^{(q+1)(p+1)(1+\theta)}}{\mu^{pN}})\leq O({\mu^{-\frac{N}{q+1}-\sigma}}).
    \end{aligned}
\end{equation}
The last inequality is due to
$$pN-\tau (p+1)(q+1)>\frac{N}{q+1},\quad\text{if  } N\geq 7,\, p\in(1,\,\frac{N}{N-2}).$$

Thus, we have
\begin{equation}\label{energy_41}
    \int_{\Omega_1 \setminus S} PU^{q+1}=O({\mu^{-\frac{N}{q+1}-\sigma}}),
\end{equation}
and
\begin{equation}\label{energy_42}
   \int_{\Omega_1 \setminus S} PU\sum_{j=1}^k U_j^q \leq C\int_{\Omega_1 \setminus S} PU^{q+1}=O({\mu^{-\frac{N}{q+1}-\sigma}}).
\end{equation}

Plugging \eqref{eq:IUV}, \eqref{energy_21}, \eqref{energy_31}, \eqref{energy_32}, \eqref{energy_41} and \eqref{energy_42} into \eqref{eq:IUV}, we establish \eqref{energy_PUPV}, that is
\begin{multline*}
        I_0(PU,PV)=k\left[\frac{2}{N}\int_{\R^N}U^{q+1}_{0,1}\right.\\
        \left.-\frac{1}{p+1}\frac {k^{p(N-2)-2}\left(a_{N,p}\wtb_{11}+\widetilde{H}(\tilde{x}_1)\right)\int_{\R^N}U_{0,1}^q}{r^{p(N-2)-2}\mu^{p(N-2)-2}}+O\left(\mu^{-\frac{N}{q+1}-\sigma}\right)\right].
\end{multline*}
\end{proof}

Now we are ready to prove Proposition~\ref{prop:Iexpan}.
\begin{proof}[Proof of Proposition \ref{prop:Iexpan}]

We have
\begin{align*}
&\quad \int_{\Omega} \nabla U_* \cdot \nabla V_*\\
&= \int_{\Omega} \nabla u_\epsilon \cdot \nabla v_\epsilon
+\int_{\Omega} \nabla PU \cdot \nabla PV
-\int_{\Omega} \nabla u_\epsilon \cdot \nabla PV
- \int_{\Omega} \nabla v_\epsilon \cdot \nabla PU\\
&= \int_{\Omega} \nabla u_\epsilon \cdot \nabla v_\epsilon
+\int_{\Omega} \nabla PU \cdot \nabla PV-\int_{\Omega} v_\epsilon^p PV
- \int_{\Omega} (u_\epsilon^q+\epsilon\beta_2 v_\epsilon)  PU.
\end{align*}
Hence, we can write
\begin{equation}\label{1-29-12}
\begin{aligned}
 I_{\epsilon}(PU_*, PV_*)
&= I_{\epsilon}(u_\epsilon,v_\epsilon) +I_0(PU, PV)-\frac{\epsilon\beta_2}{2}\int_{\Omega}PU^2
-\int_{\Omega}  v_\epsilon^pPV- \int_{\Omega} u_\epsilon^q  PU\\
&\qquad-\frac1{p+1}\int_{\Omega} \left(|v_\epsilon- PV|^{p+1}-|PV|^{p+1}-|v_\epsilon|^{p+1}\right)\\
&\qquad -\frac1{q+1}\int_{\Omega} \left(|u_\epsilon- PU|^{q+1}-|PU|^{q+1}-|u_\epsilon|^{q+1}\right)\\
&\ \\
&= I_\ep(u_\epsilon,v_\epsilon)+I_0(PU, PV) -\frac{1}{p+1} J_1- \frac{1}{q+1} J_2 -\frac{\epsilon\beta_2}{2} J_3,
\end{aligned}
\end{equation}
where
\begin{align}\label{J1J2}
  J_1 &:= k\int_{\Omega_1} \left(|v_\epsilon- PV|^{p+1}-|PV|^{p+1}-|v_\epsilon|^{p+1}+(p+1)v_\epsilon^pPV\right), \\
  J_2 &:= k\int\limits_{\Omega_1} \left(|u_\epsilon- PU|^{q+1}-|PU|^{q+1}-|u_\epsilon|^{q+1}+(q+1)u_\epsilon^qPU\right),\\
  J_3&:=\int_{\Omega}PU^2.
\end{align}

First, we estimate $J_1$.
\medskip

To this end, we divide $\Omega_1$ into two parts:
 $$\Omega_1^i:=\Omega_1\cap\{v_\epsilon< PV\} \hbox{ and } \Omega_1^o:=\Omega_1\cap\{v_\epsilon \geq PV\}.$$ Denote $S'=\Omega_1\cap B_{k^{\gamma-1}}(x_1)\supset S$, where $\gamma>0$ small enough.

It is easy to check that $S'\subset \Omega_1^i$.
Moreover, we have $V_1\geq C\sum\limits_{j=2}^k V_j\geq C v_\epsilon$ in $S'$.

We shall use the following inequalities
\begin{equation}\label{1-30-12}
    \begin{aligned}
        &\left\|v_\epsilon- PV|^{p+1}-PV^{p+1}-v_\epsilon^{p+1} +(p+1)v_\epsilon^p PV\right|
        \\
        &\hspace{10em}\leq
        \begin{cases}
            C\left(v_\epsilon^{p+1}+v_\epsilon^p PV+v_\epsilon PV^p \right)\leq
            Cv_\epsilon PV^p ,\qquad \text{if }v_\epsilon< PV,\\
            C\left( v_\epsilon^{p-1}PV^2+ PV^{p+1} \right),\hspace{8.5em}\text{if }v_\epsilon \geq PV.
        \end{cases}
    \end{aligned}
\end{equation}

On the one hand, we obtain
\begin{equation}\label{J1_1}
\begin{aligned}
        \int_{S'} v_\epsilon PV^p &\leq C\int_{S} V_1^p+ C\int_{S'\setminus S}\left(\sum\limits_{j=1}^k V_j \right)^{p}\\
    &\leq C \int_{B_{\frac{\mu}{k}}(0)} \frac{\mu^{\frac{pN}{p+1}-N}}{(1+|z|)^{p(N-2)}}+\frac{C}{\mu ^{\frac{pN}{q+1}}} \int_{S'\setminus S}\frac{1}{|y-x_1|^{p(N-2)}}+\frac{k^{p(1+\theta)}}{|y-x_1|^{p(N-3-\theta)}}\dy\\
    &\leq C(\frac{1}{\mu})^{\frac{p(N-2)}{p+1}}+\frac{C}{\mu ^{\frac{pN}{q+1}}}\left(k^{(N-p(N-2))(\gamma-1)}+k^{p(1+\theta)+(N-p(N-3-\theta))(\gamma-1)}\right)\\
    &=O({\mu^{-\frac{N}{q+1}-\sigma}}).
\end{aligned}
\end{equation}

On the other hand, since $(p+1)(N-3)>N$ and $p(N-2)-2>p+1$ holds for all $N\geq7$, $p\in(1, \frac{N}{N-2})$, we obtain
\begin{equation}\label{J1_2}
    \begin{aligned}
         \int_{\Omega_1 \setminus S'} PV^{p+1} &\leq C\int_{\Omega_1 \setminus S'} \left(\sum\limits_{j=1}^k V_j \right)^{p+1}\\
         &\leq C\int_{\Omega_1 \setminus S'} \left(\frac{1}{\mu ^{\frac{N}{q+1}}}\sum\limits_{j=1}^k \frac{1}{|y-x_j|^{N-2}} \right)^{p+1}\dy\\
         &\leq \frac{C}{\mu^{\frac{(p+1)N}{q+1}}}\int_{\Omega_1 \setminus S'}\frac{1}{|y-x_1|^{(p+1)(N-2)}}+ \frac{k^{(p+1)(1+\theta)}}{|y-x_1|^{(p+1)(N-3-\theta)}}\dy\\
         &= O\left( (\frac{k^{1-\gamma}}{\mu})^{p(N-2)-2}+k^{\gamma(p+1)(1+\theta)}(\frac{k^{1-\gamma}}{\mu})^{p(N-2)-2}\right)=O({\mu^{-\frac{N}{q+1}-\sigma}}).
         \end{aligned}
\end{equation}

Also, since
\[
 \frac{N}{q+1}-(N-4)\tau=\frac{2(p-1)}{p+1}>0
\]
holds for all $N\geq7$, $p\in(1, \frac{N}{N-2})$, we get
\begin{equation}\label{J1_3}
    \begin{aligned}
        \int_{\Omega_1 \setminus S'} v_\epsilon^{p-1}PV^{2} &\leq C\int_{\Omega_1 \setminus S'} \left(\sum\limits_{j=1}^k V_j \right)^{2}\\
        &\leq \frac{C}{\mu^{\frac{2N}{q+1}}}\int_{\Omega_1 \setminus S'}\frac{1}{|y-x_1|^{2N-4}}+ \frac{k^{2(1+\theta)}}{|y-x_1|^{2N-6-2\theta}}\dy\\
        &=O\left( \frac{k^{(1-\gamma)(N-4)}}{\mu^{\frac{2N}{q+1}}}
        +k^{2\gamma(1+\theta)}\frac{k^{(1-\gamma)(N-4)}}{\mu^{\frac{2N}{q+1}}}\right)
        =O({\mu^{-\frac{N}{q+1}-\sigma}}).
    \end{aligned}
\end{equation}

Combining \eqref{J1J2} and \eqref{1-30-12}--\eqref{J1_3}, we arrive at
\begin{equation}\label{J_1_1}
\begin{aligned}
J_1 &\leq C\int_{\Omega_1^i}v_\epsilon PV^p+C\int_{\Omega_1^o}v_\epsilon^{p-1} PV^2+PV^{p+1}\\
&\leq C\int_{S'}v_\epsilon PV^p+C\int_{\Omega_1^i\setminus S'} PV^{p+1}+C\int_{\Omega_1^o}(v_\epsilon^{p-1} PV^2+PV^{p+1})\\
&=O({\mu^{-\frac{N}{q+1}-\sigma}}).
\end{aligned}
\end{equation}

Next, we estimate $J_2$.
\medskip

Recall that $U_1\ge c>0$, $0< \varphi\le C$ and
$
0<\sum_{j=2}^kU_j\le C
$
in $S$. As in \eqref{1-30-12}, we have
\begin{equation}
    \begin{aligned}
        &|u_\epsilon- PU|^{q+1}-PU^{q+1}-u_\epsilon^{q+1} +(q+1)u_\epsilon^q PU
        \\
        &\hspace{4em}=
        \begin{cases}
            -(q+1)u_\epsilon PU^q+O\Big(u_\epsilon^{1+\delta}PU^{q-\delta}+u_\epsilon^{q+1}+u_\epsilon^q PU\Big)\\
            \hspace{4em}=-(q+1)u_\epsilon PU^q+O\Big(u_\epsilon^{1+\delta}U_1^{q-\delta}\Big) , &\text{if }x\in S,\\
            O\big( u_\epsilon^{q-1}PU^2+ PU^{q+1} \big),&\text{if }x\in\Omega_1\setminus S.
        \end{cases}
    \end{aligned}
\end{equation}

Denote $\bar{\eta}=(q-1)(p(N-2)-2)$. Then,
\begin{equation*}
    \begin{aligned}
        \int_{S} U_1^{q-1} \leq \frac{C}{\mu^{\frac{2N}{q+1}}}\int_{B_{\frac{\mu}{k}}(0)} U_{0,1}^{q-1}
        &=
        \begin{cases}
            O(\frac{1}{\mu^{\frac{2N}{q+1}}}) \qquad &\text{if}\quad \bar{\eta}> N\\
            O(\frac{1}{\mu^{\frac{2N}{q+1}}}\ln (\frac{\mu}{k}))\qquad &\text{if}\quad \bar{\eta}= N\\
            O(\frac{1}{\mu^{\frac{N}{q+1}}}\cdot (\frac{k}{\mu})^{q(p(N-2)-2)-N})\qquad &\text{if}\quad \bar{\eta}< N
        \end{cases}\\
        &=O({\mu^{-\frac{N}{q+1}-\sigma}}).
    \end{aligned}
\end{equation*}
Thus, we obtain
\begin{equation}
    \begin{aligned}
        \int_{S} PU^{q} u_\epsilon&=\int_{S} U_1^{q} u_\epsilon+O\left(\int_{S} U_1^{q-1} \right)\\
        &=\frac{u_\epsilon(r)}{\mu^{\frac{N}{q+1}}}\int_{\R^N}U_{0,1}^{q}+O(\mu^{-\frac{N}{q+1}-\sigma}).
    \end{aligned}
\end{equation}

In addition, as in \eqref{energy_31}, we compute
\begin{equation}
    \begin{aligned}
         \int_{S} u_\epsilon^{1+\delta} U_1^{q-\delta}\le \frac{C}{\mu^{\frac{N(1+\delta)}{q+1}}}\int_{S} U_{0,1}^{q-\delta}=O({\mu^{-\frac{N}{q+1}-\sigma}}).
    \end{aligned}
\end{equation}

On the other hand, \eqref{energy_41} and \eqref{J_3_3} give
\[
\int_{\Omega_1 \setminus S} PU^{q+1}+u_\epsilon^{q-1}PU^2 = O\left(\mu^{-\frac{N}{q+1}-\sigma}\right).
\]

As a consequence,
\begin{equation}\label{10-31-12}
J_2 = k\left[-\frac{u_\epsilon(r)}{\mu^{\frac{N}{q+1}}}\int_{\R^N}U_{0,1}^{q} + O\left(\mu^{-\frac{N}{q+1}-\sigma}\right)\right].
\end{equation}

Finally, we estimate $J_3$.
\medskip

In this part, we use the notation $S_j:=B_{\frac{\pi}{2}r_0^{-1}k^{-1}}(x_j)\cap \Omega_j$.
Then, it holds that $U_1\geq C\geq C\mu^{\frac{N}{q+1}}(k\mu^{-1})^{p(N-2)-2}\geq C\big(\sum\limits_{j=2}^k U_j+\varphi\big)$ in $S_1$.

Denote $\alpha_1:=2(p(N-2)-2)$ and $\alpha_2:=2(p(N-3)-2)$. Note that $\alpha_1\leq N$ happens only if $ N=7$. We can compute
\begin{equation}\label{J_31_1}
    \begin{aligned}
        \int_{S_1}PU^2&=\int_{S_1} \left(\sum_{j=1}^k PU_j+\varphi \right)^2\\
        &=\int_{S_1} PU_1^2+2(\sum_{j=2}^k PU_j+\varphi)PU_1+
        O\left((\sum_{j=2}^k PU_j)^2+\varphi^2\right)
    \end{aligned}
\end{equation}
We have
\begin{equation}\label{J_31_2}
    \begin{aligned}
       \int_{S_1} PU_1^2
=\begin{cases}
    \displaystyle\mu^{\frac{2N}{q+1}-N}
        \left(\int_{\R^N} U_{0,1}^2 +o(1)\right)\qquad &\text{if}\quad\alpha_1> N,\\
    \displaystyle O\left(\mu^{\frac{2N}{q+1}-N}\ln (\frac{\mu}{k}) \right)=O(\frac{1}{k^{N-\sigma}})\qquad &\text{if}\quad\alpha_1= N,\\
    \displaystyle O\left(\mu^{\frac{2N}{q+1}-N}(\frac{\mu}{k})^{N-\alpha_1}\right)=O(\frac{1}{k^N})\qquad &\text{if}\quad\alpha_1< N,
\end{cases}
    \end{aligned}
\end{equation}
Since $\sum\limits_{j=2}^k PU_j+\varphi\leq C$ in $S_1$, we have
\begin{align}\label{J_31_3}
    \left|\int_{S_1} (\sum_{j=2}^k PU_j+\varphi)PU_1\right|\leq C\int_{S_1}U_1=O(\mu^{\frac{N}{q+1}-N}(\frac{\mu}{k})^{N-[p(N-2)-2]})
  =O(\frac{1}{k^N}).
\end{align}
and
\begin{align}\label{J_3_4}
    \int_{S_1} O\left((\sum_{j=2}^k PU_j)^2+\varphi^2\right)=O(\text{Vol}(S_1))=O(\frac{1}{k^N}).
\end{align}

On the other hand, we compute that
\begin{equation*}
    \begin{aligned}
        \int_{\Omega\setminus \cup_{j=1}^kS_j} PU^2
    &\le C k \sum_{i=1}^k \int_{\Omega\setminus S_i} \left( \frac{\mu^{\frac{N}{q+1}}}{(\mu|y-x_i|)^{p(N-2)-2}}\right)^2+ \left( \frac{1}{\mu^{\frac{pN}{q+1}}} \frac{k^{p(N-2)-2}}{(k|y-x_i|)^{p(N-3-\theta)-2}}\right)^2\dy\\
    &\leq C k^2 \int_{\Omega\setminus S_1} \left( \frac{\mu^{\frac{N}{q+1}}}{(\mu|y-x_1|)^{p(N-2)-2}}\right)^2+ \left( \frac{1}{\mu^{\frac{pN}{q+1}}} \frac{k^{p(N-2)-2}}{(k|y-x_1|)^{p(N-3-\theta)-2}}\right)^2\dy.
    \end{aligned}
\end{equation*}
Also, we have
\begin{equation*}
    \begin{aligned}
        \int_{\Omega\setminus S_1}\left( \frac{\mu^{\frac{N}{q+1}}}{(\mu|y-x_1|)^{p(N-2)-2}}\right)^2
&= \frac{\mu^{\frac{2N}{q+1}}}{\mu^{2(p(N-2)-2)}}\int_{\Omega\setminus S_1} \left(  \frac{1}{|y-x_1|^{p(N-2)-2}}\right)^2\\
&=\begin{cases}
    \frac{\mu^{\frac{2N}{q+1}}}{\mu^{\alpha_1}}k^{\alpha_1-N}=O(\frac{1}{k^N})\qquad &\text{if}\quad\alpha_1> N,\\
    \frac{\mu^{\frac{2N}{q+1}}}{\mu^{\alpha_1}}\ln k=O(\frac{\ln k}{k^N})\qquad &\text{if}\quad\alpha_1= N,\\
    \frac{\mu^{\frac{2N}{q+1}}}{\mu^{\alpha_1}}=O(\frac{1}{k^{\alpha_1}})\qquad &\text{if}\quad\alpha_1< N,
\end{cases}
    \end{aligned}
\end{equation*}
and
\begin{equation*}
    \begin{aligned}
        \int_{\Omega\setminus S_1}\left( \frac{1}{\mu^{\frac{pN}{q+1}}} \frac{k^{p(N-2)-2}}{(k|y-x_1|)^{p(N-3-\theta)-2}}\right)^2
 &= \int_{\Omega\setminus S_1} \left(  \frac{1}{(k|y-x_1|)^{p(N-3-\theta)-2}}\right)^2\\
 &=\begin{cases}
        O(\frac{k^{\alpha_2-N}}{k^{\alpha_2}})=O(\frac{1}{k^N}) \qquad &\text{if}\quad\alpha_2>N,\nonumber\\
        O(\frac{\ln k}{k^{\alpha_2}})=O(\frac{\ln k}{k^N})\qquad &\text{if}\quad\alpha_2= N,\nonumber\\
        O(\frac{1}{k^{\alpha_2}}) \qquad &\text{if}\quad\alpha_2< N,\nonumber
    \end{cases}
    \\
    \end{aligned}
\end{equation*}

Since $(\min\{\alpha_2,N\}-1)\tau>\frac{N}{q+1}$ holds for all $N\geq 7$, $p\in(1,\,\frac{N}{N-2})$, we have
\begin{equation}\label{J_3_3}
    \int_{\Omega\setminus \cup_{j=1}^kS_j} PU^2=O\Big( k\mu^{-\frac{N}{q+1}-\sigma}\Big).
\end{equation}

Thus, combining \eqref{J_31_1}, \eqref{J_31_2}, \eqref{J_31_3} and \eqref{J_3_3}, we have
\begin{equation}\label{J_3_1}
    \begin{aligned}
        J_3=
        \begin{cases}
            k\Big(\mu^{\frac{2N}{q+1}-N}
        \int_{\R^N} U_{0,1}^2 +O(\mu^{-\frac{N}{q+1}-\sigma}+\mu^{\frac{2N}{q+1}-N-\sigma})\Big),&\text{if}\quad \alpha_1>N,\\
           O\Big( k\mu^{-\frac{N}{q+1}-\sigma}\Big),&\text{if}\quad \alpha_1\le N .
        \end{cases}
    \end{aligned}
\end{equation}

\medskip

Putting \eqref{1-29-12}, \eqref{J_1_1}, \eqref{10-31-12} and \eqref{J_3_1} together, we obtain \eqref{energy_000}. This completes the proof of Proposition~\ref{prop:Iexpan}.

\end{proof}

\section{The invertibility of the linear operator}
Fixing $\tau = \frac{p}{p+1} \in (0,1)$ so that $k \simeq \mu^{\tau}$, we define weighted $L^{\infty}(\R^N)$-norms:
\begin{equation}\label{eq:*-norm}
\begin{cases}
\displaystyle \|u\|_{*,1} := \sup_{y\in\Omega} \left[\sum_{j=1}^{k} \frac{\mu^{\frac{N}{q+1}}}{(1+\mu
|y-x_j|)^{\frac{N}{q+1}+\tau}}\right]^{-1}|u(y)|,\\
\displaystyle \|v\|_{*,2} := \sup_{y\in \Omega} \left[\sum_{j=1}^{k} \frac{\mu^{\frac{N}{p+1}}}{(1+\mu
|y-x_j|)^{\frac{N}{p+1}+\tau}}\right]^{-1}|v(y)|,
\end{cases}
\end{equation}
\begin{equation}\label{eq:**-norm}
\begin{cases}
\displaystyle \|f\|_{**,1} := \sup_{y\in \Omega} \left[\sum_{j=1}^{k} \frac{\mu^{\frac{N}{q+1}+2}}{(1+\mu
|y-x_j|)^{\frac{N}{q+1}+2+\tau}}\right]^{-1}|f(y)|,\\
\displaystyle \|g\|_{**,2} := \sup_{y\in \Omega} \left[\sum_{j=1}^{k} \frac{\mu^{\frac{N}{p+1}+2}}{(1+\mu
|y-x_j|)^{\frac{N}{p+1}+2+\tau}}\right]^{-1}|g(y)|,
\end{cases}
\end{equation}
and
\[\|(u,v)\|_* := \|u\|_{*,1}+\|v\|_{*,2} \quad \text{and} \quad \|(f,g)\|_{**} := \|f\|_{**,1}+\|g\|_{**,2}.\]
Then we set two Banach spaces
\begin{equation}\label{Xs}
    X_s:=\{(u,v)\in  L_s\cap [C(\Omega)\times C(\Omega)]: \|(u,v)\|_{*}<+\infty\},
\end{equation}
and
\begin{equation}\label{Ys}
    Y_s:=\{(f,g)\in  L_s\cap [C(\Omega)\times C(\Omega)]: \|(f,g)\|_{**}<+\infty\},
\end{equation}
Denote
\begin{equation}\label{YZ}
Y_{j,1} := \frac{\partial U_j}{\partial r},
\quad Y_{j,2} := \frac{\partial U_j}{\partial \mu};
Z_{j,1} := \frac{\partial V_j}{\partial r},
\quad Z_{j,2} := \frac{\partial V_j}{\partial \mu}
\end{equation}
for $j = 1,\ldots,k$.

Denote $\widetilde{L}_\ep(u,v):=(-\Delta u,-\Delta v-\ep\beta_2 u)$, from $Y_s$ to $X_s$. Then we can define $(u,v)=(\widetilde{L}_\ep)^{-1}(f,g)$ by Green representation as follows:
\begin{equation*}
\begin{aligned}
\begin{cases}
            u(x)=\int_{\Omega} G(x,y) f(y)dy,\\
    v(x)=\int_{\Omega} G(x,y) \big(g(y)+\ep\beta_2 u(y)\big)dy.
\end{cases}
\end{aligned}
\end{equation*}

Denote
\begin{equation}\label{barYZ}
(\bar{Y}_{j,1},\bar{Z}_{j,1}) := (\widetilde{L}_\ep)^{-1}\left (\frac{\partial (V_j^p)}{\partial r},\frac{\partial (U_j^q)}{\partial r} \right),
\quad (\bar{Y}_{j,2},\bar{Z}_{j,2}) := (\widetilde{L}_\ep)^{-1}\left (\frac{\partial (V_j^p)}{\partial \mu},\frac{\partial (U_j^q)}{\partial \mu} \right),
\end{equation}
for $j = 1,\ldots,k$.

Set
\begin{equation}\label{E}
    E=:\left\{ (u,v)\in X_s:  \left\langle \left(\sum\limits_{j=1}^{k} q U_j^{q-1} Y_{j,l},  \sum\limits_{j=1}^{k} p V_j^{p-1} Z_{j,l} \right), (u,v)  \right\rangle = 0, \;\; l=1,2.\right\}
\end{equation}
and
\begin{equation}\label{F}
    F=:\left\{(f,g)\in Y_s: \left\langle \left( \sum\limits_{j=1}^{k}\bar{Z}_{j,l}, \sum\limits_{j=1}^{k}\bar{Y}_{j,l}\right), (f,g)  \right\rangle = 0, \;\; l=1,2.\right\},
\end{equation}
where
$$\langle (\phi_1,\phi_2),(\psi_1,\psi_2) \rangle:= \langle \phi_1, \psi_1 \rangle + \langle \phi_2, \psi_2 \rangle: = \int_{\Omega} \phi_1\psi_1 + \phi_2 \psi_2.$$

Define
\begin{equation*}
    \begin{aligned}
        L_k (\phi, \psi) &:=\left( L_{1,k} (\phi, \psi), L_{2,k} (\phi, \psi)\right)\\
        &:=\left(-\Delta \phi - p |V_{*}|^{p-1} \psi, -\Delta \psi  -q |U_{*}|^{q-1} \phi -\epsilon\beta_2 \phi\right).
    \end{aligned}
\end{equation*}

Let us consider the following system
\begin{equation}\label{linear-equation}
 \begin{cases}
    \displaystyle L_k(\phi,\psi) = (f, g) + \sum_{l=1}^2 c_l
    \left(  p\sum_{j=1}^k V_j^{p-1}Z_{j,l},\,
                     q\sum_{j=1}^k U_j^{q-1}Y_{j,l}\right)
                                     \hspace{0.5em}\text{in}\hspace{0.5em}\Omega,\\
     \displaystyle (\phi,\psi)\in E.
 \end{cases}
\end{equation}
We have the following priori estimate result for $(\phi, \psi)$.
\begin{proposition}\label{blowup}
    Suppose $(\phi_k, \psi_k)$ solves \eqref{linear-equation} with $(f_k,g_k)$. If $\|(f_k,g_k)\|_{**} \rightarrow 0$, then $\|(\phi_k, \psi_k)\|_{*} \rightarrow 0$.
\end{proposition}

To establish Proposition \ref{blowup}, we derive first the following lemma.

\begin{lemma}\label{l1-18-2}
Assume that $N \ge 7$, $p\in (1,\frac{N}{N-2})$. The numbers $c_1$ and $c_2$ in \eqref{linear-equation} satisfy
\[
\mu|c_1| + \mu^{-1}|c_2| \le C\left( \frac1{\mu^\sigma} \|(\phi, \psi)\|_*+ \|(f, g)\|_{**} \right),
\]
where $\sigma>0$ is a fixed constant.
\end{lemma}

\begin{proof}
    In fact, the numbers $c_1$ and $c_2$ in \eqref{linear-equation} are determined by
    \begin{equation}\label{2-18-2}
\begin{aligned}
&\quad \sum_{l=1}^2 c_l \int_{\Omega} \left(p\sum_{j=1}^k V_j^{p-1}Z_{j,l}Z_{1,h} + q\sum_{j=1}^k U_j^{q-1}Y_{j,l}Y_{1,h}\right)\\
&= \int_{\Omega} \left(Z_{1,h} L_{1,k}(\phi,\psi) +Y_{1,h} L_{2,k}(\phi,\psi)\right) - \int_{\Omega} \left(Z_{1,h} f +Y_{1,h} g\right), \quad h= 1,2.
\end{aligned}
\end{equation}

We will estimate each term in \eqref{2-18-2}. Denote $n_h=1$ if $h=1$; $n_h=-1$ if $h=2$.

First, using Lemma \ref{B3}, we have
\begin{align}
   | \left\langle f , Z_{1,h} \right\rangle |
   & \leq C \|f\|_{**,1} \int_{\Omega} \dfrac{\mu^{n_h+ \frac{N}{p+1}}}{(1+\mu|y-x_1|)^{N-2}}\sum\limits_{j=1}^{k} \dfrac{\mu^{\frac{N}{q+1}+2}}{(1+ \mu|y-x_j|)^{\frac{N}{q+1}+ 2+\tau}} \ \dy \nonumber \\
   & \leq  C \mu^{n_h} \|f\|_{**,1}\sum\limits_{j=1}^{k}
   \int_{\Omega}  \dfrac{1}{(1+|z-\mu (x_j - x_1)|)^{\frac{N}{q+1} +2+\tau}}
   \dfrac{1}{(1+ |z|)^{N-2 }}  \ \dz \nonumber \\
   & \leq C \mu^{n_h} \|f\|_{**,1} \left(1 + \sum\limits_{j=2}^{k} \dfrac{1}{(\mu |x_j - x_1|)^{\frac{N}{q+1}+\tau}} \right) \leq C\mu^{n_h} \|f\|_{**,1}.\nonumber
\end{align}
Similarly, we have
\begin{equation}\label{ch2}
    | \left\langle g , Y_{1,h} \right\rangle | \leq C\mu^{n_h} \|g\|_{**,2}.
\end{equation}
By Lemma \ref{B2}, we have
\begin{equation}\label{po-1}
    \begin{aligned}
        &\quad| \left\langle \beta_2 \phi , Y_{1,h} \right\rangle | \\
     & \leq C \|\phi\|_{*,1} \int_{\Omega} \dfrac{\mu^{n_h+ \frac{N}{q+1}}}{(1+\mu|y-x_1|)^{p(N-2)-2}}\sum\limits_{j=1}^{k} \dfrac{\mu^{\frac{N}{q+1}}}{(1+ \mu|y-x_j|)^{\frac{N}{q+1}+\tau}} \ \dy\\
     & \leq C \|\phi\|_{*,1} \int_{\Omega} \dfrac{\mu^{n_h+ \frac{2N}{q+1}}}{(1+\mu|y-x_1|)^{\frac{N}{q+1}+\tau+p(N-2)-2}}\dy
     \\ &\hspace{5em}
        +C \|\phi\|_{*,1}\sum\limits_{j=2}^{k} \dfrac{1}{(\mu |x_j - x_1|)^{\tau}} \int_{\Omega}\dfrac{\mu^{n_h+ \frac{2N}{q+1}}}{(1+\mu|y-x_1|)^{\frac{N}{q+1}+p(N-2)-2}}\dy.
    \end{aligned}
\end{equation}
Note that if $q\leq 2$, $\frac{N}{q+1}+p(N-2)-2=\frac{(p+2)N}{q+1}>N$ and if $q>2$, $\frac{N}{q+1}+p(N-2)-2>N$ may not hold. We can compute that
\begin{equation}
    \begin{aligned}
        &\quad| \left\langle \beta_2 \phi , Y_{1,h} \right\rangle |\\
        &=\begin{cases}
          O\left( \|\phi\|_{*,1} \mu^{n_h+\frac{2N}{q+1}-N}\right)
          =O\left( \frac{\mu^{n_h}\|\phi\|_{*,1}}{\mu^{\sigma}}\right)
          \qquad&\text{if}\quad\frac{N}{q+1}+p(N-2)-2>N,\\
           O\left( \|\phi\|_{*,1} \mu^{n_h+\frac{2N}{q+1}-N}\ln \mu\right)
           =O\left( \frac{\mu^{n_h}\|\phi\|_{*,1}}{\mu^{\sigma}}\right)
          \qquad&\text{if}\quad\frac{N}{q+1}+p(N-2)-2=N,\\
             O\left( \|\phi\|_{*,1} \mu^{n_h+\frac{N}{q+1}-(p(N-2)-2)}\right)
          =O\left( \frac{\mu^{n_h}\|\phi\|_{*,1}}{\mu^{\sigma}}\right) \qquad&\text{if}\quad\frac{N}{q+1}+p(N-2)-2<N.
        \end{cases}
    \end{aligned}
\end{equation}

On the other hand, from Lemma~\ref{lemma:D6}, we have
    \begin{align}\label{estimate_1}
        \left| \left\langle -\Delta Z_{1,h}  - q (PU_*)^{q-1}Y_{1,h}, \phi \right\rangle\right|=O\left(\frac{\mu^{n_h}\|(\phi,\psi)\|_{*}}{\mu^{\sigma}} \right),
    \end{align}
and
    \begin{align}\label{estimate_2}
    \left|\left\langle -\Delta Y_{1,h}  - p (PV_*)^{p-1}Z_{1,h}, \psi \right\rangle \right|=O\left( \frac{\mu^{n_h}\|(\phi,\psi)\|_{*}}{\mu^{\sigma}} \right).
    \end{align}

Besides, we have
\begin{equation}\label{LHS}
    \sum\limits_{j=1}^k \left( p\left\langle V_j^{p-1}Z_{j,l}   , Z_{1,h}   \right\rangle +  q\left\langle U_j^{q-1}Y_{j,l}, Y_{1,h} \right\rangle \right) = \delta_{hl}\mu^{2n_h}a_h+o(\mu^{n_l+n_h-2}),
\end{equation}
for some $a_h>0$. Combining \eqref{2-18-2}, \eqref{ch2}, \eqref{po-1}, \eqref{estimate_1}, \eqref{estimate_2} and \eqref{LHS}, we obtain
\begin{equation}\label{c}
    c_h = \dfrac{1}{\mu^{n_h}}(o(\|(\phi,\psi)\|_{*})+ O(\|(f,g)\|_{**})), \;\; h=1,2.
\end{equation}
\end{proof}

\begin{proof}[Proof of Proposition \ref{blowup}]
    On the contrary, suppose that there exist sequences $\{(\phi_k, \psi_k)\}_{k \in \N} $, $\{(f_k,g_k)\}_{k \in \N} $
and $\{(c_{1,k}, c_{2,k})\}_{k \in \N}$ satisfying \eqref{linear-equation}, such that
$\|(\phi_k, \psi_k)\|_{*}=1$ for all $k \in \N$ and $\|(f_k,g_k)\|_{**}\to 0$ as $k \to +\infty$. 
We write
\[
L_k (\phi, \psi) = L_0(\phi, \psi)- \left(p \left(|PV_*|^{p-1}-v_\epsilon^{p-1}\right)\psi, q\left(|PU_*|^{q-1}-u_\epsilon^{q-1}\right)\phi\right),
\]
where
\[
L_0(\phi,\psi) =\left(-\Delta \phi - p v_\epsilon^{p-1} \psi, -\Delta \psi - q u_\epsilon^{q-1} \phi-\ep\beta_2 \phi\right), \quad (\phi,\psi) \in \mathbf{E}.
\]
For simplicity, we drop the subscript $k$.

Noting that $p\in(1,\frac{N}{N-2})\subset(1,2)$ for $N\geq 4$, we have
\begin{equation*}
    \begin{aligned}
        \left| |PV_*|^{p-1}-v_\epsilon^{p-1}\right|\leq C |PV|^{p-1}
    \end{aligned}
\end{equation*}
and
\begin{equation*}
    \begin{aligned}
        \left\|PU_*|^{q-1}-u_\epsilon^{q-1} \right|\le
        \begin{cases}
            CPU^{q-1}\qquad &q\leq 2,\\
            CPU^{q-1}+Cu_\epsilon^{q-2}PU\qquad &q> 2.
        \end{cases}
    \end{aligned}
\end{equation*}

It follows from \eqref{G1k}--\eqref{G2k1}, \eqref{G3k1}--\eqref{G4k1} and \eqref{10-27-1}--\eqref{11-27-1} that, for $x\in \Omega$,
    \begin{align}
        |\phi(x)|
&\le C\left[\int_{\Omega} \frac{\left(PV^{p-1}|\psi|\right)(y)}{|y-x|^{N-2}} dy + \int_{\Omega} \frac{\left(PU^{q-1} |\phi|\right)(y)}{|y-x|^{N-4}} dy + \int_{\Omega} \frac{\left(u_\epsilon^{q-2}PU |\phi|\right)(y)}{|y-x|^{N-4}} dy\right] \nonumber \\
&\ + C\int_{\Omega} \frac1{|y-x|^{N-2}} \left| \left(\sum_{h=1}^2c_h p\sum_{j=1}^k V_j^{p-1}Z_{j,h}+|f|\right)(y)\right| dy\nonumber\\
&\ +C\int_{\Omega} \frac1{|y-x|^{N-4}} \left|\left(\sum_{h=1}^2c_h q\sum_{j=1}^k U_j^{q-1}Y_{j,h}+|g|\right)(y)\right| dy . \label{1-31-1}
    \end{align}
and
\begin{align}
|\psi(x)| &\le C\left[\int_{\Omega} \frac{\left(PU^{q-1} |\phi|\right)(y)}{|y-x|^{N-2}} dy + \int_{\R^N} \frac{\left(u_\epsilon^{q-2}PU |\phi|\right)(y)}{|y-x|^{N-2}} dy + \int_{\R^N} \frac{\left(PV^{p-1}|\psi|\right)(y)}{|y-x|^{N-4}} dy\right] \nonumber \\
&\ +C\int_{\Omega} \frac1{|y-x|^{N-2}} \left|\left(\sum_{h=1}^2c_h q\sum_{j=1}^k U_j^{q-1}Y_{j,h}+|g|\right)(y)\right| dy \nonumber\\
&\ + C\int_{\Omega} \frac1{|y-x|^{N-4}} \left| \left(\sum_{h=1}^2c_h p\sum_{j=1}^k V_j^{p-1}Z_{j,h}+|f|\right)(y)\right| dy. \label{1-31-13}
\end{align}
In \eqref{1-31-1}--\eqref{1-31-13}, the integrals involving the function $u_\epsilon^{q-2}PU |\phi|$ can be substituted with $0$ provided $q \le 2$.

Next, we will prove that the following function
    \begin{equation}
        \left[\sum_{j=1}^{k}\frac{\mu^{\frac{N}{q+1}}}{(1+\mu |y-x_j|)^{\frac{N}{q+1}+\tau}}\right]^{-1}|\phi(y)| +
        \left[\sum_{j=1}^{k}\frac{\mu^{\frac{N}{p+1}}}{(1+\mu |y-x_j|)^{\frac{N}{p+1}+\tau}}\right]^{-1}|\psi(y)|
    \end{equation}
    can only achieve its maximum in $\cup_{i=1}^m \overline{B(x_i,M\mu^{-1})}$ for some large constant $M>1$. For this purpose, we will estimate the above function when $y \in \Omega\backslash \cup_{i=1}^m {B(x_i,M\mu^{-1})}$.

\textbf{Estimate $\phi$ and $\psi$.}
First, using Lemma \ref{l1-18-2}, \ref{B2} and \ref{B3}, we can prove
\begin{equation}\label{1-24-2}
\begin{medsize}
\displaystyle \int_{\Omega} \frac1{|y-x|^{N-2}} \Bigg|\Bigg(c_h \sum_{j=1}^k V_j^{p-1}Z_{j,h}\Bigg)(y)\Bigg| dy
\le C\left(\frac1{\mu^\sigma}\|(\phi, \psi)\|_* + \|(f, g)\|_{**}\right) \sum_{j=1}^{k}\frac{\mu^{\frac{N}{q+1}}}{(1+\mu |x-x_j|)^{\frac{N}{q+1}+\tau}},
\end{medsize}
\end{equation}
\begin{equation}\label{2-24-2}
\begin{medsize}
\displaystyle \int_{\Omega} \frac1{|y-x|^{N-2}} \Bigg|\Bigg(c_h \sum_{j=1}^k U_j^{q-1}Y_{j,h}\Bigg)(y)\Bigg| dy
\le C\left(\frac1{\mu^\sigma}\|(\phi, \psi)\|_* + \|(f, g)\|_{**}\right) \sum_{j=1}^{k}\frac{\mu^{\frac{N}{p+1}}}{(1+\mu |x-x_j|)^{\frac{N}{p+1}+\tau}},
\end{medsize}
\end{equation}
and
\begin{equation}\label{l20-23-21}
    \int_{\Omega}\frac{|f(y)|}{|y-x|^{N-2}} dy
\le C\|f\|_{**,1} \sum_{j=1}^{k}\frac{\mu^{\frac{N}{q+1}}}{(1+\mu
|x-x_j|)^{\frac{N}{q+1}+\tau}}
\end{equation}

\begin{equation}\label{l20-23-22}
    \int_{\Omega}\frac{|g(y)|}{|y-x|^{N-2}} dy
\le C\|g\|_{**,2} \sum_{j=1}^{k}\frac{\mu^{\frac{N}{p+1}}}{(1+\mu
|x-x_j|)^{\frac{N}{p+1}+\tau}}.
\end{equation}

In addition, assume that $N \ge 7$, $p\in (1,\frac{N}{N-2})$, we get
\begin{align}
&\quad \int_{\Omega} \frac1{|y-x|^{N-4}} \Bigg|\Bigg(c_h \sum_{j=1}^k V_j^{p-1}Z_{j,h}\Bigg)(y)\Bigg| dy\nonumber
\\
&\le C\left(\frac1{\mu^\sigma}\|(\phi, \psi)\|_* + \|(f, g)\|_{**}\right)
\sum_{j=1}^{k}\frac{\mu^{\frac{pN}{p+1}-4}}{(1+\mu |x-x_j|)^{N-4}} \nonumber\\
&\le C\left(\frac1{\mu^\sigma}\|(\phi, \psi)\|_* + \|(f, g)\|_{**}\right) \sum_{j=1}^{k}\frac{\mu^{\frac{N}{p+1}}}{(1+\mu|x-x_j|)^{\frac{N}{p+1}+\tau}}.\label{4-24-2}
\end{align}
The last inequality follows from
\begin{multline*}
    \frac{\mu^{\frac{pN}{p+1}-4}}{(1+\mu |x-x_j|)^{N-4}}\leq \frac{C\mu^{\frac{N}{p+1}}}{(1+\mu|x-x_j|)^{\frac{N}{p+1}+\tau}},\, j=1,\cdots,k, \text{ in a bounded domain }\Omega.
\end{multline*}

Similarly, we obtain
\begin{equation}\label{3-24-2}
\begin{medsize}
\displaystyle \int_{\Omega} \frac1{|y-x|^{N-4}} \Bigg|\Bigg(c_h \sum_{j=1}^k U_j^{q-1}Y_{j,h}\Bigg)(y)\Bigg| dy \\
\le C\left(\frac1{\mu^\sigma}\|(\phi, \psi)\|_* + \|(f, g)\|_{**}\right) \sum_{j=1}^{k}\frac{\mu^{\frac{N}{q+1}}}{(1+\mu |x-x_j|)^{\frac{N}{q+1}+\tau}},
\end{medsize}
\end{equation}

\begin{equation}\label{eq:f4}
\int_{\Omega} \frac{|f(y)|}{|y-x|^{N-4}} dy \le C \|f\|_{**,1}\sum_{j=1}^{k}\frac{\mu^{\frac{N}{p+1}}}{(1+\mu |x-x_j|)^{\frac{N}{p+1}+\tau}},
\end{equation}
and
\begin{equation}\label{eq:g4}
\int_{\Omega} \frac{|g(y)|}{|y-x|^{N-4}} dy \le C \|g\|_{**,2}\sum_{j=1}^{k}\frac{\mu^{\frac{N}{q+1}}}{(1+\mu |x-x_j|)^{\frac{N}{q+1}+\tau}}.
\end{equation}

Now we split the first integral on the right-hand side of \eqref{1-31-1} into several parts:
\begin{align*}
    &\int_{\Omega} \dfrac{1}{|y-z|^{N-2}}PV^{p-1}(z)|\psi(z)| dz\\
    &\leq C\|\psi\|_{*,2}\left(\int_{\Omega\backslash \cup_{i=1}^k {B(x_i,M\mu^{-1})}}+\sum_{i=1}^k \int_{B(x_i,M\mu^{-1})} \right)\frac{\mu^{\frac{pN}{p+1}}}{|y-z|^{N-2}}\\
    &\quad\times\left[\sum_{j=1}^{k}\frac{1}{(1+\mu|z-x_j|)^{N-2}}\right]^{p-1} \sum_{j=1}^{k}\frac{\dz}{(1+\mu|z-x_j|)^{\frac{N}{p+1}+\tau}}
    :=I_0+\sum_{i=1}^mI_i.
\end{align*}
Note that $N-2>\frac{N}{p+1}+\tau$ when $p\in(1,\frac{N}{N-2})$ and $N\ge5$. Therefore, for some small constant $\sigma>0$, we have
    \begin{align}
        I_0
        &\leq \frac {C\|\psi\|_{*,2}}{M^{\sigma}}\int_{\Omega\backslash \cup_{i=1}^k
        {B(x_i,M\mu^{-1})}} \frac{\mu^{\frac{pN}{p+1}}}{|y-z|^{N-2}} \left[\sum_{j=1}^{k}\frac{1}{(1+\mu|z-x_j|)^{\frac{N}{p+1}+\tau}}\right]^p \ \dz \nonumber\\
        &\leq \frac {C\|\psi\|_{*,2}}{M^{\sigma}}\int_{\Omega\backslash \cup_{i=1}^k {B(x_i,M\mu^{-1})}}\frac{\mu^{\frac{pN}{p+1}}}{|y-z|^{N-2}}\sum_{j=1}^{k}\frac{1}{(1+\mu|z-x_j|)^{\frac{pN}{p+1}+\tau}}\ \dz \label{I0}\\
        &\leq \frac {C\|\psi\|_{*,2}}{M^{\sigma}}\sum_{j=1}^{k}\frac{\mu^{\frac{N}{q+1}}}{(1+\mu|y-x_j|)^{\frac{N}{q+1}+\tau}}, \nonumber
    \end{align}
 where the second inequality follows from Lemma \ref{B1-1} and the last one follows from Lemma \ref{B3} and $\frac{pN}{p+1}=2+\frac{N}{q+1}$.

Next, we estimate $I_i$ for $i=1,2,\cdots,k$. Note that for any $\alpha\geq \tau$ and $z \in {B(x_k,M\mu^{-1})}$, it holds that
\begin{equation}\label{Ik-0}
    \sum_{j\neq k}\frac{1}{(1+\mu|z-x_j|)^{\alpha}}\leq C\leq \frac{C}{(1+\mu|z-x_k|)^{\alpha}}.
\end{equation}
Therefore, by Lemma \ref{B3} and \eqref{Ik-0}, for $y\in \Omega \backslash \cup_{i=1}^m {B(x_i,M\mu^{-1})}$,
\begin{equation}\label{Ik}
    \begin{split}
        I_i&\leq C\|\psi\|_{*,2}\int_{ {B(x_i,M\mu^{-1})}} \frac{1}{|y-z|^{N-2}}\frac{\mu^{\frac{pN}{p+1}}}{(1+\mu|z-x_i|)^{(N-2)(p-1)+\frac{N}{p+1}+\tau}}\ \dz\\
        &\leq C\|\psi\|_{*,2} \frac{\mu^{\frac{N}{q+1}}}{(1+\mu|y-x_i|)^{\min\{(N-2)(p-1)+\frac{N}{p+1}+\tau-2,N-2-\theta\}}}\\
        &\leq \frac {C\|\psi\|_{*,2}}{M^{\sigma}}\frac{\mu^{\frac{N}{q+1}}}{(1+\mu|y-x_i|)^{\frac{N}{q+1}+\tau}}.
    \end{split}
\end{equation}
Employing \eqref{I0} and \eqref{Ik} yields that for any $y\in \Omega\backslash \cup_{i=1}^k {B(x_i,M\mu^{-1})}$,
\begin{align}\label{omega1-1}
    \int_{\Omega} \dfrac{1}{|y-z|^{N-2}}PV^{p-1}(z)|\psi(z)| dz\leq \frac {C\|\psi\|_{*,2}}{M^{\sigma}}\sum_{j=1}^{k}\frac{\mu^{\frac{N}{q+1}}}{(1+\mu|y-x_j|)^{\frac{N}{q+1}+\tau}}.
\end{align}

Combining Lemma~\ref{l1-23-4} and a similar argument in \eqref{4-24-2}, we can obtain, for $x \in \Omega \setminus \cup_{i=1}^k {B(x_i,M\mu^{-1})}$,
\begin{equation}\label{15-1-2}
\int_{\Omega} \frac{\left(PU^{q-1} |\phi|\right)(y)}{|y-x|^{N-4}} dy
\le \frac{C}{M^{\sigma}} \|\phi\|_{*,1} \sum_{j=1}^k \frac{ \mu^{\frac{N}{q+1}} }{ (1+\mu |x-x_j| )^{ \frac{N}{q+1}+\tau }}.
\end{equation}

If $q > 2$, it holds that for $x \in \Omega \setminus \cup_{i=1}^k {B(x_i,M\mu^{-1})}$,
\begin{equation}\label{16-1-21}
\begin{aligned}
&\quad \int_{\Omega} \frac{\left(u_\epsilon^{q-2}PU |\phi|\right)(y)}{|y-x|^{N-4}} dy \\
&\le \frac{C}{M^{\sigma}} \|\phi\|_{*,1} \int_{\Omega \setminus \cup_{j=1}^k B_{M\mu^{-1}}(x_j)} \frac1{|y-x|^{N-4}} \left[\sum_{j=1}^{k}\frac{\mu^{\frac{N}{q+1}}}{(1+\mu
|y-x_j|)^{\frac{N}{q+1}+\tau}}\right]^2 dy \\
&\ + C \|\phi\|_{*,1} \sum_{j=1}^k \int_{B_{M\mu^{-1}}(x_j)} \frac1{|y-x|^{N-4}} \frac{\mu^{\frac{2N}{q+1}}}{(1+\mu
|y-x_j|)^{N-2+\frac{N}{q+1}+\tau}} dy.
\end{aligned}
\end{equation}
Since $2 < \frac{N}{q+1}+\tau < \frac{N}{2}$ for $N \ge 6$, the first term is bounded by
\begin{equation}\label{17-1-21}
\begin{aligned}
&\quad \frac{C}{M^{\sigma}} \|\phi\|_{*,1} \int_{\Omega \setminus \cup_{j=1}^k B_{M\mu^{-1}}(x_j)} \frac1{|y-x|^{N-4}} \cdot k\sum_{j=1}^{k}\frac{\mu^{\frac{2N}{q+1}} }{(1+\mu
|y-x_j|)^{2(\frac{N}{q+1}+\tau)}} dy \\
&\le \frac{C}{M^{\sigma}} \|\phi\|_{*,1} \sum_{j=1}^{k}\frac{k\mu^{\frac{2N}{q+1}-4}}{(1+\mu
|x-x_j|)^{2(\frac{N}{q+1}+\tau-2)}} \\
&\le \frac{C}{M^{\sigma}} \|\phi\|_{*,1} \sum_{j=1}^k \frac{ \mu^{\frac{N}{q+1}} }{ (1+\mu |x-x_j| )^{ \frac{N}{q+1}+\tau }}.
\end{aligned}
\end{equation}
Besides, the second term is bounded by
\[
C \|\phi\|_{*,1} \sum_{j=1}^k \frac{\mu^{\frac{2N}{q+1}-4}}{(1+\mu|x-x_j|)^{N-4}} \le \frac{C}{M^{\sigma}} \|\phi\|_{*,1} \sum_{j=1}^k \frac{ \mu^{\frac{N}{q+1}} }{ (1+\mu |x-x_j| )^{ \frac{N}{q+1}+\tau }}.
\]

Consequently, the above estimates gives that for $y \in \Omega \setminus \cup_{i=1}^k B_{M\mu^{-1}}(x_i)$,
\begin{equation}\label{21-1-6-21}
    \left[\sum_{j=1}^{k}\frac{\mu^{\frac{N}{q+1}}}{(1+\mu |y-x_j|)^{\frac{N}{q+1}+\tau}}\right]^{-1}|\phi_k(y)| \le C\left(\frac{1}{M^{\sigma}}+o(1)\right)\|(\phi_k, \psi_k)\|_{*} + C\|(f_k, g_k)\|_{**}.
\end{equation}

Analogously, by employing \eqref{1-31-13}, we can prove that for $y \in \Omega \setminus \cup_{i=1}^k B_{M\mu^{-1}}(x_i)$,
\begin{equation}\label{n21-1-6-21}
    \left[\sum_{j=1}^{k}\frac{\mu^{\frac{N}{p+1}}}{(1+\mu |y-x_j|)^{\frac{N}{p+1}+\tau}}\right]^{-1}|\psi_k(y)| \le C\left( \frac{1}{M^{\sigma}}+o(1)\right)\|(\phi_k, \psi_k)\|_{*} + C\|(f_k, g_k)\|_{**}.
\end{equation}

\textbf{Derive a contradiction.}
Since $\|(\phi_k, \psi_k)\|_{*} = 1$, we can deduce from \eqref{21-1-6-21} and \eqref{n21-1-6-21}  that there are $R>0$ and $c_0 > 0$ such that
\[
    \|\mu^{-\frac{N}{q+1}}\phi_k\|_{L^\infty(B_{R/\mu}(x_i))} + \|\mu^{-\frac{N}{p+1}}\psi_k\|_{L^\infty(B_{R/\mu}(x_i))} \geq c_0 > 0,
\]
for some $i$. Define
\[
(\bar{\phi}_k(y), \bar{\psi}_k(y))=(\mu^{-\frac{N}{q+1}}\phi_k(\mu^{-1}y+x_i), \mu^{-\frac{N}{p+1}}\psi_k(\mu^{-1}y+x_i)).
\]
Then $(\bar{\phi}_k, \bar{\psi}_k)$ converges uniformly in any compact set of $\Omega$ to a solution $(\Phi, \Psi)$ of
\begin{equation}\label{linear}
   \begin{cases}
   -\Delta \Phi = p V_{0,\Lambda}^{p-1} \Psi \;\;\;  \hbox{in } \R^N,\\
   -\Delta \Psi = q U_{0,\Lambda}^{q-1} \Phi  \;\;\;  \hbox{in } \R^N,
   \end{cases}
\end{equation}
for some $\Lambda \in [\lambda_{00},\lambda_{00}^{-1}]$. However, since $(\phi_k, \psi_k) \in E$, $(\Phi, \Psi)$ is perpendicular to the kernel of equation \eqref{linear}. Hence $\Phi=0$ and $\Psi=0$, which is a contradiction.
\end{proof}

The following lemma is essential to the reduction argument.
\begin{lemma}\label{prop1}
    Suppose $Y_s$ is defined as \eqref{Ys} and $F$ is defined as \eqref{F}. Then $F$ and
    \[
        \overline{F}:=\text{span}\left\{ \left( \sum\limits_{j=1}^{k} pV_{j}^{p-1} Z_{j,l}, \sum\limits_{j=1}^{k} qU_{j}^{q-1} Y_{j,l} \right),\ l=1,2.\right\}
    \]
    are topological complements of each other, and $Y_s=F\oplus \bar{F}$. Moreover, define the projection \textbf{P} from $Y_s$ to $F$ as follows:
    $$
        \textbf{P}(f,g)= \left(f-\sum\limits_{l=1}^{2} c_l\sum\limits_{j=1}^{k} pV_{j}^{p-1} Z_{j,l},g-\sum\limits_{l=1}^{2} c_l \sum\limits_{j=1}^{k} qU_{j}^{q-1} Y_{j,l} \right),\,\, (f,g)\in Y_s
    $$
    where $\{c_l\}$ are chosen so that $\textbf{P}(f,g)\in F$.
    Then \textbf{P} is a linear bounded operator from $Y_s$ to $F$.
    \end{lemma}

    \begin{proof}
    It is sufficient to show that for any $(f,g)\in Y_s$, there exist $c_1, c_2$ and a unique pair $(f_0,g_0)\in F$, such that
    $$(f,g)=(f_0,g_0)+\sum\limits_{l=1}^{2}c_l\left(\sum\limits_{j=1}^{k} pV_j^{p-1} Z_{j,l} , \sum\limits_{j=1}^{k} qU_j^{q-1} Y_{j,l}\right),$$
    which is equivalent to solving the following equations involved $c_l$:
    \begin{equation}\label{c_l_0}
        \begin{gathered}
            \sum\limits_{l=1}^{2}c_l \Big\langle \big(\sum\limits_{j=1}^{k} p V_j^{p-1} Z_{j,l},  \sum\limits_{j=1}^{k} q U_j^{q-1} Y_{j,l} \big), \big(\sum\limits_{j=1}^{k}  \bar{Z}_{j,h},  \sum\limits_{j=1}^{k} \bar{Y}_{j,h} \big)  \Big\rangle
            = \Big\langle (f,g) , \big(\sum\limits_{j=1}^{k}  \bar{Z}_{j,h},  \sum\limits_{j=1}^{k} \bar{Y}_{j,h} \big)\Big\rangle,
            \\  h=1,2.
        \end{gathered}
    \end{equation}
     Recall that
     \begin{equation*}
\widetilde{L}_\ep(\bar{Y}_{j,h},\bar{Z}_{j,h}) =\left (\frac{\partial (V_j^p)}{\partial \Box_h},\frac{\partial (U_j^q)}{\partial \Box_h} \right)
\end{equation*}
for $j = 1,\ldots,k$, where $\Box_h$ denotes $r$ if $h=1$, $\mu$ if $h=2$. Denote
\[
(\tilde{u}_j,\tilde{v}_j)=(\bar{Y}_{j,h}-Y_{j,h},\bar{Z}_{j,h}-Z_{j,h}).
\]
Then, we have
\[
\widetilde{L}_\ep(\tilde{u}_j,\tilde{v}_j)=(0,\ep \beta_2 Y_{j,h}),
\]
which implies that for any $x\in\Omega$,
\begin{align}\label{bar_uj}
    \tilde{u}_j(x)=0,
\end{align}
and
\begin{align}\label{bar_vj}
    |\tilde{v}_j(x)|\leq \int_{\Omega}\frac{C}{|x-y|^{N-2}}Y_{j,h}\leq
\frac{\mu^{n_h+\frac{N}{q+1}-2}}{(1+\mu|x-x_j|)^{p(N-2)-4}}.
\end{align}

We employ \eqref{bar_uj}, \eqref{bar_vj} and Lemma~\ref{B2} to obtain
\begin{equation}
    \begin{aligned}
       \Big\langle \big(  \sum\limits_{j=1}^{k} p V_j^{p-1} Z_{j,l},\sum\limits_{j=1}^{k} q U_j^{q-1} Y_{j,l}\big), \big(\sum\limits_{j=1}^{k}\tilde{v}_j,\sum\limits_{j=1}^{k}\tilde{u}_j \big) \Big\rangle =o(k\mu^{n_l+n_h}).
    \end{aligned}
\end{equation}

Moreover, it is known that
\begin{equation}
    \begin{aligned}
        \Big\langle \big(  \sum\limits_{j=1}^{k} p V_j^{p-1} Z_{j,l},\sum\limits_{j=1}^{k} q U_j^{q-1} Y_{j,l}\big), \big(\sum\limits_{j=1}^{k}{Z}_{j,h},\sum\limits_{j=1}^{k}{Y}_{j,h} \big) \Big\rangle=\delta_{hl}k\mu^{2n_h}a_h+o(k\mu^{n_l+n_h}),
    \end{aligned}
\end{equation}
for some $a_h>0$, $h=1,2$.

Thus,
\begin{equation}\label{LHS_1}
    \begin{aligned}
        \Big\langle \big(\sum\limits_{j=1}^{k} p V_j^{p-1} Z_{j,l},  \sum\limits_{j=1}^{k} q U_j^{q-1} Y_{j,l} \big), \big(\sum\limits_{j=1}^{k}  \bar{Z}_{j,h},  \sum\limits_{j=1}^{k} \bar{Y}_{j,h} \big)  \Big\rangle=\delta_{hl}k\mu^{2n_h}a_h+o(k\mu^{n_l+n_h}).
    \end{aligned}
\end{equation}

On the other hand, since $\Omega$ is bounded, it holds that $\frac{1}{\mu}
\leq
\frac{C}{1+\mu|y-x_j|}.$
Thus,
\begin{equation}\label{tilde_v}
    |\tilde v_j(y)|
\leq
C\mu^{n_h}
\frac{\mu^{\frac{N}{p+1}}}
{(1+\mu|y-x_j|)^{
p(N-2)-2+\frac{N}{p+1}-\frac{N}{q+1}
}}.
\end{equation}
Consequently,
\begin{align}
\left|\int_{\Omega}\sum_{j=1}^{k} \tilde v_j\cdot f \right|
&\leq
Ck\mu^{n_h}\|f\|_{**,1}
\int_{\Omega}
\frac{\mu^{\frac{N}{p+1}}}
{(1+\mu|y-x_j|)^{
p(N-2)-2+\frac{N}{p+1}-\frac{N}{q+1}
}}
\nonumber \\
&\qquad\qquad\qquad\qquad\times
\sum_{i=1}^{k}
\frac{\mu^{\frac{N}{q+1}+2}}
{(1+\mu|y-x_i|)^{\frac{N}{q+1}+2+\tau}}
\,dy \nonumber \\
&\leq
Ck\mu^{n_h}\|f\|_{**,1},
\end{align}
where we have used
$
p(N-2)+\frac{N}{p+1}+\tau>N,
$
for
$
N\geq 7,\, p\in\left(1,\frac{N}{N-2}\right).
$

Computing as \eqref{tilde_v}, we get
\begin{align}\label{RHS_1}
\left|
\Big\langle
(f,g),
\left(
\sum_{j=1}^{k}\bar Z_{j,h},
\sum_{j=1}^{k}\bar Y_{j,h}
\right)
\Big\rangle
\right|
&\leq
Ck\mu^{n_h}
\left(
\|f\|_{**,1}+\|g\|_{**,2}
\right) \nonumber \\
&\leq
Ck\mu^{n_h}\|(f,g)\|_{**}.
\end{align}

    Inserting \eqref{LHS_1} and \eqref{RHS_1} into \eqref{c_l_0}, we conclude that \eqref{c_l_0} is solvable and
   $$c_h=O\left(\mu^{-n_h}\|(f,g)\|_{**}\right).$$

    \textbf{Prove that P is bounded from $Y_s$ to $F$.}
         For any $y \in \Omega_i$, and $\zeta \geq \tau$, we have
        \begin{equation}\label{I_32_2}
            \begin{aligned}
                \sum\limits_{j=1}^m \frac{1}{(1+\mu|y-x_j|)^{\zeta}}
                &\leq \frac{1}{(1+\mu|y-x_1|)^{\zeta}}+\frac{1}{(1+\mu|y-x_1|)^{\zeta-\tau}}\sum\limits_{j=2}^m \frac{1}{(\mu|x_1-x_j|)^{\tau}}\\
                &\leq  \frac{C}{(1+\mu|y-x_1|)^{\zeta-\tau}}.
            \end{aligned}
        \end{equation}
       Hence,
\begin{align*}
    \left|\left|\sum\limits_{j=1}^{k} pV_j^{p-1} Z_{j,l}\right|\right|_{**,1} &\leq
                C\mu^{n_l}\sup\limits_{y\in \Omega}\left(\sum\limits_{j=1}^{k}\frac{1}{(1+\mu|y-x_{j}|)^{\frac{N}{q+1}+2+\tau}}\right)^{-1}\cdot
                \sum\limits_{j=1}^{k}\frac{1}{(1+\mu|y-x_{j}|)^{(N-2)p}}\\
                &\leq C\mu^{n_l}\max\limits_{1\leq i\leq m}\sup\limits_{y\in \Omega_i} \left(\sum\limits_{j=1}^{k}\frac{(1+\mu|y-x_{i}|)^{(N-2)p-\tau}}{(1+\mu|y-x_{j}|)^{\frac{N}{q+1}+2+\tau}}\right)^{-1} \leq C \mu^{n_l},
\end{align*}
        where we have used
        $p(N-2)-\tau-\left(\frac{N}{q+1}+2+\tau\right)=
\frac{p}{p+1}\big((N-2)p-4\big)>0$,
for
$
N\geq 7,\, p\in\left(1,\frac{N}{N-2}\right).
$

Similarly, we get
\begin{align}
\left\|
\sum_{j=1}^{k}qU_j^{q-1}Y_{j,l}
\right\|_{**,2}
&\leq
C\mu^{n_l}\max\limits_{1\leq i\leq m}\sup\limits_{y\in \Omega_i} \left(\sum\limits_{j=1}^{k}\frac{(1+\mu|y-x_{i}|)^{q(p(N-2)-2)-\tau}}{(1+\mu|y-x_{j}|)^{\frac{N}{p+1}+2+\tau}}\right)^{-1} \nonumber \\
&\leq C\mu^{n_l}.
\end{align}

        As a result, we obtain
        \begin{equation*}
            \begin{aligned}
                 \left|\left|\sum\limits_{l=1}^{2} c_l\big(\sum\limits_{j=1}^{k} pV_j^{p-1} Z_{j,l},  \sum\limits_{j=1}^{k} qU_j^{q-1} Y_{j,l}\big)\right|\right|_{**} &\leq \sum\limits_{l=1}^{2} c_l \left(\left| \left| \sum\limits_{j=1}^{k} pV_j^{p-1} Z_{j,l} \right| \right|_{**,1} + \left| \left|\sum\limits_{j=1}^{k} qU_j^{q-1} Y_{j,l} \right| \right|_{**,2}\right)\\
                 &\leq O(\|(f,g)\|_{**}),
            \end{aligned}
        \end{equation*}
        which implies that
        $$\|\textbf{P}(f,g)\|_{**}\leq C \|(f,g)\|_{**}.$$

    \end{proof}

    Define the operator
    $$T(\phi,\psi):=(\widetilde{L}_\ep)^{-1}\cdot \textbf{P}(p(PV_*)^{p-1}\psi,q(PU_*)^{q-1}\phi).$$
It follows from Lemma \ref{blowup} and \ref{prop1} that $T$ is a bounded linear operator from $E$ to itself. Moreover, we can derive from the Arzelà-Ascoli Theorem that $T$ is compact. Hence, Lemma~\ref{blowup} and Fredholm alternative give

\begin{lemma}\label{existence1}
    There is a $k_0>0$ such that if $k>k_0$, $\mu \in [\lambda_{00}^{-1}k^{\frac{p+1}{p}}, \lambda_{00}k^{\frac{p+1}{p}}]$, \eqref{linear-equation} has a unique solution $(\phi,\psi):=\textbf{L}_k(f,g)$ satisfying
    \begin{align}\label{phi_psi_c}
         \|(\phi,\psi)\|_{*} \leq C\|f,g\|_{**}, \;\; |c_l| \leq C\mu^{-n_l}\|f,g\|_{**}.
    \end{align}
\end{lemma}

\section{Lyapunov-Schmidt reduction }\label{sec:L-S}
Now we consider the following non-linear problem:
\begin{equation}\label{linear-equation2}
   \begin{cases}
     \displaystyle L(\phi,\psi) = (l_1,l_2)+ \big(N_1(\psi),N_2(\phi)\big) + \sum_{l=1}^2 c_l
    \big(  p\sum_{j=1}^k V_j^{p-1}Z_{j,l},\,
                     q\sum_{j=1}^k U_j^{q-1}Y_{j,l}\big)\;\;\; \hbox{in} \;\;\; \Omega,\\
    \displaystyle  (\phi, \psi) \in E ,
   \end{cases}
\end{equation}
where
\begin{equation}
\begin{aligned}
\begin{cases}
\displaystyle l_1 := \Delta PU_*+ |PV_*|^{p-1}PV_*,\\
\displaystyle l_2 := \Delta PV_*+ |PU_*|^{q-1}PU_*+\epsilon\beta_2 PU_*\\
\end{cases}
\end{aligned}
\end{equation}
and
\begin{equation}
\begin{aligned}
\begin{cases}
\displaystyle N_1(\psi) := |PV_*+\psi|^{p-1}( PV_*+\psi) -|PV_*|^{p-1} PV_*-p|PV_*|^{p-1}\psi, \\
\displaystyle N_2(\phi) := |PU_*+\phi|^{q-1}( PU_*+\phi) -|PU_*|^{q-1} PU_*-q|PU_*|^{q-1}\phi.
\end{cases}
\end{aligned}
\end{equation}

We estimate the quantities $\|(l_1,l_2)\|_{**}$ and $\|\big(N_1(\psi),N_2(\phi)\big)\|_{**}$.

\begin{lemma}\label{N}
    Suppose $N\geq 7$, $p\in(1, \frac{N}{N-2})$ and $(p,q)$ satisfies condition \eqref{c-hyperbola}, then for $\|(\phi,\psi)\|_*$ small enough, it holds that
    \[
        \|N(\phi,\psi)\|_{**} \leq C\|(\phi,\psi)\|_{*}^{p},
    \]
where $N(\phi,\psi)=\big(N_1(\psi),N_2(\phi)\big)$.
\end{lemma}

\begin{proof}
   Recall that $N_2(\phi) = |PU_*+\phi|^{q-1}( PU_*+\phi) -|PU_*|^{q-1} PU_*-q|PU_*|^{q-1}\phi$. Thus, we have
\[ |N_2(\phi)| \leq
\begin{cases}
   C|\phi|^q, \;\;\; \hbox{if}\;\; q \leq 2 ;\\
   C|PU_*|^{q-2}\phi^2 + C|\phi|^q \;\; \hbox{if}\;\; q > 2.
\end{cases}
\]
First, we can derive from Lemma \ref{B1-1} that
\[
\begin{split}
    |\phi|^{q}\leq \mu^{\frac{qN}{q+1}} \|\phi\|_{*,1}^{q} \left(\sum\limits_{j=1}^{k} \dfrac{1}{(1+\mu|y-x_j|)^{\frac{N}{q+1}+\tau}}\right)^{q}
    \leq C \|\phi\|_{*,1}^{q} \sum\limits_{j=1}^{k}\dfrac{\mu^{\frac{N}{p+1}+2}}{(1+\mu|y-x_j|)^{\frac{N}{p+1}+2+\tau}}.
\end{split}
\]
Note that
\begin{align}\label{l-1}
    u_\epsilon(y)\leq C\leq C\sum\limits_{j=1}^{k} \dfrac{\mu^{\frac{N}{q+1}}}{(1+\mu|y-x_j|)^{\frac{N}{q+1}+\tau}},\qquad y\in \Omega.
\end{align}
If $q>2$, from \eqref{l-1}, Lemma \ref{B1-1} and Lemma \ref{l1-23-4}, we estimate that
\[
\begin{split}
    \big|PU_*\big|^{q-2}\phi^2&\leq C\|\phi\|_{*,1}^{2}
    \left( u_\epsilon^{q-2}+PU^{q-2}\right)\left(\sum\limits_{j=1}^{k} \dfrac{\mu^{\frac{N}{q+1}}}{(1+\mu|y-x_j|)^{\frac{N}{q+1}+\tau}}\right)^2\\
    &\leq C\|\phi\|_{*,1}^{2}\left(\sum\limits_{j=1}^{k} \dfrac{\mu^{\frac{N}{q+1}}}{(1+\mu|y-x_j|)^{\frac{N}{q+1}+\tau}}\right)^q\\
    &\leq  C\|\phi\|_{*,1}^{2} \sum\limits_{j=1}^{k}\dfrac{\mu^{\frac{N}{p+1}+2}}{(1+\mu|y-x_j|)^{\frac{N}{p+1}+2+\tau}},\qquad y\in\Omega.
\end{split}
\]
Hence, we have
$$\big|N_2(\phi)\big|\leq C\|\phi\|_{*,1}^{\min \{q,2\}} \sum\limits_{j=1}^{k}\dfrac{\mu^{\frac{N}{p+1}+2}}{(1+\mu|y-x_j|)^{\frac{N}{p+1}+2+\tau}}.$$
Similarly, since $p\in(1, \frac{N}{N-2})\subset(1,2)$ when $N>4$, we have that
$$
    |N_1(\psi)| \leq C|\psi|^p .
$$
Thus, we have
$$\big|N_1(\psi)\big|\leq C\|\psi\|_{*,2}^{p} \sum\limits_{j=1}^{k}\dfrac{\mu^{\frac{N}{q+1}+2}}{(1+\mu|y-x_j|)^{\frac{N}{q+1}+2+\tau}}.$$
Since $\min\{q,2 \} > p$ and $\|(\phi, \psi)\|_{*}$ is small enough, we have
\[
   \|N(\phi, \psi)\|_{**} \leq C\|(\phi,\psi)\|_{*}^{p}.
\]
\end{proof}

\begin{lemma}\label{l}
  If $N\geq 8$, $p\in(1, \frac{N}{N-2})$ or $N=7$, $p\in (\frac{11+\sqrt{57}}{16},\frac{7}{5})$ and $(p,q)$ satisfies condition \eqref{c-hyperbola}, then we have
   \[
       \|(l_1, l_2)\|_{**} \leq C\mu^{-\frac{N}{2(q+1)}-\sigma},
   \]
where $\sigma > 0$ is a fixed small constant.
\end{lemma}
\begin{proof}
   Recall that
    \begin{equation}
\begin{aligned}
      l_1 &= \Delta PU_*+ |PV_*|^{p-1}PV_*\\
     &=|PV_*|^{p-1}PV_*-|v_\epsilon|^{p-1}v_\epsilon+PV^p:=H_1,
\end{aligned}
\end{equation}
and
\begin{equation}
    \begin{aligned}
         l_2 &= \Delta PV_*+ |PU_*|^{q-1}PU_*+\ep\beta_2 PU_*\\
&=\left\{|PU_*|^{q-1}PU_*-|u_\epsilon|^{q-1}u_\epsilon+\sum_{j=1}^k U_j^q\right\}+\ep\beta_2 PU:=H_2+K_1.
    \end{aligned}
\end{equation}

We will estiamte $H_1$, $H_2$, $K_1$ respectively. Without loss of generality, we may assume $y \in \Omega_1$. Then, we have $|y-x_j| \geq |y-x_1|$. Let $S = B_{\frac{\pi}{2}r_0^{-1}k^{-1}}(x_1) \subset \Omega_1$.

\textbf{Estimate $H_1$ in $S$.}

It holds that
\begin{align}\label{est:01}
     \sum_{j=2}^k PV_j \le C\sum_{j=2}^k V_j&\le C \sum_{j=2}^k\frac{\mu^{\frac{N}{p+1}}}{|\mu(x_1-x_j)|^{N-2}} \nonumber\\
     &\approx  C\mu^{\frac{N}{p+1}}\cdot (\frac{k}{\mu})^{N-2} \le C \frac{\mu^{\frac{N}{p+1}}}{(1+\mu|x_1-x_j|)^{N-2}}\le C V_1.
\end{align}
Therefore,
\[|H_1| \le C + \left\|PV_*|^{p-1}PV_*+PV^{p}\right| \le CV_1^{p-1} \le C\frac{\mu^{\frac{(p-1)N}{p+1}}}{(1+\mu|y-x_1|)^{(p-1)(N-2)}},\quad y\in S.\]
We shall show that there exists $t > \frac{N}{2(q+1)}$ such that
\[\frac{\mu^{\frac{(p-1)N}{p+1}}}{(1+\mu|y-x_1|)^{(p-1)(N-2)}} \le \frac{C}{\mu^t}
\frac{\mu^{\frac{N}{q+1}+2}}{(1+\mu|y-x_1|)^{\frac{N}{q+1}+2+\tau}}, \quad y \in S,\]
which is equivalent to
\begin{equation}\label{eq:error1}
(1+\mu|y-x_1|)^{\frac{N}{q+1}+2+\tau-(p-1)(N-2)} \le C \mu^{\frac{N}{q+1}+2-\frac{(p-1)N}{p+1}-t}, \quad y \in S.
\end{equation}

Note that $\frac{N}{q+1}+2+\tau-(p-1)(N-2)> 0$, then \eqref{eq:error1} is valid provided
\[
(\frac{\mu}{k})^{\frac{pN}{p+1}+\tau-(p-1)(N-2)}\le C \mu^{\frac{N}{p+1}-t}.
\]
So we can take
\[t = \frac{N}{p+1} - (1-\tau)\left[\frac{pN}{p+1}+\tau-(p-1)(N-2)\right].\]
Then $t>\frac N{2(q+1)}$ holds for $N \ge 3$ and $p \in (1,\frac{N}{N-2})$.

\textbf{Estimate $H_1$ in $\Omega_1 \setminus S$.}

\textbf{Step 1.} We estimate $v_\epsilon^{p-1} PV$. Denote
$$S_{\rho}=B_\rho(x_1)\cap \Omega_1.$$

We determine $t_1>0$ such that
\begin{equation}\label{1-13-2-0}
    \begin{aligned}
        v_\epsilon^{p-1}(y) PV(y)\leq C\sum_{j=1}^k V_j(y)\leq \frac{C}{\mu^{t_1}}
\frac{\mu^{\frac{N}{q+1}+2}}{(1+\mu|y-x_j|)^{\frac{N}{q+1}+2+\tau}}, \quad y \in \Omega_1\setminus S_{\mu^{-\beta}},
    \end{aligned}
\end{equation}
where $0<\beta<\tau$.
It suffices to prove the above inequality by showing
\begin{equation}\label{1-13-2}
(1+\mu|y-x_j|)^{-\frac{N}{p+1}+2+\tau}\le C
\mu^{N-\frac{2N}{p+1}-t_1}, \quad y \in \Omega_1\setminus S_{\mu^{-\beta}} .
\end{equation}

Note that
$$\frac{N}{p+1}-2-\tau>0,\quad \text{if}\quad N \ge 7,\,p \in (1,\frac{N}{N-2}).$$
In addition, in $\Omega_1\setminus S_{\mu^{-\beta}}$, it holds
\[\mu|y-x_j|\ge c \mu^{1-\beta}\],
for some $c > 0$. Thus, \eqref{1-13-2} will follow if
\[
\mu^{(1-\beta)(-\frac{N}{p+1}+2+\tau)}\le
C\mu^{N-\frac{2N}{p+1}-t_1}.
\]
So we can choose
\[t_1(\beta) =N-\frac{2N}{p+1}+(1-\beta)\left(\frac{N}{p+1}-2-\tau\right).\]
Define
\begin{align}\label{def:B1}
    \mathcal{B}_1:=\sup \{\beta\in(0,\tau): \eqref{1-13-2-0} \text{ with } t_1(\beta)>\frac{N}{2(q+1)}\text{ holds for }N \ge 7,\,p \in (1,\frac{N}{N-2}).\}
\end{align}

Since
$$t_1(0)=N-2-\frac{N}{p+1}-\tau>\frac{N}{2(q+1)},$$
we have $\mathcal{B}_1\in (0,\tau)$ exists.

\textbf{Step 2.} We estimate $v_\epsilon PV^{p-1}$.

We follows from \eqref{pvi}, \eqref{V10est} and \eqref{tip_1} to obtain
\[
PV(y)\leq C \sum_{j=1}^k V_j(y)\leq \frac{1}{\mu^{\frac{N}{q+1}}}\left( \frac{1}{|y-x_1|^{N-2}}+ \frac{k^{1+\theta}}{|y-x_1|^{N-3-\theta}}\right),\qquad y\in \Omega_1\setminus S.
\]

We determine $t_2>0$ such that
\begin{equation}\label{1-13-2-01}
    \begin{aligned}
        PV^{p-1}(y)&\leq  \frac{C}{\mu^{\frac{(p-1)N}{q+1}}}\left( \frac{1}{|y-x_1|^{(p-1)(N-2)}}+ \frac{k^{(p-1)(1+\theta)}}{|y-x_1|^{(p-1)(N-3-\theta)}}\right)\\
        &\leq \frac{C}{\mu^{t_2}}
\frac{\mu^{\frac{N}{q+1}+2}}{(1+\mu|y-x_1|)^{\frac{N}{q+1}+2+\tau}}, \quad y \in S_{\mu^{-\beta}}\setminus S.
    \end{aligned}
\end{equation}
It suffices to prove the above inequality by showing
\begin{align*}
    |y-x_1|^{\frac{N}{q+1}+2+\tau-(p-1)(N-2)}\leq C \mu^{\frac{(p-1)N}{q+1}-\tau-t}
\end{align*}
and
\begin{align}\label{1-13-2-1}
    |y-x_1|^{\frac{N}{q+1}+2+\tau-(p-1)(N-3-\theta)}\leq C \mu^{\frac{(p-1)N}{q+1}-(1+(p-1)(1+\theta))\tau-t}
\end{align}
where
\[
\frac{N}{q+1}+2+\tau-(p-1)(N-3-\theta)>\frac{N}{q+1}+2+\tau-(p-1)(N-2)>0.
\]
Thus, we can choose
\begin{align*}
    t_2(\beta)&=\min \Big\{\frac{(p-1)N}{q+1}-p\tau+\beta\big(\frac{N}{q+1}+2+\tau-(p-1)(N-3)\big),\\
    &\hspace{10em}\frac{(p-1)N}{q+1}-\tau+\beta\big(\frac{N}{q+1}+2+\tau-(p-1)(N-2)\Big\}\\
    &=\frac{(p-1)N}{q+1}-p\tau+\beta\big(\frac{N}{q+1}+2+\tau-(p-1)(N-3)\big).
\end{align*}
Define
\begin{equation}\label{def:B2}
    \mathcal{B}_2:=\inf \{\beta\in(0,\tau): \eqref{1-13-2-01} \text{ with } t_2(\beta)>\frac{N}{2(q+1)}\text{ holds for }N \ge 7,\,p \in (1,\frac{N}{N-2}).\}
\end{equation}

Since
$$t_2(\tau)-\frac{N}{2(q+1)}=\frac{(N-2)p^2-(N-2)p+6}{2(p+1)^2}=\frac{(N-2)p(p-1)+6}{2(p+1)^2}>0,$$
we have $\mathcal{B}_2\in (0,\tau)$ exists.

\textbf{Step 3.}
From Lemma \ref{lemma:alg_1}, we can choose a $\mathcal{B}_0\in(\mathcal{B}_2,\mathcal{B}_1)$. Then we can divide $\Omega_1\setminus S$ as follows.
\begin{align*}
    \Omega_1\setminus S&=\Big\{\big\{v_\epsilon\leq PV\big\}\cap \big\{S_{\mu^{-\mathcal{B}_0}}\setminus S\big\} \Big\}
    \cup
    \Big\{\big\{v_\epsilon\leq PV\big\}\cap \big\{\Omega_1\setminus S_{\mu^{-\mathcal{B}_0}} \big\} \Big\}\\
    &\hspace{5em}\cup
    \Big\{\big\{v_\epsilon> PV\big\}\cap \big\{S_{\mu^{-\mathcal{B}_0}}\setminus S\big\} \Big\}
    \cup
    \Big\{\big\{v_\epsilon> PV\big\}\cap \big\{\Omega_1\setminus S_{\mu^{-\mathcal{B}_0}}\big\} \Big\}\\
    &:=\Omega^{11}\cup\Omega^{12}\cup \Omega^{21}\cup \Omega^{22}.
\end{align*}

By direct computation, we have
\begin{equation}
    \begin{aligned}
        |H_1| &\leq
        \begin{cases}
            \displaystyle C(p)v_\epsilon PV^{p-1}\leq
            \frac{C(p)}{\mu^{t_2(\mathcal{B}_0)}}
\sum_{j=1}^k\frac{\mu^{\frac{N}{q+1}+2}}{(1+\mu|y-x_j|)^{\frac{N}{q+1}+2+\tau}}
&\quad \text{in }\Omega^{11}\\
            \displaystyle C(p)v_\epsilon^{p-1} PV\leq
            \frac{C(p)}{\mu^{t_1(\mathcal{B}_0)}}
\sum_{j=1}^k\frac{\mu^{\frac{N}{q+1}+2}}{(1+\mu|y-x_j|)^{\frac{N}{q+1}+2+\tau}}
            &\quad \text{in }\Omega^{22}
        \end{cases}\\
        &\leq \frac{C(p)}{\mu^{\frac{N}{2(q+1)}+\sigma}}\sum_{j=1}^k\frac{\mu^{\frac{N}{q+1}+2}}{(1+\mu|y-x_j|)^{\frac{N}{q+1}+2+\tau}} \qquad \text{in }\Omega^{11}\cup\Omega^{22}.
    \end{aligned}
\end{equation}

On the other hand, due to $p-1\in (0,1)$, we have the following inequality
\[
ab^{p-1}\le a^{p-1}b , \quad b\ge a > 0.
\]
Thus, we obtain
\begin{equation}
    \begin{aligned}
        |H_1| &\leq
        \begin{cases}
            \displaystyle C(p)v_\epsilon PV^{p-1}\leq
            C(p)v_\epsilon^{p-1} PV
            \leq \frac{C(p)}{\mu^{t_1(\mathcal{B}_0)}}
\sum_{j=1}^k\frac{\mu^{\frac{N}{q+1}+2}}{(1+\mu|y-x_j|)^{\frac{N}{q+1}+2+\tau}}
&\quad \text{in }\Omega^{12}\\
            \displaystyle C(p)v_\epsilon^{p-1} PV
            \leq C(p)v_\epsilon PV^{p-1} \leq
            \frac{C(p)}{\mu^{t_2(\mathcal{B}_0)}}
\sum_{j=1}^k\frac{\mu^{\frac{N}{q+1}+2}}{(1+\mu|y-x_j|)^{\frac{N}{q+1}+2+\tau}}
            &\quad \text{in }\Omega^{21}
        \end{cases}\\
        &\leq \frac{C(p)}{\mu^{\frac{N}{2(q+1)}+\sigma}}\sum_{j=1}^k\frac{\mu^{\frac{N}{q+1}+2}}{(1+\mu|y-x_j|)^{\frac{N}{q+1}+2+\tau}} \qquad \text{in }\Omega^{12}\cup\Omega^{21}.
    \end{aligned}
\end{equation}

As a consequence, we have
$$|H_1|\leq \frac{C}{\mu^{\frac{N}{2(q+1)}+\sigma}}\sum_{j=1}^k\frac{\mu^{\frac{N}{q+1}+2}}{(1+\mu|y-x_j|)^{\frac{N}{q+1}+2+\tau}} \qquad \text{in }\Omega_1\setminus S.$$

\textbf{Estimate $H_2$ in $S$.}

It holds that $U_1\ge c>0$ and $\sum_{j=2}^k  PU_j + \varphi\le C$ in $S$. Then, we have
\[\|PU_*|^{q-1}PU_*+ PU_{1}^q| \le C PU_1^{q-1} \left( u_\epsilon +\sum_{j=2}^k  PU_j\right) \le C U_1^{q-1}\]
and
\[
\sum_{j=2}^k  U^q_j\le C U_1^{q-1} \sum_{j=2}^k  U_j \le C U_1^{q-1}.
\]
Hence,
\[
|l_2|\le CU_1^{q-1} \quad \text{in } S.
\]
Now we determine $t>0$ such that
\begin{align}\label{eq:H2_S}
    U_1^{q-1}(y) \le \frac C{\mu^t}\frac{\mu^{\frac{N}{p+1}+2}}{(1+\mu
|y-x_{1}|)^{\frac{N}{p+1}+2+\tau}}, \quad y \in S.
\end{align}
This is equivalent to
\begin{equation}\label{10-9-3}
(1+\mu|y-x_{1}|)^{\frac{N}{p+1}+2+\tau-(q-1)(p(N-2)-2)}\le
C\mu^{\frac{N}{p+1}+2-\frac{(q-1)N}{q+1}-t}, \quad y \in S.
\end{equation}
If $\frac{N}{p+1}+2+\tau-(q-1)(p(N-2)-2)\le 0$, then the left-hand side of \eqref{10-9-3} is bounded. In this case, we can take
\[
t=\frac{N}{p+1}+2-\frac{(q-1)N}{q+1}=\frac{N}{q+1}>\frac{N}{2(q+1)}.
\]

We deal with the case $\frac{N}{p+1}+2+\tau-(q-1)(p(N-2)-2)> 0$. In $S$, we have $\mu
|y-x_{1}|\le C\mu^{1-\tau}$. So, \eqref{10-9-3} holds if
\[
\mu
^{(1-\tau) [\frac{N}{p+1}+2+\tau-(q-1)(p(N-2)-2)]}\le
C\mu^{\frac{N}{p+1}+2-\frac{(q-1)N}{q+1}-t},
\]
and we can choose
\[
t =\frac{N}{q+1}-(1-\tau) \left[\frac{qN}{q+1}+\tau-(q-1)(p(N-2)-2)\right].
\]
One can check that $t>\frac N{2(q+1)}$ which holds for all $N \ge 7$ and $p \in (1,\frac{N}{N-2})$.

\textbf{Estimate $H_2$ in $\Omega_1 \setminus S$.}

First, combining
$$|y-x_j|>Ck^{-1}\quad\text{in }\Omega_1 \setminus S,\quad\text{for }j=1,\cdots,k$$
and $k\approx \mu^\tau$, we can show that
\begin{equation*}
    \begin{aligned}
        \sum_{i=1}^k \frac{\mu^{\frac{N}{q+1}}}{(1+\mu|y-x_i|)^{p(N-2)-2}}
        &\le C \frac{\mu^{\frac{N}{q+1}}}{(1+\mu|y-x_1|)^{p(N-2)-2}}+C \sum_{i=2}^k \frac{\mu^{\frac{N}{q+1}}}{(\mu|x_1-x_i|)^{p(N-2)-2}}\\
        &\le C\mu^{\frac{N}{q+1}}\cdot (\frac{k}{\mu})^{p(N-2)-2}\le C,
    \end{aligned}
\end{equation*}
and
\begin{equation*}
    \begin{aligned}
        &\quad \frac{1}{\mu^{\frac{pN}{q+1}}} \sum_{i=1}^k \frac{k^{p(N-2)-2}}{(1+k|y-x_i|)^{p(N-3-\theta)-2}}\\
        &\le C\mu^{\frac{N}{q+1}}\cdot (\frac{k}{\mu})^{p(N-2)-2}\left(\frac{1}{(1+k|y-x_1|)^{p(N-3-\theta)-2}} + \sum_{i=2}^k \frac{1}{(k|x_1-x_i|)^{p(N-3-\theta)-2}}\right)
        \le C.
    \end{aligned}
\end{equation*}
Thus, it follows from Lemma \ref{lemma:U} that $PU\le C$ in $\Omega_1\setminus S$, which yields
\begin{align}\label{H-2-0}
   |H_2|&\leq \left\|PU_*|^{q-1}PU_*-|u_\epsilon|^{q-1}u_\epsilon\right|+\sum_{j=1}^k U_j^q\nonumber\\
   &\leq C\left( PU+ U_1^{q-1}\sum_{j=1}^k U_j\right)\leq C PU  \quad \text{in } \Omega_1\setminus S.
\end{align}
Now we determine $t>0$ such that
\begin{equation}
    \begin{aligned}
        PU(y) &\le C \sum_{j=1}^k \frac{\mu^{\frac{N}{q+1}}}{(1+\mu|y-x_j|)^{p(N-2)-2}} + \frac{C}{\mu^{\frac{pN}{q+1}}} \sum_{j=1}^k \frac{k^{p(N-2)-2}}{(1+k|y-x_j|)^{p(N-3-\theta)-2}}\\
&\le \frac{C}{\mu^{\frac{pN}{q+1}}} \sum_{j=1}^k \frac{k^{p(N-2)-2}}{(k|y-x_j|)^{p(N-3-\theta)-2}}\le \frac C{\mu^t}\sum_{j=1}^k\frac{\mu^{\frac{N}{p+1}+2}}{(1+\mu
|y-x_{j}|)^{\frac{N}{p+1}+2+\tau}}, \quad y \in \Omega_1\setminus S.
    \end{aligned}
\end{equation}
It is sufficient to show that
\begin{align}\label{H-2-0-1}
    |y-x_j|^{\frac{N}{p+1}+2+\tau-(p(N-3-\theta)-2)}\leq Ck^{-p(1+\theta)}\mu^{\frac{pN}{q+1}-\tau-t}, \quad y \in \Omega_1\setminus S.
\end{align}

If $\frac{N}{p+1}+2+\tau-(p(N-3)-2)>0$, then \eqref{H-2-0-1} is valid with
$$t=\frac{pN}{q+1}-\tau-p(1+\theta)\tau$$
where $t>\frac{N}{2(q+1)}$ holds for
\begin{equation*}
    \begin{aligned}
        p\in \begin{cases}
            (\frac{11+\sqrt{57}}{16},\frac{7}{5}),\qquad &N=7,\\
           (1,\frac{N}{N-2}) ,\qquad &N\geq 8.
        \end{cases}
    \end{aligned}
\end{equation*}

Note that when $N=7$, it holds that $p_*<\frac{11+\sqrt{57}}{16}$.

If $\frac{N}{p+1}+2+\tau-(p(N-3)-2)\leq0$, then \eqref{H-2-0-1} is valid with
$$t=\frac{pN}{q+1}-\tau-\tau[p(N-2)-2-\frac{N}{p+1}-2-\tau]$$
where $t>\frac{N}{2(q+1)}$ for $N \ge 7$ and $p \in (1,\frac{N}{N-2})$.

\textbf{Estimate $K_1$.}

From \eqref{H-2-0}, we only need to estimate $PU$ in $S$, that is to determine $t>0$ such that
\[
PU(y)\leq C\cdot U_1(y) \le \frac C{\mu^t}\frac{\mu^{\frac{N}{p+1}+2}}{(1+\mu
|y-x_{1}|)^{\frac{N}{p+1}+2+\tau}}, \quad y \in S.
\]
This is equivalent to
\begin{equation}
(1+\mu|y-x_{1}|)^{\frac{N}{p+1}+2+\tau-p(N-2)+2}\le
C\mu^{\frac{N}{p+1}+2-\frac{N}{q+1}-t}, \quad y \in S.
\end{equation}
Then, one can check that
$$t=\frac{(q-1)N}{q+1}-(1-\tau)(\frac{N}{p+1}+2+\tau-p(N-2)+2)>\frac{N}{2(q+1)}$$
holds for $N \ge 7$ and $p \in (1,\frac{N}{N-2})$.

We also note that assumption $q>2$ and \eqref{eq:H2_S} also gives
$$PU(y)\leq C U_1(y)\leq C U_1^{q-1}(y)\le C{\mu^{-\frac{N}{2(q+1)}-\sigma}}\sum_{j=1}^k\frac{\mu^{\frac{N}{p+1}+2}}{(1+\mu
|y-x_{1}|)^{\frac{N}{p+1}+2+\tau}}.$$

\end{proof}

The rest of this section is devoted to proving the following proposition by the contraction mapping theorem.
\begin{proposition}\label{existence2}
    There is a $k_0>0$ such that if $k>k_0$, $\mu \in [\lambda_{00}^{-1}k^\frac{p+1}{p}, \lambda_{00}k^\frac{p+1}{p}]$, \eqref{linear-equation2} has a unique solution $(\phi,\psi)=(\phi_{r,\mu}, \psi_{r,\mu})$ that satisfies
    \begin{align}\label{phi_psi_c_2}
         \|(\phi,\psi)\|_{*} \leq C\mu^{-\frac{N}{2(q+1)}-\sigma}, \;\; |c_l| \leq C\mu^{-n_l-\frac{N}{2(q+1)}-\sigma},
    \end{align}
\end{proposition}

\begin{proof}

   Define
  $$ \begin{array}{ll}
      \bar{E}=\displaystyle\left\{ (\phi,\psi) \in E \cap X\ | \ \ \|(\phi,\psi)\|_{*} \leq \mu^{-\frac{N}{2(q+1)}-\sigma_0}\right\},
   \end{array}$$
   where $\sigma_0 > 0$ is slightly less than $\sigma$, which appears in Lemma \ref{l}.

   We will find a solution of (\ref{linear-equation2}) in $\bar{E}$ which is equivalent to
\begin{align}
    (\phi, \psi) = A(\phi,\psi) : = \textbf{L}_k(N_1(\psi),N_2(\phi)) + \textbf{L}_k(l_1,l_2), \;\; \hbox{for} \;\;  (\phi,\psi) \in \bar{E}.\label{eq:bf_L}
\end{align}

  First, we prove that $A$ maps $ \bar{E}$ into $ \bar{E}$. For any $(\phi,\psi) \in  \bar{E}$, by Lemma \ref{existence1}, Lemma \ref{N} and Lemma \ref{l}, we have
   \begin{align}
      \|A(\phi,\psi)\|_{*}
       &\leq C \|(N_1(\psi),N_2(\phi))\|_{**} + C \|(l_1,l_2)\|_{**} \nonumber \\
      & \leq C \|(\phi,\psi)\|_{*}^{p} + C \|(l_1,l_2)\|_{**}
        \leq \mu^{-\frac{N}{2(q+1)}-\sigma_0}   , \label{A}
     \end{align}
  Here we choose $k$ large enough such that $A(\phi,\psi) \in \bar{E}$.

   Secondly, we prove $A$ is a contraction map. For any $(\omega_1, \omega_2)$ and $(\phi_1,\phi_2) \in \bar{E}$, we have
   \[
      \begin{split}
         \|A(\omega_1,\omega_2) - A(\phi_1,\phi_2)\|_* & = \|\textbf{L}_k(N_1(\omega_2),N_2(\omega_1)) - \textbf{L}_k(N_1(\phi_2),N_2(\phi_1))\|_*\\
         & \leq C\|(N_1(\omega_2) - N_1(\phi_2) ,  N_2(\omega_1) - N_2(\phi_1))\|_{**} .
      \end{split}
   \]

It holds that
   \[
      |N_2^{\prime}(t)| \leq \begin{cases}
          C |t|^{q-1}, \;\;\; \hbox{if}\;\; q \leq 2 ;\\
   C(PU_*)^{q-2}t + C|t|^{q-1}, \;\; \hbox{if}\;\; q > 2.
      \end{cases}
   \]
   We only give the argument for the case $q>2$, since the other case is similar. Similar to the proof of Lemma \ref{N}, we can derive from Lemma \ref{l}, Lemma \ref{B1-1} and Lemma \ref{l1-23-4} that
   \begin{align*}
           |N_2(\omega_1) - N_2(\phi_1)|   &\leq C \left( (PU_*)^{q-2}(| \omega_1 | + | \phi_1 | )  + |\omega_1|^{q-1} + |\phi_1|^{q-1}\right)|\omega_1 - \phi_1| \\
         & \leq C (\|\omega_1\|_{*,1}^{q-1}+\|\phi_1\|_{*,1}^{q-1}) \|\omega_1 - \phi_1\|_{*,1} \left( \sum\limits_{j=1}^{k} \dfrac{\mu^{\frac{N}{q+1}}}{(1+\mu|y-x_j|)^{\frac{N}{q+1}+\tau}}  \right)^{q} \\
         & \quad+ C (\|\omega_1\|_{*,1}+\|\phi_1\|_{*,1}) \|\omega_1 - \phi_1\|_{*,1}\left(  \sum\limits_{j=1}^{k} \dfrac{\mu^{\frac{N}{q+1}}}{(1+\mu|y-x_j|)^{\frac{N}{q+1}+\tau+\theta}} \right)^{q} \\
         & \leq C \mu^{-\sigma} \|\omega_1 - \phi_1\|_{*,1} \sum\limits_{j=1}^m \dfrac{\mu^{\frac{N}{p+1}+2}}{(1+\mu|y-x_j|)^{\frac{N}{p+1}+2+\tau}}.
     \end{align*}
   where $\sigma>0$ is a small constant.

  Hence, we have
   \[
      \|N_2(\omega_1) - N_2(\phi_1)\|_{**,2} \leq C \mu^{-\sigma}\|\omega_1-\phi_1\|_{*,1}.
   \]
   Similarly, we also have
   \[
      \|N_1(\omega_2) - N_1(\phi_2)\|_{**,1} \leq \mu^{-\sigma}\|\omega_2-\phi_2\|_{*,2}.
   \]
   Thus, we obtain
   \[
   \begin{split}
      \|A(\omega_1,\omega_2) - A(\phi_1,\phi_2)\|_* & \leq C\|(N_1(\omega_2) - N_1(\phi_2) ,  N_2(\omega_1) - N_2(\phi_1))\|_{**} \\
      & \leq C\mu^{-\sigma} \|(\omega_1 - \phi_1, \omega_2-\phi_2)\|_{*} \leq \dfrac{1}{2} \|(\omega_1, \omega_2) - (\phi_1, \phi_2)\|_{*}.
   \end{split}
   \]
   The last inequality holds if we choose $k$ large enough.

   As a result, there is a $k_0>0$ such that for any $k \geq k_0$, $A$ is a contraction map from $\bar{E}$ to $\bar{E}$ and it follows that there is a unique solution $(\phi,\psi) \in \bar{E}$ of equation (\ref{linear-equation2}). Moreover, by (\ref{A}), we have
   \[
   \|(\phi,\psi)\|_{*} = \|A(\phi,\psi)\|_{*} \leq  C\mu^{-\frac{N}{2(q+1)}-\sigma}.
   \]
   Finally, the estimate of $c_l$ follows from Lemma \ref{existence1}.
\end{proof}

\section{Existence of the solution}

Given $k \in \N$ large enough, we set a functional
\begin{equation}\label{eq:redene}
K_\ep(r,\lambda) := I_\ep(PU_*[r,\mu]+\phi[r,\mu], PV_*[r,\mu]+\psi[r,\mu]),
\end{equation}
where $(PU_*,PV_*)$ is defined by \eqref{projection}--\eqref{def:approx}, $(\phi,\psi)$ is built in Proposition~\ref{existence2}, and $\mu = \lambda k^{\frac{p+1}{p}}$.

By the standard reduction argument, if we prove the existence of a critical point for $K_\ep$, then we will obtain a solution for system \eqref{mainsystem} of the form : $$(u,v) = (PU_*+\phi, PV_*+\psi).$$

\begin{proof}[Proof of Theorem~\ref{Thm2}]
By exploiting \eqref{def:I_ep}, \eqref{N}, \eqref{l}, \eqref{eq:bf_L} and $(\phi,\psi) \in  E$, we see
\begin{align}
K_\ep(r,\lambda) - I_\ep( PU_*, PV_* )
&= O\left(\int_{\Omega} |\psi|^{p+1} + \int_{\Omega} (|l_1|+|N_1(\psi)|)|\psi|\right) \label{eq:diff}
\\
&\ + O\left(\int_{\Omega} |PU_*|^{q-2} \phi^3 + \int_{\Omega} |\phi|^{q+1} + \int_{\Omega} (|l_2|+|N_2(\phi)|)|\phi|\right). \nonumber
\end{align}
Furthermore, arguing as in the proof of \cite[Proposition 3.1]{WY}, we obtain
\begin{equation}\label{eq:diff2}
\int_{\Omega} |\psi|^{p+1} + \int_{\Omega} |PU_*|^{q-2} \phi^3 + \int_{\Omega} |\phi|^{q+1} \le Ck^{1+\theta} \|(\phi, \psi)\|_*^{\min\{p+1,\, 3,\, q+1\}},
\end{equation}
for some $\theta > 0$ arbitrarily small, and
\begin{equation}\label{eq:diff3}\begin{array}{ll}
&\displaystyle \int_{\Omega} (|l_1|+|N_1(\psi)|)|\psi| + \int_{\Omega} (|l_2|+|N_2(\phi)|)|\phi| \\
&\le \displaystyle Ck(\|l\|_{**}+\|N(\phi,\psi)\|_{**})\|(\phi, \psi)\|_*.\end{array}
\end{equation}

From Proposition~\ref{existence2}, Lemma~\ref{N}, Lemmas~\ref{l}, we obtain the expansion
\begin{align*}
K_\ep(r,\lambda) &= I_\ep( PU_*, PV_* ) + O\left(k\|(\phi, \psi)\|_*^2 + k(\|l\|_{**}+\|N(\phi,\psi)\|_{**})\|(\phi, \psi)\|_*\right) \\
&= I_\ep( PU_*, PV_* ) + O\left(k\mu^{-\frac N{q+1}-\sigma}\right)
\end{align*}
for some $\sigma > 0$ small. Hence, by Proposition~\ref{prop:Iexpan}, we have that
\begin{align}\label{61-24-2}
K_\ep(r,\lambda) &= I_\epsilon(u_\epsilon,v_\epsilon)+ kA
+ k\left[-\frac{(B_1+B_2 \widetilde{H}(\tilde x_{1})) k^{p(N-2)-2}}{r^{p(N-2)-2} \mu^{p(N-2)-2}} + \frac{B_4 u_\epsilon(r)}{\mu^{\frac{N}{q+1}}}+O\left({\mu^{-\frac{N}{q+1}-\sigma}}\right)\right]\nonumber\\
&\hspace{8em}-\chi_{\{2(p(N-2)-2)>N\}}\cdot k\left[\frac{\epsilon \beta_2 B_3}{\mu^{N-\frac{2N}{q+1}}}+O\left({\mu^{\frac{2N}{q+1}-N-\sigma}}\right)\right].
\end{align}

Define
\begin{align}\label{bf_K}
    \textbf{K}_\ep(r,\lambda)=-\frac{(B_1+B_2 \widetilde{H}(\tilde x_{1})) }{(r\lambda)^{p(N-2)-2} } + \frac{B_4 u_\epsilon(r)}{\lambda^{\frac{N}{q+1}}}
\end{align}

To view $K_\epsilon$ as a perturbation of $\textbf{K}_\ep$, we impose that
\begin{align}
    \mu^{-\frac{N}{q+1}}> \mu^{-N+\frac{2N}{q+1}}.\label{restriction-1}
\end{align}
Equivalently, $q>2$. By \eqref{c-hyperbola}, this is the same as $p<\frac{N+6}{2(N-3)}$.

Therefore, one can check that \( \textbf{K}_\ep \) has a maximum point at the point \((r_0,\lambda_0)\) where \( r_0 \) maximizes the function \( r \to r^{\frac{N}{q+1}} u_\ep(r) \) and
\begin{equation}\label{lambda_0}
\lambda_0 := \left( \frac{(p+1)\left(B_1+B_2 \widetilde{H}(\tilde x_{1})\right)}{B_4 u_\ep(r_0) r_0^{p(N-2)-2}} \right)^{\frac{q+1}{pN}},
\end{equation}
which is stable under \( C^0 \)-perturbation. Therefore, the reduced energy \( K_\ep \) has a critical point \((\lambda_k, r_k)\), which produces the solution $(PU_* + \phi ,PV_*+\psi)$ to the problem \eqref{mainsystem}.

\end{proof}

\appendix

\section{Properties of the bubbles and the projection }

In this section, we recall some essential properties of the ground state $(U_{0,1}, V_{0,1})$ of \eqref{lane-embden} and its projection used in the previous sections.

First, for reader's convenience, we list the asymptotic behavior and the non-degeneracy of the ground state $(U_{0,1}, V_{0,1})$ of \eqref{lane-embden}.

\begin{lemma}[\cite{HV}] \label{L1}
 Assume that $p\leq \frac{N+2}{N-2} \leq q.$ There exist positive constants $a=a_{N,p}$ and $b=b_{N,p}$ depending only on $N$ and $p$ such that
$$ \lim\limits_{r\to \infty} r^{N-2} V_{0,1}(r) =b ;$$
while
\begin{equation*}
\begin{cases}
    \lim\limits_{r\to\infty} r^{(N-2)p-2}U_{0,1}(r) =a,  \;\; &\hbox{if } p<\frac{N}{N-2},\\
    \lim\limits_{r\to\infty} \frac{r^{N-2}}{\ln r}U_{0,1}(r) =a,  \;\; &\hbox{if } p=\frac{N}{N-2},\\
    \lim\limits_{r\to\infty} r^{N-2}U_{0,1}(r) =a,  \;\; &\hbox{if } p>\frac{N}{N-2}.
\end{cases}
\end{equation*}
Furthermore, in the last case, we have $b^p=a( (N-2)p-2  )(N-(N-2)p).$
\end{lemma}

\begin{lemma}[\cite{KM}]\label{L3}
There exists a constant $C > 0$ depending only on $N$ and $p$ such that
\begin{equation}\label{V10est}
    \left|V_{0,1}(r) - \frac{b_{N,p}}{r^{N-2}}\right| \le \frac{C}{r^N}.
\end{equation}
Besides,
\begin{equation}\label{U10est}
\begin{cases}
    \displaystyle \left|U_{0,1}(r) - \frac{a_{N,p}}{r^{N-2}}\right| \le \frac{C}{r^{N-2+\kappa_0}} &\text{if } p \in (\frac{N}{N-2}, \frac{N+2}{N-2}], \\
    \displaystyle \left|U_{0,1}(r) - \frac{a_{N,p} \log r}{r^{N-2}}\right| \le \frac{C}{r^{N-2}} &\text{if } p = \frac{N}{N-2}, \\
    \displaystyle \left|U_{0,1}(r) - \frac{a_{N,p}}{r^{p(N-2)-2}}\right| \le \frac{C}{r^{p(N-2)-2+\kappa_1}} &\text{if } p \in (\frac{2}{N-2}, \frac{N}{N-2}),
\end{cases}
\end{equation}
where $\kappa_0 := p(N-2)-N > 0$ and $\kappa_1$ is any number in $(0, \min\{N-p(N-2),2(p+1)\})$.
\end{lemma}

\begin{lemma}[\cite{FKP21}]\label{L2}
Set
   $$(\Psi_{0,1}^0,\Phi_{0,1}^0) = \left(  y \cdot \nabla U_{0,1} +\frac{NU_{0,1}}{q+1},\; y \cdot \nabla V_{0,1}+\frac{NV_{0,1}}{p+1}  \right)$$
and
   $$ (\Psi_{0,1}^l, \Phi_{0,1}^l) = (\partial_l U_{0,1}, \partial_l V_{0,1} ),\;\; \hbox{for }\ \ l=1,\cdots,N. $$
Then the space of solutions to the linear system
   \begin{equation}\label{8}
   \begin{cases}
   -\Delta \Psi =pV_{0,1}^{p-1} \Phi,\;\;\; \hbox{in } \mathbb R^N,\\
   -\Delta \Phi =qU_{0,1}^{q-1} \Psi,\;\;\; \hbox{in } \mathbb R^N,\\
   (\Psi,\Phi)\in \dot{W}^{2,\frac{p+1}{p}}(\mathbb R^N) \times \dot{W}^{2,\frac{q+1}{q}}(\mathbb R^N),
   \end{cases}
   \end{equation}
 is spanned by
   $$ \left\{  (\Psi_{0,1}^0,\Phi_{0,1}^0), (\Psi_{0,1}^1,\Phi_{0,1}^1)  ,\cdots, (\Psi_{0,1}^N,\Phi_{0,1}^N)    \right\} .$$
   \end{lemma}

   Next, we state some properties of the projection of the bubbles used in the previous sections.

   For each $j=1,\cdots,k$, we recall that $$\big(U_j(y),V_j(y)\big)=\left(\mu^{\frac{N}{q+1}} U_{0,1}\big(\mu(y - x_j)\big), \ \mu^{\frac{N}{p+1}} V_{0,1}\big(\mu(y - x_j)\big)\right).$$
   Moreover, the pair $(PU_j,PV_j)$ is the unique solution of \eqref{projection}.

   We also set the following
\begin{equation*}
\begin{aligned}
\big(\Psi_{j,0}(y),\Phi_{j,0}(y)\big)
&=\left(
\mu^{\frac{N}{q+1}-1}
\Psi^{0}_{0,1}\big(\mu(y-x_j)\big),
\mu^{\frac{N}{p+1}-1}
\Phi^{0}_{0,1}\big(\mu(y-x_j)\big)
\right),
\end{aligned}
\end{equation*}
and
\begin{equation*}
\begin{aligned}
\big(\Psi_{j,l}(y),\Phi_{j,l}(y)\big)
&=\left(
\mu^{\frac{N}{q+1}+1}
\Psi^{l}_{0,1}\big(\mu(y-x_j)\big),
\mu^{\frac{N}{p+1}+1}
\Phi^{l}_{0,1}\big(\mu(y-x_j)\big)
\right),
\end{aligned}
\end{equation*}
for $l=1,\ldots,N$. Let the pair $(P\Psi_{j,l},P\Phi_{j,l})$ be the unique smooth solution of the system
\begin{equation*}
\begin{aligned}
\begin{cases}
-\Delta P\Psi_{j,l}=pV_j^{p-1}\Phi_{j,l}, & \text{in } \Omega,\\
-\Delta P\Phi_{j,l}=qU_j^{q-1}\Psi_{j,l}, & \text{in } \Omega,\\
P\Psi_{j,l}=P\Phi_{j,l}=0, & \text{on } \partial\Omega,
\end{cases}
\end{aligned}
\end{equation*}
for $l=0,1,\ldots,N$.

Then, we have

\begin{lemma}[Lemma 2.9 in \cite{KP21}]
Let $H : \Omega \times \Omega \to \mathbb{R}$ be a smooth function such that
\begin{equation*}
\begin{cases}
-\Delta_x  H(x,y) = 0, & \text{for } x \in \Omega, \\[0.3em]
 H(x,y) = \dfrac{\gamma_N}{|x-y|^{N-2}}, & \text{for } x \in \partial\Omega,
\end{cases}
\end{equation*}
and $\widehat H : \Omega \times \Omega \to \mathbb{R}$ be a smooth function such that

\begin{equation*}
\begin{cases}
-\Delta_x \widehat H(x,y) = 0, & \text{for } x \in \Omega, \\[0.3em]
\widehat H(x,y) = \dfrac{1}{|x-y|^{(N-2)p-2}}, & \text{for } x \in \partial\Omega,
\end{cases}
\end{equation*}

given any $y \in \Omega$. Here, $\gamma_N := (N-2)^{-1}|\mathbb{S}^{N-1}|^{-1} > 0$.
If $i = 1, \cdots, k$, then we have

\begin{equation}\label{pui}
\begin{aligned}
P U_i(x)
&= U_i(x)
- a_{N,p}{\mu^{-\frac{N p}{q+1}}}\,\widehat H(x,\xi_i)
+ o\!\left({\mu^{-\frac{N p}{q+1}}}\right).
\end{aligned}
\end{equation}

and

\begin{equation}\label{pvi}
\begin{aligned}
P V_i(x)
&= V_i(x)
- \left(\frac{b_{N,p}}{\gamma_N}\right)
{\mu^{-\frac{N}{q+1}}}\,H(x,\xi_i)
+ o\!\left({\mu^{-\frac{N}{q+1}}}\right).
\end{aligned}
\end{equation}

for $x \in \Omega$.

\end{lemma}

\begin{lemma}[Lemma 2.10 in \cite{KP21}]
    \begin{equation*}
\begin{aligned}
P\Psi_{j,l}(x)
=
\begin{cases}
\Psi_{j,l}(x)
+\dfrac{Np}{q+1}a_{N,p}\mu^{-\frac{Np}{q+1}-1}
\widehat H(x,x_j)
+o\left(\mu^{-\frac{Np}{q+1}-1}\right),
& \text{for } l=0,\\[0.6em]
\Psi_{j,l}(x)
+a_{N,p}\mu^{-\frac{Np}{q+1}}
\partial_{x_j,l}\widehat H(x,x_j)
+o\left(\mu^{-\frac{Np}{q+1}}\right),
& \text{for } l=1,\ldots,N,
\end{cases}
\end{aligned}
\end{equation*}
and
\begin{equation*}
\begin{aligned}
P\Phi_{j,l}(x)
=
\begin{cases}
\Phi_{j,l}(x)
+\left(\dfrac{N}{q+1}\dfrac{b_{N,p}}{\gamma_N}\right)
\mu^{-\frac{N}{q+1}-1}H(x,x_j)
+o\left(\mu^{-\frac{N}{q+1}-1}\right),
& \text{for } l=0,\\[0.6em]
\Phi_{j,l}(x)
+\left(\dfrac{b_{N,p}}{\gamma_N}\right)
\mu^{-\frac{N}{q+1}}
\partial_{x_j,l}H(x,x_j)
+o\left(\mu^{-\frac{N}{q+1}}\right),
& \text{for } l=1,\ldots,N,
\end{cases}
\end{aligned}
\end{equation*}
for $x\in\Omega$. Here, $\partial_{x_j,l}\widehat H(x,x_j)$ and
$\partial_{x_j,l}H(x,x_j)$ stand for the $l$-th components of
$\nabla_{x_j}\widehat H(x,x_j)$ and $\nabla_{x_j}H(x,x_j)$, respectively.
\end{lemma}

\section{Estimate of the approximate solution (PU,PV)}\label{sec:est_PU}
In this section, we estimate $PU$, the unique solution of \eqref{eq:PU}. 

Denote
\begin{equation}
    \varphi(y)=PU(y)-\sum_{j=1}^kPU_j(y),
\end{equation}
we first derive the expansion of $\varphi$.

Let $G_{R}$ be the Green’s function of the Laplacian $-\Delta$ in $B_{R}(0)$ with the Dirichlet boundary condition.
Let also $H = H_R : B_{R}(0) \times B_{R}(0) \to \mathbb{R}$ be its regular part,
which solves
\begin{equation*}
\begin{cases}
-\Delta_x H_R(x,y) = 0, & \text{for } x \in B_{R}(0), \\[0.3em]
H_R(x,y) = \dfrac{\gamma_N}{|x-y|^{N-2}}, & \text{for } x \in \partial B_{R}(0),
\end{cases}
\end{equation*}
for each $y \in B_{R}(0)$ with
$\gamma_N := (N-2)^{-1}|\mathbb{S}^{N-1}|^{-1} > 0$.
Then we have
\[
H_R(y,x)=H_R(x,y)=\frac{R^{N-2}}{|x|^{N-2}}\frac{\gamma_N}{|y-x^{R*}|^{N-2}},
\]
where $x^{R*}=\frac{R^2x}{|x|^2}$, and
\begin{equation}\label{green}
\begin{aligned}
0 < G_R(x,y) = G_R(y,x)
= \frac{\gamma_N}{|x-y|^{N-2}} - H_R(x,y)
< \frac{\gamma_N}{|x-y|^{N-2}} .
\end{aligned}
\end{equation}
for $(x,y) \in B_{R}(0) \times B_{R}(0)$ such that $x \neq y$.

\begin{lemma}%[Lemma 3.2,\cite{GKPY}]
\label{nnl2-19-1}
Assume that $N \ge 7$, $p\in(1,\frac{N}{N-2})$. For $k \in \N$ large enough, there is a positive constant $B_0>0$, such that
\begin{equation}\label{nn1-19-1}
    \varphi(y)=\frac {k^{p(N-2)-2}}{r^{p(N-2)-2}\mu^{\frac {Np}{q+1}}}\widetilde{H}(r^{-1} k y)+O\Bigl(   \frac{k^{p(N-2)-2}}{\mu^{\frac{pN}{q+1}+\sigma}}\Bigr), \quad y \in \Omega
\end{equation}
uniformly in $\mathcal{P}$, where $k\simeq\mu^ {\tau}$, $\tau=\frac{p}{p+1}$, $\sigma>0$ is a sufficiently small number and $\widetilde{H}$ is defined by
\begin{equation}\label{H}
\widetilde{H}(y) := \int_{(r^{-1}k)\Omega} G_{r^{-1}k}(y,z)\left[\left(\sum_{j=1}^k \frac{b_{N,p}}{|z-\tilde{x}_j|^{N-2}}\right)^p - \sum_{j=1}^k \frac{b_{N,p}^p}{|z-\tilde{x}_j|^{p(N-2)}}\right] dz,
\end{equation}
for $y \in B_{r^{-1}k}(0)$, $\tilde x_j= (k \cos\frac{2(j-1)\pi }k, k\sin\frac{2(j-1)\pi }k, 0)$ and $b_{N,p} > 0$ is the constant in \eqref{pvi}.

\end{lemma}

\begin{proof}

Using the representation and the condition $p(N-2)<N$,
one can derive the following estimate for  $\varphi$:
\begin{align}
    \varphi(y)&=\sum_{j=1}^k\int_{B_{\mu^{-\kappa}}(x_j)} G(y,z) \left[PV^p- \sum_{j=1}^k V_j^p\right](z)\dz\nonumber \\
    &\qquad+\quad \int_{\Omega\setminus \cup_{j=1}^k B_{\mu^{-\kappa}}(x_j)} G(y,z) \left[PV^p- \sum_{j=1}^k V_j^p\right](z)\dz\nonumber\\
    &:=\sum_{j=1}^k\bar{I}_{1j}+ \bar{I}_{2}.\nonumber
\end{align}
where $\kappa \in (\tau,1)$ is a fixed number.

Let us first estimate $\bar{I}_{1j}$ for $j=1,\cdots, k$. By symmetry, it suffices to consider the case $y\in \Omega_1$. From \eqref{est:01} and \eqref{pvi}, we can obtain that
\begin{equation*}
    \begin{aligned}
        |\bar{I}_{11}|&\le C \int_{B_{\mu^{-\kappa}}(x_1)} \frac{1}{|y-z|^{N-2}} \left[ V_1^{p-1}(z)\sum_{j=2}^k V_j(z)+\sum_{j=2}^k V_j^p(z)\right]\dz\\
        &\le C\frac{k^{N-2}}{\mu^{\frac{pN}{q+1}}} \int_{B_{\mu^{-\kappa}}(x_1)} \frac{1}{|y-z|^{N-2}} \frac{1}{|z-x_1|^{(p-1)(N-2)}} \dz\\
        &=O(\frac{k^{N-2}}{\mu^{\frac{pN}{q+1}}}\cdot(\mu^{-\kappa})^{N-{p(N-2)}})
        =O(\frac{k^{p(N-2)-2}}{\mu^{\frac{pN}{q+1}+\sigma}}),
    \end{aligned}
\end{equation*}
where the last step is due to $\kappa>\tau$.

  On the other hand,
\begin{equation*}
    \begin{aligned}
        |\sum_{j=2}^k\bar{I}_{1j}|&\le C \sum_{j=2}^k\int_{B_{\mu^{-\kappa}}(x_j)} \frac{1}{|y-z|^{N-2}} \left[ V_j^{p-1}(z)\sum_{i\neq j} V_i(z)\right]\dz\\
        &\leq C \frac{1}{\mu^{\frac{pN}{q+1}}}\sum_{j=2}^k\int_{B_{\mu^{-\kappa}}(x_j)}\frac{1}{|x_1-x_j|^{N-2}} \cdot\frac{1}{|z-x_j|^{(p-1)(N-2)}}\cdot
        \sum_{i\neq j}\frac{1}{|x_i-x_j|^{N-2}}\dz\\
        &\leq C \frac{k^{N-2}}{\mu^{\frac{pN}{q+1}}}\cdot k^{N-2}\cdot (\mu^{-\kappa})^{N-{(p-1)(N-2)}}
        =O(\frac{k^{p(N-2)-2}}{\mu^{\frac{pN}{q+1}+\sigma}}).
    \end{aligned}
\end{equation*}
where the second inequality holds since $y\in \Omega_1$ and $z\in B_{\mu^{-\kappa}}(x_j)$. In addition, we use the condition $\kappa>\tau$ in the last step.

We next evaluate $\bar{I}_2$. According to \eqref{green}, \eqref{V10est} and \eqref{pvi}, we estimate
\begin{equation*}
\begin{aligned}
&\bar{I}_2
= \mu^{-\frac{Np}{q+1}}
\int_{\Omega\setminus \bigcup_{j=1}^k B(x_j,\mu^{-\kappa})}
G(y,z)
\Bigg[
\left(\sum_{i=1}^k \frac{b_{N,p}}{|z-x_i|^{N-2}}\right)^p
- \sum_{i=1}^k
\frac{b_{N,p}^p}{|z-x_i|^{p(N-2)}} \\
&\qquad\qquad
+ O\!\left(
 \Big(\sum_{i=1}^k\frac{1}{|z-x_i|^{N-2}}\Big)^{p-1}
 \Big( \sum_{i=1}^k H(z,x_i)+\sum_{i=1}^k \frac{1}{\mu |z-x_i|^{N-1}}\Big)
\right)\Bigg]
\,dz \\
&=\mu^{-\frac{Np}{q+1}}k^{p(N-2)-2}
\left[\widetilde{H}(r^{-1}ky)+o(1)+ \mathcal{R}_1\right]
\end{aligned}
\end{equation*}
where
\begin{align}
   \mathcal{R}_1:&= O\Big(\int_{r^{-1}k\left(\Omega\setminus \bigcup_{j=1}^k B(x_j,\mu^{-\kappa})\right)}\frac{1}{|z-r^{-1}ky|^{N-2}}\times\nonumber\\
&\hspace{5em}\big(\sum_{i=1}^k\frac{1}{|z-\tilde{x}_i|^{N-2}}\big)^{p-1}\big( \sum_{i=1}^k H_{rk^{-1}}(z,\tilde{x}_i)+\sum_{i=1}^k \frac{1}{\mu |z-\tilde{x}_i|^{N-1}}\big)\dz\Big).
\end{align}

We can compute
\begin{equation}\label{R1_11}
    \begin{aligned}
        \sum_{i=1}^k H_{rk^{-1}}(z,\tilde{x}_i)&\leq C\sum_{i=1}^k\frac{1}{|z-r^{-2}\tilde{x}_i|^{N-2}}
        \leq  \frac{Ck}{\text{dist}(r^{-2}\tilde{x}_i,B_{r^{-1}k}(0))^{N-2}}\\
        &\leq Ck^{3-N},\qquad \text{for any }z\in B_{r^{-1}k}(0)=r^{-1}k\cdot \Omega.
    \end{aligned}
\end{equation}
Let $\Gamma_r=\{(r\cos \theta,r\sin \theta,\textbf{0}):\theta\in [0,2\pi),\textbf{0}\in \R^{N-2}\}\subset \R^2\times \R^{N-2}$, then we can derive from \eqref{tip_1} and the fact $p(N-2)<2$ that
\begin{equation}\label{R1_12}
    \begin{aligned}
        &\quad\int_{r^{-1}k\left(\Omega\setminus \bigcup_{j=1}^k B(x_j,\mu^{-\kappa})\right)}\frac{1}{|z-r^{-1}ky|^{N-2}}\big(\sum_{i=1}^k\frac{1}{|z-\tilde{x}_i|^{N-2}}\big)^{p-1}\\
        &\leq \sum_{j=1}^k \int_{r^{-1}k\cdot\Omega_j}\frac{1}{|z-r^{-1}ky|^{N-2}}\left(\frac{1}{|z-\tilde{x}_j|^{(p-1)(N-2)}}+\frac{1}{|z-\tilde{x}_j|^{(p-1)(N-3-\sigma)}}\right)dz\\
        &\leq C\sum_{j=1}^k\int_{\Omega_j}\frac{1}{|z-y|^{N-2}}\left(\frac{k^{2-(p-1)(N-2)}}{\text{dist}(z,\Gamma_r)^{(p-1)(N-2)}}+ \frac{k^{2-(p-1)(N-3-\sigma)}}{\text{dist}(z,\Gamma_r)^{(p-1)(N-3-\sigma)}}\right)\dz\\
        &\leq C(k^{2-(p-1)(N-2)}+k^{2-(p-1)(N-3-\sigma)})=O(k^{2-(p-1)(N-2)}).
    \end{aligned}
\end{equation}

Thus, combining \eqref{R1_11}, \eqref{R1_12} and a similar computation in Lemma~\ref{lemma:D6} gives
$$\mathcal{R}_1=o(1).$$
This completes the proof.

\end{proof}

\begin{lemma}
Assume that $N \ge 6$ and $p \in (1,\frac{N}{N-2})$. Then the function $\widetilde{H}$ is well-defined in $B_{r^{-1}k}(0)$. Furthermore, the following properties hold: Let $\delta,\, \theta > 0$ be sufficiently small numbers.
\begin{itemize}
\item [(a)] For $y\in \cup_{j=1}^k B_\delta(\tilde x_j)$, it holds that
    \[
    0<\widetilde{H}(y)\le C
    \]
for some constant $C > 0$ depending only on $N$ and $p$.
\item [(b)] For $y\in B_{r^{-1}k}(0)\setminus \cup_{j=1}^k B_\delta(\tilde x_j)$, it  holds that
\[|\widetilde{H}(y)|\le C\sum_{j=1}^k \left[\frac 1{|y-\tilde x_j|^{p(N-2)-3-\theta}} + \frac 1{|y-\tilde x_j|^{p(N-3-\theta)-2}}\right].\]
\end{itemize}
\end{lemma}

\begin{proof}It is easy to check that property (a) holds. We only need to show property (b).

We can derive that $\big|\widetilde{H}(y)\big|$ is bounded by
\begin{equation}
    \begin{aligned}
     \sum_{j=1}^k\int_{\Omega_j}
        \frac{C}{|y-z|^{N-2}}\left[\frac{1}{|z-\tilde{x}_1|^{(p-1)(N-2)}}\sum_{j=2}^k \frac{1}{|z-\tilde{x}_j|^{N-2}} +\left(\sum_{j=2}^k \frac{1}{|z-\tilde{x}_j|^{N-2}}\right)^p\right]\dz
    \end{aligned}
\end{equation}

From Lemma~\ref{B2} and \eqref{tip_1}, let $\theta$ be some small number, we have for any $z\in \Omega_1$,
\begin{equation*}
    \begin{aligned}
        \frac{1}{|z-\tilde{x}_1|^{(p-1)(N-2)}}\sum_{j=2}^k \frac{1}{|z-\tilde{x}_j|^{N-2}}
        \leq \frac{C}{|z-\tilde{x}_1|^{p(N-2)-1-\theta}}.
    \end{aligned}
\end{equation*}
and
\begin{align}
    \sum_{j=2}^k\frac{1}{|z-\tilde{x}_j|^{N-2}}
    \leq \frac{C}{|z-\tilde{x}_1|^{N-3-\theta}}.
\end{align}
Then from Lemma \ref{B3}, we have
\begin{equation}\label{equ:411}
    \begin{aligned}
        &\quad\int_{\Omega_1}
        \frac{C}{|y-z|^{N-2}}\left[\frac{1}{|z-\tilde{x}_1|^{(p-1)(N-2)}}\sum_{j=2}^k \frac{1}{|z-\tilde{x}_j|^{N-2}} +\left(\sum_{j=2}^k \frac{1}{|z-\tilde{x}_j|^{N-2}}\right)^p\right]\dz\\
        &\leq \int_{\Omega_1} \frac{C}{|y-z|^{N-2}}\left( \frac{1}{|z-\tilde{x}_1|^{p(N-2)-1-\theta}}+\frac{1}{|z-\tilde{x}_1|^{p(N-3-\theta)}}\right)\dz
        \leq \frac{C}{|y-\tilde{x}_1|^{p(N-3-\theta)-2}}.
    \end{aligned}
\end{equation}

By symmetry, we obtain property (b).

\end{proof}

Next, we give a pointwise upper bound for $PU$.

\begin{lemma}\label{lemma:U}
Assume that $N \ge 7$ and $p \in (1,\frac{N}{N-2})$. Then it holds that
\[
PU(y) \le C \sum_{i=1}^k \frac{\mu^{\frac{N}{q+1}}}{(1+\mu|y-x_i|)^{p(N-2)-2}} + \frac{C}{\mu^{\frac{pN}{q+1}}} \sum_{i=1}^k \frac{k^{p(N-2)-2}}{(1+k|y-x_i|)^{p(N-3-\theta)-2}}
\]
for $y \in \Omega$, where $\theta \in (0,1)$ is small.
\end{lemma}
\begin{proof}
The representation formula for $PU$ yields
\begin{align*}
0 < PU(y) &= \int_{\Omega} G(y,z)\bigg(\sum_{j=1}^k PV_j\bigg)^p(z)\, dz \\
&\leq C\sum_{i=1}^k \left[\int_{\Omega_i} \frac{1}{|y-z|^{N-2}} \bigg(\sum_{j=1}V_j\bigg)^p(z)\, dz\right], \quad y \in \Omega.
\end{align*}

For $i=1,\cdots, k$, we have
\[
\int_{\Omega_i} \frac{1}{|y-z|^{N-2}} V_i^p(z)\, dz \le C
U_i(y) \le \frac{C\mu^{\frac{N}{q+1}}}{(1+\mu|y-x_i|)^{p(N-2)-2}}, \quad y \in \Omega.
\]

On the other hand, there exists $c>0$ such that for $j\ne i$,
\[
|z-k x_j|\ge c (1+ |z-k x_i|),\quad z\in\Omega_i.
\]
Thus, similar to \eqref{equ:411} we find
\begin{align*}
&\quad\int_{\Omega_i} \frac{1}{|y-z|^{N-2}} \bigg(\sum_{j \ne i} V_j\bigg)^p(z)\, dz \\
&\le \frac{Ck^{p(N-2)-2}}{\mu^{\frac{pN}{q+1}}} \int_{k\Omega_i} \frac{1}{|ky-z|^{N-2}} \left(\sum_{j \ne i} \frac{1}{|z-kx_j|^{N-2}}\right)^p dz \\
&\le \frac{Ck^{p(N-2)-2}}{\mu^{\frac{pN}{q+1}}} \int_{k\Omega_i } \frac{1}{|ky-z|^{N-2}} \frac{dz}{(1+|z-kx_i|)^{p(N-3-\theta)}} \\
&\le \frac{C}{\mu^{\frac{pN}{q+1}}}  \frac{k^{p(N-2)-2}}{(1+k|y-x_i|)^{p(N-3-\theta)-2}}.
\end{align*}
for $q$ satisfying \eqref{c-hyperbola} and $\theta \in (0,1)$ small. Thus the result follows.
\end{proof}

Although Lemmas \ref{nnl2-19-1} and \ref{lemma:U} give precise estimates for $PU$, in some cases we use the following simpler estimate.
\begin{lemma}\label{l1-23-4}
Suppose that $N \geq 7$, $p \in (1,\frac{N}{N-2})$ and \eqref{c-hyperbola} hold. For any $y \in \Omega$, we have
\begin{equation}\label{10-23-4}
 PU(y)\le C\sum_{j=1}^k \frac{ \mu^{\frac{N}{q+1}}  }{ (1+\mu |y-x_j|)^{  \frac{N}{q+1}+\tau +\theta} }.
\end{equation}
\end{lemma}
\begin{proof}
By symmetry, it suffices to verify \eqref{10-23-4} for $y \in \Omega_1$.

\medskip
First, if $y \in S$, then $0\le \varphi(y) \le C$, and
\begin{equation}\label{11-23-4}
PU_j(y)\le \frac{C \mu^{\frac{N}{q+1}}  }{ (1+\mu |y-x_j|)^{  \frac{N}{q+1}+\tau } },\qquad j=1,2,\cdots , k.
\end{equation}
since $p(N-2)-2>\frac{N}{q+1}+\tau$ holds for $N \geq 7$, $p \in (1,\frac{N}{N-2})$. Therefore, \eqref{10-23-4} holds.

\medskip
Suppose that $y\in \Omega_1\setminus S$, from Lemma~\ref{lemma:U} and \eqref{11-23-4}, we just need to  prove that
\begin{equation}\label{12-23-4}
 \frac{1}{\mu^{\frac{pN}{q+1}}} \frac{k^{p(N-2)-2}}{(k|y-x_j|)^{p(N-3-\theta)-2}}\le
 \frac{C \mu^{\frac{N}{q+1}}  }{ (\mu |y-x_j|)^{  \frac{N}{q+1}+\tau } },\quad y\in \Omega_1\setminus S.
\end{equation}
Thus, \eqref{12-23-4} holds if
\begin{equation}\label{13-23-4}
|y-x_j|^{\frac{N}{q+1}+\tau-p(N-3-\theta)+2}\le C\mu^{\frac{pN}{q+1}-\tau(1+p(1+\theta))},\quad y\in \Omega_1\setminus S.
\end{equation}
We also note
\[ p(N-3)-2>\frac N{q+1}+\tau\]
for all $N \ge 7$ and $p \in (1,\frac N{N-2})$.
On the other hand, since $p(N-2)-2 = \frac{(p+1)N}{q+1}$, we have
\[\frac{pN}{q+1}-\tau -\tau\left[p(N-2)-2- \frac{N}{q+1}-\tau  \right] = (1-\tau)\left(\frac{pN}{q+1}-\tau\right)\]
whose right-hand side is positive thanks to $\tau \in (0,1)$ and $\frac{pN}{q+1}>\tau$. Therefore, \eqref{13-23-4} is valid.
\end{proof}

\section{Green's function}\label{sec:Green}

Throughout this appendix, we assume that $N \ge 7$ and $(p,q)$ satisfies $p \in (1,\frac{N}{N-2})$ and \eqref{c-hyperbola}.

\medskip
Given a function $w$, the rotation operator $\Phi_j$
 and the reflection operator $\Psi_h$ defined in \eqref{Aj} and \eqref{Bi}
 respectively, let
\[
\hat w(y) := \frac1 k\sum_{j=1}^k w(\Phi_j y)
\]
and
\begin{equation}\label{wstar}
w^*(y) := \frac1{2(N-1)}\sum_{h=2}^N \big(\hat w(y)+ \hat w(\Psi_h y)\big)
\end{equation}
for $y \in \R^N$. We call $w^*$ the symmetrization of $w$. If $(u,v)\in
C(\Omega)\times C(\Omega)$,  then $(u^*,v^*) \in {L}_s$,
where  ${L}_s$ is the function space in \eqref{Ls}.

Next we introduce an operator
\begin{multline*}
L(u,v) := \left(-\Delta u - p v_\epsilon^{p-1} v-\ep\alpha u -\ep\beta_1 v, -\Delta v  -q U_{0,\mu_0}^{q-1} u-\ep\alpha v -\ep\beta_2 u\right),\\
(u,v) \in {L}_s \cap {X}_{p,q}
\end{multline*}
and its formal dual operator
\begin{multline}\label{1-25-10n2}
L^*(u, v) := \left(-\Delta u - q u_\epsilon^{q-1} v-\ep\alpha u-\ep\beta_2 v, -\Delta v - p v_\epsilon^{p-1} u-\ep\alpha v-\ep\beta_1 u\right),\\
(u,v) \in {L}_s \cap {X}_{q,p}.
\end{multline}
Note that in our settings, the operator $L$ is non-degenerate, so does the operator $L^*$, which means
\begin{align}\label{kernel}
    \text{the kernel of }L^*=\{0\}.
\end{align}

In this appendix, we will investigate the Green's function of $L^*$ tailored to our setting. Define
$$\delta_x^* := \frac1{2(N-1)k} \sum_{h=2}^N \sum_{j=1}^k \left(\delta_{\Phi_j x} + \delta_{\Psi_h \Phi_j x}\right). $$

We consider the following problem:
\begin{equation}\label{2-25-10n}
\begin{cases}
    L^* (u, v) =  (\delta_x^*, 0)&\quad \text{in } \Omega\\
    (u, v) =  0 &\quad \text{on } \partial\Omega
\end{cases}
\end{equation}
and
\begin{equation}\label{nn2-25-10n}
\begin{cases}
    L^* (u, v) =  (0, \delta_x^*)&\quad \text{in } \Omega\\
    (u, v) =  0 &\quad \text{on } \partial\Omega
\end{cases}
\end{equation}
for each $x \in \Omega$.

\begin{proposition}\label{p1-25-10n}
Given a smooth bounded domain, there exists a function $(G_{1,k}, G_{2,k})$ such that $(G_{1,k}(\cdot, x), G_{2,k}(\cdot, x)) \in {L}_s$ solves \eqref{2-25-10n} for each $x \in \Omega$.

Moreover, there exists a constant $C>0$ depending only on $N$, $p$ and $(u_\epsilon,v_\epsilon))$ such that
 for all $x,\ y \in \Omega$, it holds
\begin{equation}\label{G1k}
|G_{1,k}(y, x)| \le \frac{C}{k} \sum_{h=2}^N\sum_{j=1}^k \left(\frac1{|y-\Phi_j x|^{N-2}} + \frac1{|y-\Psi_h\Phi_j x|^{N-2}}\right)
\end{equation}
and
\begin{equation}\label{G2k1}
|G_{2,k}(y, x)| \le \displaystyle \frac{C}{k} \sum_{h=2}^N\sum_{j=1}^k \left(\frac{1}{|y-\Phi_j x|^{N-4}} + \frac{1}{|y-\Psi_h\Phi_j x|^{N-4}}\right).
\end{equation}
\end{proposition}
\begin{proof}
We set $(u_1,v_1)=(\frac{\gamma_N}{|\,\cdot -x\,|^{N-2}},0)$. Clearly,
\begin{equation*}
\begin{aligned}
\begin{cases}
-\Delta u_1-q u_\epsilon^{q-1}v_1=\delta_x & \text{in }\Omega,\\
-\Delta v_1-(p v_\epsilon^{p-1}+\epsilon\beta_2)u_1=-(p v_\epsilon^{p-1}+\epsilon\beta_2)\dfrac{\gamma_N}{|\cdot-x|^{N-2}} & \text{in }\Omega.
\end{cases}
\end{aligned}
\end{equation*}
Let $\bar{v}_2$ be the solution of
\begin{equation*}
\begin{aligned}
\begin{cases}
-\Delta v=-(p v_\epsilon^{p-1}+\epsilon\beta_2)\dfrac{\gamma_N}{|\cdot-x|^{N-2}} & \text{in }\Omega,\\
v=0 & \text{on }\partial\Omega.
\end{cases}
\end{aligned}
\end{equation*}
Then,
\begin{equation*}
\begin{aligned}
|\bar{v}_2(y)|
\le C\int_{\Omega}\frac{1}{|y-z|^{N-2}}\frac{1}{|z-x|^{N-2}}\,dz
\le \frac{C}{|y-x|^{N-4}}
\end{aligned}
\end{equation*}
for $y\in \Omega$. Moreover, $(u_2,v_2):=(u_1,v_1-\bar{v}_2)$
satisfies
\begin{equation*}
\begin{aligned}
\begin{cases}
-\Delta u_2-q u_\epsilon^{q-1}v_2=\delta_x+q u_\epsilon^{q-1}\bar{v}_2 & \text{in }\Omega,\\
-\Delta v_2-p v_\epsilon^{p-1}u_2=0 & \text{in }\Omega.
\end{cases}
\end{aligned}
\end{equation*}
Let $\bar{u}_3$ be the solution of
\begin{equation*}
\begin{aligned}
\begin{cases}
-\Delta u=q u_\epsilon^{q-1}\bar{v}_2 & \text{in }\Omega,\\
u=0 & \text{on }\partial\Omega.
\end{cases}
\end{aligned}
\end{equation*}
Then,
\begin{equation*}
\begin{aligned}
|\bar{u}_3(y)|
\le \dfrac{C}{|y-x|^{N-6}}
\end{aligned}
\end{equation*}
for $y\in \Omega$, and $(u_3,v_3):=(u_2-\bar{u}_3,v_2)$ satisfies
\begin{equation*}
\begin{aligned}
\begin{cases}
-\Delta u_3-q u_\epsilon^{q-1}v_3=\delta_x & \text{in }\Omega,\\
-\Delta v_3-(p v_\epsilon^{p-1}+\epsilon\beta_2)u_3=(p v_\epsilon^{p-1}+\epsilon\beta_2)\bar{u}_3 & \text{in }\Omega.
\end{cases}
\end{aligned}
\end{equation*}
We can continue this process to build $(\bar{u}_l,\bar{v}_l)$ for an integer $l\ge 4$, and find
\begin{equation*}
\begin{aligned}
|\bar{u}_l(y)|+|\bar{v}_l(y)|
\le
\begin{cases}
\dfrac{C}{|y-x|^{N-(2l-2)}} & \text{if } N\ge 2l-1,\\[1ex]
\log\dfrac{C}{|y-x|} & \text{if } N=2l-2,\\[1ex]
C & \text{if } 5\le N\le 2l-3
\end{cases}
\end{aligned}
\end{equation*}
for $y\in \Omega$ and $l\ge 2$.

We select an integer $l_0\ge \dfrac{N+3}{2}$ so that $|\bar{u}_l(y)|+|\bar{v}_l(y)|\le C$
for $y\in \Omega$. Then the associated pair $(u_{l_0},v_{l_0})$ satisfies
\begin{equation*}
\begin{aligned}
\begin{cases}
\displaystyle -\Delta u_{l_0}-q u_\epsilon^{q-1}v_{l_0}=\delta_x+\bar{f}_1 & \text{in }\Omega,\\
\displaystyle -\Delta v_{l_0}-(p v_\epsilon^{p-1}+\epsilon\beta_2)u_{l_0}=\bar{f}_2 & \text{in }\Omega
\end{cases}
\end{aligned}
\end{equation*}
for some $(\bar{f}_1,\bar{f}_2)$ such that $\|\bar{f}_1\|_{L^\infty(\Omega)}+\|\bar{f}_2\|_{L^\infty(\Omega)}\le C.$

Then we consider
\begin{equation}\label{eq:green_1}
    \begin{aligned}
        \begin{cases}
L^*(w_1,w_2)=(\bar{f}_1,\bar{f}_2) & \text{in }\Omega,\\
(w_1,w_2)=(0,0) & \text{on }\partial\Omega.
\end{cases}
    \end{aligned}
\end{equation}
By \eqref{kernel}, \eqref{eq:green_1} has a unique solution $(w_1,w_2)\in {L}_s \cap {X}_{q,p}$. In addition, the standard elliptic $L^p$ estimate yields
\begin{align}\label{est:w12}
    \|w_1\|_{L^\infty(\Omega)}+\|w_2\|_{L^\infty(\Omega)}\le C\|\bar{f}_1\|_{L^\infty(\Omega)}+\|\bar{f}_2\|_{L^\infty(\Omega)}.
\end{align}

Thus, let $(u^*_{l_0},v^*_{l_0})$ be the symmetrization of $(u_{l_0},v_{l_0})$ defined by \eqref{wstar}, the pair $(G_{1,k}(\cdot, x), G_{2,k}(\cdot, x)):=(u^*_{l_0}-w_1,v^*_{l_0}-w_2)\in {L}_s $ solves \eqref{2-25-10n}.

Furthermore, together with \eqref{est:w12}, we can obtain $u^*_{l_0}-w_1$ and $v^*_{l_0}-w_2$ satisfy \eqref{G1k} and \eqref{G2k1}.

\end{proof}

Similar to Proposition \ref{p1-25-10n}, we have the following results for the solution of \eqref{nn2-25-10n}.
\begin{proposition}
Given a smooth bounded domain, there exists a function\\ $(G_{3,k}(\cdot, x), G_{4,k}(\cdot, x))$ such that $(G_{3,k}(\cdot, x), G_{4,k}(\cdot, x)) \in {L}_s$ solves \eqref{nn2-25-10n} for each $x \in \Omega$.

Moreover, there exists a constant $C>0$ depending only on $N$, $p$ and $(u_\epsilon,v_\epsilon))$ such that
 for all $x,\ y \in \Omega$, it holds
\begin{equation}\label{G3k1}
|G_{3,k}(y, x)| \le \frac{C}{k} \sum_{h=2}^N\sum_{j=1}^k \left(\frac1{|y-\Phi_j x|^{N-4}} + \frac1{|y-\Psi_h\Phi_j x|^{N-4}}\right)
\end{equation}
and
\begin{equation}\label{G4k1}
|G_{4,k}(y, x)| \le \displaystyle \frac{C}{k} \sum_{h=2}^N\sum_{j=1}^k \left(\frac{1}{|y-\Phi_j x|^{N-2}} + \frac{1}{|y-\Psi_h\Phi_j x|^{N-2}}\right).
\end{equation}
\end{proposition}

Thus, we can derive a representation formula for any $(u,v) \in {E}$ satisfying $L(u, v)= (f,g)$.
\begin{proposition}
If $(u,v) \in {E}$ satisfies $L(u, v)= (f,g)$, it holds that
\begin{equation}\label{10-27-1}
u(x)=\int_{\Omega} G_{1,k}(y, x) f(y)\,dy+  \int_{\Omega} G_{2,k}(y, x) g(y)\,dy
\end{equation}
and
\begin{equation}\label{11-27-1}
v(x)=\int_{\Omega} G_{3,k}(y, x) f(y)\,dy+  \int_{\Omega} G_{4,k}(y, x) g(y)\,dy
\end{equation}
for all $x \in \Omega$. Here, the functions $G_{1,k}$, $G_{2,k}$, $G_{3,k}$   and $ G_{4,k}$ satisfy \eqref{G1k}, \eqref{G2k1}, \eqref{G3k1} and \eqref{G4k1}.
\end{proposition}

\section{Some technical estimates}
In this section, we present some technical estimates used in the previous sections. The proof of Lemma \ref{A1}--\ref{B3} could be found in \cite{WY}.
\begin{lemma}\label{A1}
    It holds that
    \begin{equation*}
    \sum\limits_{j=2}^{k} \dfrac{1}{|x_j-x_1|^{\alpha}} =
        \begin{cases}
            O(k^{\alpha}/r^{\alpha}), \;\; \alpha>1; \\
            O(k^{\alpha} \log m/r^{\alpha}), \;\; \alpha=1; \\
            O(k/r^{\alpha}), \;\; 0<\alpha<1; \\
        \end{cases}
    \end{equation*}
\end{lemma}

\begin{lemma}\label{B1}
   For any $\alpha > 0$, we have
   \[
      \sum\limits_{j=1}^{k} \dfrac{1}{(1+|y-x_j|)^{\alpha}} \leq C\left( 1 + \sum\limits_{j=2}^{k}\dfrac{1}{|x_1 - x_j|^{\alpha}} \right).
   \]
   Here, the constant $C>0$ does not depend on $k$.
\end{lemma}

\begin{lemma}\label{B2}
 Suppose $\alpha > 1$ and $\beta > 1$ and $i \neq j$. Then, for any $\sigma \in [0, min (\alpha, \beta)]$, we have
   \[
      \dfrac{1}{(1+|y-x_i|)^{\alpha}}\dfrac{1}{(1+|y-x_j|)^{\beta}} \leq \dfrac{C}{|x_i - x_j|^{\sigma}} \left( \dfrac{1}{(1+|y-x_i|)^{\alpha+\beta-\sigma}} + \dfrac{1}{(1+|y-x_j|)^{\alpha+\beta-\sigma}} \right),
   \]
   where $C$ is a positive constant.
\end{lemma}

\begin{lemma}\label{B3}
   If $\sigma \in (0,N-2)$, we have
   \[
      \int_{\R^N} \dfrac{1}{|y-z|^{N-2}} \dfrac{1}{(1+|z|)^{2+\sigma}} dz \leq \dfrac{C}{(1+|y|)^{\sigma}}.
   \]
   If $\sigma > N-2$, we have
   \[
      \int_{\R^N} \dfrac{1}{|y-z|^{N-2}} \dfrac{1}{(1+|z|)^{2+\sigma}} dz \leq \dfrac{C}{(1+|y|)^{N-2}}.
   \]
\end{lemma}
\begin{lemma}\label{B1-1}
Suppose that $p > 1$ and \eqref{c-hyperbola} holds. Let $\tau > 0$ be a number such that $k \simeq \mu^{\tau}$. If $\tau' \ge \tau$, then we have that for any $y \in \R^N$,
\begin{equation}\label{0-6-2}
\left[\sum_{j=1}^k \frac{1}{(1+\mu|y-x_j|)^{\frac{N}{p+1}+\tau'}}\right]^{p}
\le C\sum_{j=1}^k \frac{ 1 }{ (1+\mu |y-x_j| )^{  \frac{N}{q+1}+2+\tau' } }
\end{equation}
and
\[\left[\sum_{j=1}^k \frac{1}{(1+\mu|y-x_j|)^{\frac{N}{q+1}+\tau'}}\right]^{q}
\le C\sum_{j=1}^k \frac{ 1 }{ (1+\mu |y-x_j| )^{  \frac{N}{p+1}+2+\tau' } }.\]
\end{lemma}
\begin{proof}
We just prove \eqref{0-6-2}. By H\"older's inequality and \eqref{c-hyperbola},
\begin{align*}
\left[\sum_{j=1}^k\frac{1}{(1+\mu|y-x_j|)^{\frac{N}{p+1}+\tau'}}\right]^{p}
&= \left[\sum_{j=1}^k \frac{1}{(1+\mu |y-x_j|)^{(\frac{N}{q+1}+2+\tau')\frac{1}{p}+\frac{N}{p+1}+\tau' - (\frac{N}{q+1}+2+\tau') \frac{1}{p}}}  \right]^{p} \\
&\le C \sum_{j=1}^k \frac{1}{(1+\mu |y-x_j|)^{\frac{N}{q+1}+2+\tau'}}
\left[\sum_{j=1}^k \frac{1}{(1+\mu |y-x_j|)^{\tau'}}\right]^{p-1}  \\
&\le C \sum_{j=1}^k \frac{1}{(1+\mu |y-x_j|)^{\frac{N}{q+1}+2+\tau'}},
\end{align*}
where we used $\tau' \ge \tau$ to get $\sum_{j=1}^k (1+\mu |y-x_j|)^{-\tau'} \le C$.
\end{proof}

\begin{lemma}\label{lemma:D6}
    Assume $N\geq 7$, $p\in (1,\frac{N}{N-2})$, $(p,q)$ satisfies \eqref{c-hyperbola}, it holds that
    \begin{align}\label{estimate_11}
        \left| \left\langle -\Delta Z_{1,h}  - q (PU_*)^{q-1}Y_{1,h}, \phi \right\rangle\right|=O\left(\frac{\mu^{n_h}\|(\phi,\psi)\|_{*}}{\mu^{\sigma}} \right),
    \end{align}
and
    \begin{align}\label{estimate_22}
    \left|\left\langle -\Delta Y_{1,h}  - p (PV_*)^{p-1}Z_{1,h}, \psi \right\rangle \right|=O\left( \frac{\mu^{n_h}\|(\phi,\psi)\|_{*}}{\mu^{\sigma}} \right).
    \end{align}
where $(\phi,\psi)$ is the same as \eqref{estimate_1}.
\end{lemma}
\begin{proof}
We only estimate \eqref{estimate_11}, \eqref{estimate_22} can be treated similarly.

We derive
    \begin{equation}\label{eq:d6_1}
\begin{aligned}
&\quad
\left|
\int_\Omega
\left(-\Delta Z_{1,h}-q(PU_*)^{q-1}Y_{1,h}\right)\phi
\right|\\
&\leq
\left|
\int_{\Omega}
q\left(U_1^{q-1}-PU^{q-1}\right)Y_{1,h}\phi
\right|
+
\left|
\int_{\Omega}
q\left(|PU_*|^{q-1}-PU^{q-1}\right)Y_{1,h}\phi
\right|.
\end{aligned}
\end{equation}

On the one hand, we get
\begin{equation*}
    \begin{aligned}
        \int_{S}(PU^{q-1}-U_1^{q-1})Y_{1,h}\phi
        \leq \int_S U_1^{q-2}(u_\ep+\sum_{j=2}^k U_j)Y_{1,h}\phi
        =O\big(\mu^{n_h-\sigma}\|(\phi,\psi)\|_{*} \big).
    \end{aligned}
\end{equation*}

On the other hand, it follows from Lemma~\ref{lemma:U} that
\begin{equation}
    \begin{aligned}
        &\quad\int_{\Omega\setminus\Omega_1}(PU^{q-1}-U_1^{q-1})Y_{1,h}\phi\\
        &\leq \int_{\Omega\setminus\Omega_1}\Big(\sum_{j=1}^k \frac{\mu^{\frac{N}{q+1}}}{(1+\mu|y-x_j|)^{p(N-2)-2}}+\sum_{j=1}^k \frac{\mu^{-\frac{pN}{q+1}}k^{p(N-2)-2}}{(1+k|y-x_j|)^{p(N-3-\theta)-2}}\Big)^{q-1}Y_{1,h}\phi.
    \end{aligned}
\end{equation}

In $\Omega_i$, applying \eqref{tip_1}, we have
\begin{align*}
            \sum_{j\neq i} \frac{\mu^{\frac{N}{q+1}}}{(1+\mu|y-x_j|)^{\frac{N}{q+1}+\tau}}
            \leq \frac{1}{|y-x_i|^{\frac{N}{q+1}}},
\end{align*}
and
\begin{align*}
    \sum_{j\neq i} \frac{\mu^{\frac{N}{q+1}}}{(1+\mu|y-x_j|)^{p(N-2)-2}}\leq \frac{\mu^{-\frac{pN}{q+1}}k^{1+\theta}}{|y-x_i|^{p(N-2)-3-\theta}}.
\end{align*}

Then, we obtain
    \begin{equation*}
        \begin{aligned}
            &\quad\int_{\Omega_1\setminus S}
            \frac{\mu^{\frac{N}{q+1}}}{(1+\mu|y-x_1|)^{p(N-2)-2}}
            \\
            &\hspace{5em}\times\Big( \sum_{j=1}^k \frac{\mu^{\frac{N}{q+1}}}{(1+\mu|y-x_j|)^{p(N-2)-2}}\Big)^{q-1}\sum_{j=1}^k \frac{\mu^{\frac{N}{q+1}}}{(1+\mu|y-x_j|)^{\frac{N}{q+1}+\tau}}\dy\\
            &\leq \mu^{-\frac{pqN}{q+1}}
            \int_{\Omega_1\setminus S}\Big(\frac{1}{|y-x_1|^{(p(N-2)-2)(q-1)}}+\frac{k^{(1+\theta)(q-1)}}{|y-x_1|^{(p(N-2)-3-\theta)(q-1)}} \Big) \frac{1}{|y-x_1|^{p(N-2)-2+\frac{N}{q+1}}}\\
            &\leq k^{q(p(N-2)-2)+\frac{N}{q+1}-N}\mu^{-\frac{pqN}{q+1}}=O(\mu^{-\sigma}),
        \end{aligned}
    \end{equation*}

    and

    \begin{equation*}
        \begin{aligned}
            &\quad\sum_{i=2}^k\int_{\Omega_i} 
            \frac{\mu^{\frac{N}{q+1}}}{(1+\mu|y-x_1|)^{p(N-2)-2}}
            \\
            &\hspace{5em}\times\Big( \sum_{j=1}^k \frac{\mu^{\frac{N}{q+1}}}{(1+\mu|y-x_j|)^{p(N-2)-2}}\Big)^{q-1}\sum_{j=1}^k \frac{\mu^{\frac{N}{q+1}}}{(1+\mu|y-x_j|)^{\frac{N}{q+1}+\tau}}\dy\\
            &\leq \mu^{-\frac{pqN}{q+1}}\sum_{i=2}^k
            \int_{\Omega_i}
            \Big(\frac{1}{|y-x_i|^{p(N-2)-2}}+\frac{k^{1+\theta}}{|y-x_i|^{p(N-2)-3-\theta}}\Big)^{q-1}\frac{1}{|y-x_1|^{p(N-2)-2}}\frac{1}{|y-x_i|^{\frac{N}{q+1}}}\dy\\
            &\leq \mu^{-\frac{pqN}{q+1}}\Big(\sum_{i=2}^k
            \frac{1}{|x_1-x_i|^{\frac{pqN}{q+1}}}+\sum_{i=2}^k
            \frac{k^{(q-1)(1+\theta)}}{|x_1-x_i|^{\frac{pqN}{q+1}-(q-1)(1+\theta)}}\Big)=O(\mu^{-\sigma}).
        \end{aligned}
    \end{equation*}
    where we have used $q(p(N-2)-2)+\frac{N}{q+1}>q(p(N-2)-2)-(q-1)(1+\theta)+\frac{N}{q+1}>N$ holds for $N\geq7,\,p\in (1,\frac{N}{N-2})$.

Similarly, we have
\begin{equation*}
    \begin{aligned}
        &\quad\int_{\Omega\setminus \Omega_1}
            \frac{\mu^{\frac{N}{q+1}}}{(1+\mu|y-x_1|)^{p(N-2)-2}}
            \\
            &\hspace{5em}\times\Big( \sum_{j=1}^k \frac{\mu^{-\frac{pN}{q+1}}k^{p(N-2)-2}}{(1+k|y-x_j|)^{p(N-3-\theta)-2}}\Big)^{q-1}\sum_{j=1}^k \frac{\mu^{\frac{N}{q+1}}}{(1+\mu|y-x_j|)^{\frac{N}{q+1}+\tau}}\\
            &=O(\mu^{-\sigma}).
    \end{aligned}
\end{equation*}

Combining the above estimates, we have
\begin{equation}
    \begin{aligned}
        \int_{\Omega}(PU^{q-1}-U_1^{q-1})Y_{1,h}\phi=O\big(\mu^{n_h-\sigma}\|(\phi,\psi)\|_{*} \big),
    \end{aligned}
\end{equation}

A similar computation yields
$$\left|
\int_{\Omega}
q\left(|PU_*|^{q-1}-PU^{q-1}\right)Y_{1,h}u
\right|=O\big(\mu^{n_h-\sigma}\|(\phi,\psi)\|_{*} \big).$$
Hence, we obtain \eqref{estimate_11}.

\end{proof}

\begin{lemma}\label{lemma:H}
    Under the same assumption in Theorem A, there exists a positive radial solution $(u_\epsilon,v_\epsilon)$ to \eqref{mainsystem} which is non-degenerate.
\end{lemma}
 \begin{proof}

    Let $G$ be the Green's function of the Laplacian $-\Delta$ in $\Omega$ with
Dirichlet boundary condition. Let $H$ be its regular part, then
$$G(x,y)=S(x,y)-H(x,y)$$ with
$S(x,y)=\dfrac{\gamma_N}{|x-y|^{N-2}}$, where
$\gamma_N=\dfrac{1}{(N-2)|\mathbb{S}^{N-1}|}$. 

Then set
$$\widetilde{G}(x,y)=\int_{\Omega}G(x,z)G^p(z,y)\,\dz,$$
$$\widetilde{H}(x,y)=\int_{\Omega}G(x,z)G^p(y,z)\,\dz
-\int_{\mathbb{R}^N}S(x,z)S^p(z,y)\,\dz$$ and
$$\Tau(x)=\widetilde{H}(x,x).$$

 Consider the configuration space
 $$\Lambda=\{(d,\xi):d\in(\delta_1,\delta_1^{-1}),\xi\in\Omega, \text{dist}(\xi,\partial\Omega)>\delta_2\},$$
 for suitable $\delta_1\in(0,1)$ and $\delta_2>0$.

 Let $\bar{\mu}>0$ be small. Define $PV_{\xi,d}=PV_{\xi,\bar{\mu}^{-1}d}$ and $PU_{\xi,d}$ as the unique solution of
 \begin{equation}
\begin{cases}
    -\Delta PU_{d,\xi} =PV_{d,\xi}^p, & \text{in } \Omega,\\
    PU_{d,\xi}=0, & \text{on } \partial \Omega.
\end{cases}
\end{equation}
The authors in \cite{KP21} constructed $(u_\ep,v_\ep)$ with the form
$$(PU_{d,\xi}+\Psi_{d,\xi},PV_{d,\xi}+\Phi_{d,\xi})$$
where $(\Psi_{d,\xi},\Phi_{d,\xi})$ is uniquely determined by a reduction procedure (see Section 4 in \cite{KP21}), $\bar{\mu}\to 0$, \(d_{\epsilon} \to d_{\beta_1, \beta_2, \alpha} > 0\) (\(d_{\beta_1, \beta_2, \alpha, N}\) represents a constant related to \(\beta_1, \beta_2, \alpha, N\)), \(\xi_\ep \to \xi_0(\xi_0 \in \Omega \text{ and } \nabla \tau(\xi_0) = 0)\) as $\ep\to 0$.

They prove $(u_\ep,v_\ep)$ is a solution to \eqref{mainsystem} by showing that $(d,\xi)\in \text{int}(\Lambda)$ is a critical point of the reduced energy
$$J_\ep(d,\xi)=I_\ep(PU_{d,\xi}+\Psi_{d,\xi},PV_{d,\xi}+\Phi_{d,\xi}).$$

 When $\Omega$ is a unit ball, $J_\ep$ is invariant under the action of $\mathcal{O}(N)$, that is
 $$J_\ep(d,Q\xi)=J_\ep(d,\xi),\qquad\text{for any }Q\in \mathcal{O}(N).$$
Thus, we have
$$\nabla_{\xi}J_\ep (d,0)=0$$
which gives single bubble solutions to \eqref{mainsystem} centered at $\xi_\ep=0$, and hence positive radial solutions to \eqref{mainsystem}.

Next, we prove the non-degeneracy of such radial solutions.

We have
     \begin{align*}
         \Tau(x)&=\int_{\R^N}S^{p+1}(z,x)\dz-\int_{\Omega}G^{p+1}(z,x)\dz\\
         &=\int_{\Omega}\big[S^{p+1}(z,x)-G^{p+1}(z,x)\big]\dz
         +\int_{\R^N\setminus\Omega}S^{p+1}(z,x)\dz\\
         :&=\Tau_1(x)+\Tau_2(x).
     \end{align*}

Direct computation gives
\begin{equation*}
    \begin{aligned}
        \Tau_1(x)=\lim_{\epsilon\to 0}\int_{\Omega\setminus B_{\epsilon}(0)}S^{p+1}(z,x)-G^{p+1}(z,x)\dz
        =\tau_1(0)+c_{12} |x|^2 +O(|x|^4),
    \end{aligned}
\end{equation*}
where
$$\begin{aligned}
    c_{12}&=\frac{p(p+1)(N-2)^2}{2N}\gamma_N^{p+1}\times\\
    &\quad\lim_{\epsilon\to 0}\int_{\Omega\setminus B_{\epsilon}(0)}
\Bigl(
|z|^{-(N-2)(p-1)-2N}
-\bigl(|z|^{2-N}-1\bigr)^{p-1}\bigl(|z|^{-N}-1\bigr)^2
\Bigr)
|z|^2\,dz>0.
\end{aligned}$$

 On the other hand, denote $\alpha_0=(p+1)(N-2)$, by symmetry, we have
\begin{equation*}
    \begin{aligned}
        \Tau_2(x)&=\Tau_2(0)\\
&\hspace{1em}+\gamma_N^{p+1}\int_{\R^N\setminus\Omega}\Big(\frac{\alpha_0(\alpha_0+2)}{2}|z|^{-\alpha_0-4}(z\cdot x)^2
-\frac{\alpha_0}{2}|z|^{-\alpha_0-2}|x|^2\Big)
\dz+O(|x|^4)\\
&=\Tau_2(0)+
c_{22}|x|^2+O(|x|^4)
    \end{aligned}
\end{equation*}
where $c_{22}=\gamma_N^{p+1}\frac{\alpha_0}{2}
(\frac{\alpha_0+2}{N}-1)
\int_{\mathbb R^N\setminus \Omega}|z|^{-\alpha_0-2}\,\dz>0$.

Thus, we obtain
$$\Tau(x)=\Tau(0)+(c_{12}+c_{22})|x|^2+O(|x|^4),$$
which implies $0$ is a non-degenerate critical point of $\Tau$. We complete this proof.
 \end{proof}

\begin{lemma}\label{lemma:alg_1}
    Assume $N\geq 8$, $p\in (1,\frac{N}{N-2})$ or $N=7$, $p\in (p_*,\frac{7}{5})$ where $p_*$ is the unique root of $36p^3-129p^2+117p-22=0$ in the interval $(1,\frac{7}{5})$ we have
    $$\mathcal{B}_2<\mathcal{B}_1$$
    where $\mathcal{B}_1$ and $\mathcal{B}_2$ are defined in \eqref{def:B1} and \eqref{def:B2}.
\end{lemma}

\begin{proof}
A direct computation yields
\[
\beta_1-\beta_2
=
-\frac{F_N(p)}{2(N-3p-2)K(p)},
\]
where
\[
K(p):=(N-3)+(N+1)p-(N-3)p^2.
\]
and
\[
F_N(p)
:=
(N^2-N-6)p^3-(3N^2-2N-4)p^2+2N^2p+(N+12)p-4N+6.
\]
It is easy to check if \(1<p<\frac{N}{N-2}\) and \(N\ge 8\), we have
\[
N-3p-2>0.
\]
and
\[
\min K(p)=\min\{K(1),K(\frac{N}{N-2})\}>0.
\]
Therefore, it suffices to prove that
\[
F_N(p)<0, 
\text{for all }p\in\left(1,\frac{N}{N-2}\right).
\]

We compute
\[
F_N'(p)
=
3(N^2-N-6)p^2-(6N^2-4N-8)p+(2N^2+N+12),
\]
and
\[
\max F_N'(p)=\max\{F_N'(1),F_N'(\frac{N}{N-2})\}<0.
\]
Hence \(F_N\) is strictly decreasing on this interval. We conclude that
\[
F_N(p)<F_N(1)\le 0,
\]
namely,
\[
F_N(p)<0,
\text{for all }p\in\left(1,\frac{N}{N-2}\right).
\]

As a consequence, we obtain
\[
\mathcal{B}_2<\mathcal{B}_1.
\]
\end{proof}

\section*{Statements and Declarations}

The authors confirm that there are no relevant financial or non-financial competing interests to report.
	
\section*{Data Availability Statements}

All data generated or analyzed during this study are included in this article.

\end{document}